\documentclass[11pt]{amsart}

\usepackage{amssymb}
\usepackage{epsfig}
\usepackage{a4wide}

\usepackage{amsmath}
\usepackage{graphicx}
\usepackage{amsthm}
\usepackage{enumerate}
\usepackage{mathrsfs}
\usepackage{bbold}
\usepackage{soul}
\usepackage{color}

\newtheorem{theorem}{Theorem}
\newtheorem{proposition}[theorem]{Proposition}
\newtheorem{lemma}[theorem]{Lemma}
\newtheorem{corollary}[theorem]{Corollary}

\newtheorem*{conjecture}{Conjecture}

\newtheorem{maintheorem}{Theorem}

\newcommand{\R}{\mathbb{R}}

\newcommand{\Z}{\mathbb{Z}}
\newcommand{\N}{\mathbb{N}}
\newcommand{\Le}{\mathbb{L}^e}

\newcommand{\cE}{\mathscr{E}}

\newcommand{\cP}{\mathscr{P}}
\newcommand{\cQ}{\mathscr{Q}}

\newcommand{\Xe}{\tilde{X}^e}
\newcommand{\tZ}{\tilde{z}}

\def\Fix#1{{\rm Fix}\,#1}
\def\mod#1{\hskip 10pt \left({\rm mod\ }#1\right)}%   modulus of a congruence
\def\fl#1{\lfloor #1\rfloor}%   floor function

\def\endproof{\hfill\vbox{\hrule
    \hbox{\vrule\kern4pt\vbox{\kern4pt
    \kern4pt}\kern4pt\vrule}\hrule}\bigskip}%   end of proof box

\begin{document}

\title{Near-integrable behaviour\\in a family of discretised rotations}

%\keywords{Round-off errors, arithmetic dynamics \LaTeXe{}}

\author{Heather Reeve-Black}
\address{School of Mathematical Sciences, Queen Mary, 
University of London, London E1 4NS, UK}
\email{h.reeve-black@qmul.ac.uk}

\author{Franco Vivaldi}
\address{School of Mathematical Sciences, Queen Mary,
University of London, London E1 4NS, UK}
\email{f.vivaldi@maths.qmul.ac.uk}
\urladdr{http://www.maths.qmul.ac.uk/\~{}fv}

\begin{abstract}
We consider a one-parameter family of invertible maps of a two-dimensional 
lattice, obtained by discretising the space of planar rotations.
We let the angle of rotation approach $\pi/2$, and show that the limit of 
vanishing discretisation is described by an integrable piecewise-smooth 
Hamiltonian flow, whereby the plane foliates into families of invariant 
polygons with an increasing number of sides.
Considered as perturbations of the flow, the lattice maps assume a 
different character, described in terms of strip maps, 
a variant of those found in outer billiards of polygons.
The perturbation introduces phenomena reminiscent of the 
Kolmogorov-Arnold-Moser scenario: a positive fraction of the unperturbed curves survives. 
We prove this for symmetric orbits, under a condition that allows us to 
obtain explicit values for their density, the latter being a rational 
number typically less than 1. 
This result allows us to conclude that the infimum of the density of all 
surviving curves ---symmetric or not--- is bounded away from zero.
%Under similar assumptions, we also prove that, locally, the symmetric perturbed 
%curves are uniformly distributed in phase space. 

\end{abstract}

\date{\today}

\maketitle

%%%%%%%%%%%%%%%%%%%%%%%%%%%%%%%%%%%%%%%%%%%%%%%%%%%%%%%%%%%%%%%%%%%%%%%%%%
\section{Introduction}
%%%%%%%%%%%%%%%%%%%%%%%%%%%%%%%%%%%%%%%%%%%%%%%%%%%%%%%%%%%%%%%%%%%%%%%%%%

The study of near-integrable Hamiltonian dynamics on a discrete phase space 
presents a unique set of problems. 
On the continuum, the orbits of an integrable symplectic nonlinear map 
are rotations on invariant tori, which, generically, 
are quasi-periodic. According to Kolmogorov-Arnold-Moser (KAM) theory, a positive fraction of these tori, 
identified by their frequency, will survive a sufficiently small smooth perturbation.
The complement of KAM tori consists of a hierarchical arrangement of 
island chains and thin stochastic layers. In low-dimensions, the tori 
disconnect the space and hence ensure stability.

Reproducing these structures in a discrete space is problematic, 
due to the lack of a framework for perturbation theory.
Typically, quasi-periodic orbits do not exist. Invariant sets acting
as surrogate KAM surfaces must thus be identified, and their evolution 
must be tracked as a perturbation parameter is varied.
Even in low dimensions, these invariant sets need not disconnect 
the space, so their relevance to stability must be re-assessed.

There are various approaches to space discretisation. For an algebraic system 
it is natural to replace the real or complex coordinate fields by a finite field. 
Because this procedure erases all topological information, in the discrete phase 
space there is no near-integrable regime at all, and one witnesses a discontinuous 
transition from integrable to non-integrable behaviour. 
This transition manifests itself probabilistically via a (conjectured) abrupt 
change in the asymptotic (large field) distribution of the periods of the orbits 
\cite{JogiaRobertsVivaldi,RobertsVivaldi:09}.

Round-off in computer arithmetic brings about an equally blunt discretisation of space.
Here a small perturbation causes the dynamics to collapse onto a discrete set. 
In rare cases, round-off fluctuations act like small amplitude noise, and give rise 
to Gaussian transport; the latter, however, wipes out all small-scale dynamical features. 
More commonly, the noise model is inappropriate, but there is no general theory to rely on. 
(Shadowing theory is unhelpful here because it requires hyperbolicity \cite[Section 18.1]{KatokHasselblatt}.)
In all, the literature devoted to the study of deterministic (as opposed 
to probabilistic) manifestation of near-integrability in computer arithmetic 
is minimal \cite{EarnTremaine,BlankKrugerPustylnikov,ZhangVivaldi} (see also \cite{Blank:97}).

A different kind of space discretisation occurs in piecewise isometric systems, 
as the combined effect of discontinuity and rationality. In isometries involving 
rational rotations, the space gets tessellated by polygons with algebraic
number coordinates. As these polygons move rigidly, the phase space is discrete.
Outer billiards of polygons are symplectic maps of this type, which feature a skeletal 
version of divided phase space (see \cite{Schwartz}, and references therein).
Under appropriate rationality conditions, these systems support a countable family 
of bounding invariant sets, which prevent orbits from escaping to infinity 
\cite{VivaldiShaidenko,Kolodziej,GutkinSimanyi}. These sets, which resemble more 
a chain of integrable resonant zones than KAM invariants, are uniformly 
distributed on the plane ---this theme will also appear in the present work. 
The existence of unbounded orbits for some irrational parameters has been 
a significant recent advance \cite{Schwartz:07} (see also \cite{DolgopyatFayad}).

In this paper we explore discrete near-integrability in the family of invertible 
lattice maps
\begin{equation} \label{eq:Map}
F:\,\Z^2\rightarrow\Z^2\qquad
(x,y)\,\mapsto\,(\lfloor \lambda x \rfloor - y, \,x)
\qquad \lambda=2\cos(2\pi\nu)
\end{equation}
where $\lfloor \cdot \rfloor$ is the floor function 
---the largest integer not exceeding its argument.
If we remove the floor function in equation (\ref{eq:Map}),
we obtain a one-parameter family of linear maps of the plane, 
$(x,y)\mapsto (\lambda x-y,x)$, which are linearly conjugate 
to a rotation by the angle $2\pi\nu$, where $\nu$ is the rotation number.
The floor function provides the discretisation, rounding the image point 
$(\lambda x-y,x)$ to the nearest lattice point on the left. 
An essential property of this model is its invertibility, 
which a typical round-off scheme applied to a rotation
does not have. Furthermore, rounding by the floor function 
---as opposed to the nearest integer--- is arithmetically 
nicer and has fewer symmetries.
In the lattice map $F$, the discretisation length is fixed, and the 
limit of vanishing discretisation corresponds to motions at infinity.
This asymptotic regime is our main interest.

The deceptively simple model (\ref{eq:Map}) displays a rich landscape 
of mathematical phenomena, connecting discrete dynamics and arithmetic. 
This model originated in dynamical system theory
\cite{Vivaldi:94b,LowensteinHatjispyrosVivaldi,LowensteinVivaldi:98,
LowensteinVivaldi:00,BosioVivaldi,VivaldiVladimirov,KouptsovLowensteinVivaldi},
and was subsequently studied in number theory, within the context of
shift radix systems \cite{AkiyamaBrunottePethoThuswaldner,AkiyamaBrunottePethoSteiner}.
The following unsolved question 
\cite{AkiyamaBrunottePethoThuswaldner,Vivaldi:06}
distills the difficulties encountered in the analysis of this model.\footnote{A related
general conjecture on the boundedness of discretised Hamiltonian rotations was first 
formulated in \cite{Blank}.}
\begin{conjecture} \label{conj:Periodicity}
For all real $\lambda$ with $|\lambda|<2$, all orbits of $F$ are periodic.
\end{conjecture}
Due to invertibility, periodicity is equivalent to boundedness. 
This conjecture holds trivially for $\lambda=0,\pm1$, where the map $F$ is of finite order.
Beyond this, the boundedness of all round-off orbits has been proved for only 
\textit{eight} values of $\lambda$, which correspond to the rational 
values of the rotation number $\nu$ for which $\lambda$ is a quadratic irrational:
\begin{equation}\label{eq:Lambdas}
\lambda=\frac{\pm1\pm\sqrt{5}}{2},\quad \pm\sqrt{2},\quad \pm\sqrt{3}.
\end{equation}
(The denominator of $\nu$ is $5,\,10,\,8$, and 12, respectively.)
In these cases the map $F$ admits a dense and uniform embedding in 
a two-dimensional torus, where the round-off map extends continuously 
to a piecewise isometry, which has {\it zero entropy} (and is not ergodic).
The natural density on the lattice $\Z^2$ is carried into the Lebesgue measure,
namely the Haar measure on the torus.
The case $\lambda=(1-\sqrt{5})/2$ was established in 
\cite{LowensteinHatjispyrosVivaldi}, with computer assistance.
Similar techniques were used to extend the result to the other parameter 
values, but only for a set of initial conditions having full density
\cite{KouptsovLowensteinVivaldi}.
The conjecture for the eight parameters \eqref{eq:Lambdas} was settled 
in \cite{AkiyamaBrunottePethoSteiner} with an analytical proof.
For any other rational value of $\nu$, there is a similar embedding
in a piecewise isometry of a higher-dimensional torus; these systems 
are still unexplored, even in the cubic case.

Irrational values of $\nu$ bring about a different dynamics, and a different 
theory. The simplest cases correspond to rational values of $\lambda$, and, 
in particular, to rational numbers whose denominator is the power of a prime $p$. 
In this case the map $F$ admits a dense and uniform embedding in the ring 
$\Z_p$ of $p$-adic integers \cite{BosioVivaldi}. 
The embedded system extends continuously to the composition of a full 
shift and an isometry (which has {\it positive entropy}), 
and the natural density in $\Z^2$ is now carried into the Haar measure on $\Z_p$.
This construct was later used to prove a central limit theorem for the 
departure of the round-off orbits from the unperturbed ones \cite{VivaldiVladimirov}. 
This phenomenon injects a probabilistic element in the determination of the 
period of the lattice orbits, highlighting the nature of the difficulties 
that surround conjecture \ref{conj:Periodicity}. Very recently,
Akiyama and Peth\H{o} \cite{AkiyamaPetho} proved that (\ref{eq:Map}) has
infinitely many periodic orbits for any parameter value.

In this work we consider a new regime, namely the limit $\lambda\to 0$ 
of equation (\ref{eq:Map}), corresponding to the rotation number $\nu\to 1/4$. 
This is one of five limits (the other limits being $\lambda\to\pm1,\pm2$) 
where the dynamics at the limit is trivial because there is no round-off. 
After scaling, we embed the lattice $\Z^2$ in 
$\R^2$, and show that there is a non-smooth integrable Hamiltonian 
\textit{flow} (not a rotation), which represents the limiting unperturbed 
dynamics. This integrable system is non-linear, namely its time-advance 
map satisfies a twist condition.
Thus the limit $\lambda\to 0$ is singular. The parameter $\lambda$ 
acts as a perturbation parameter, and a discrete version of near-integrable 
symplectic dynamics emerges on the lattice when the perturbation is switched on.

%%%%%%%%%%%%%%%%%%%%%%%%%%%%%%%%%%%%%%%%%%%%%%%%%%%%%%%% FIGURE
\begin{figure}[h]
\centering
\hfil
\epsfig{file=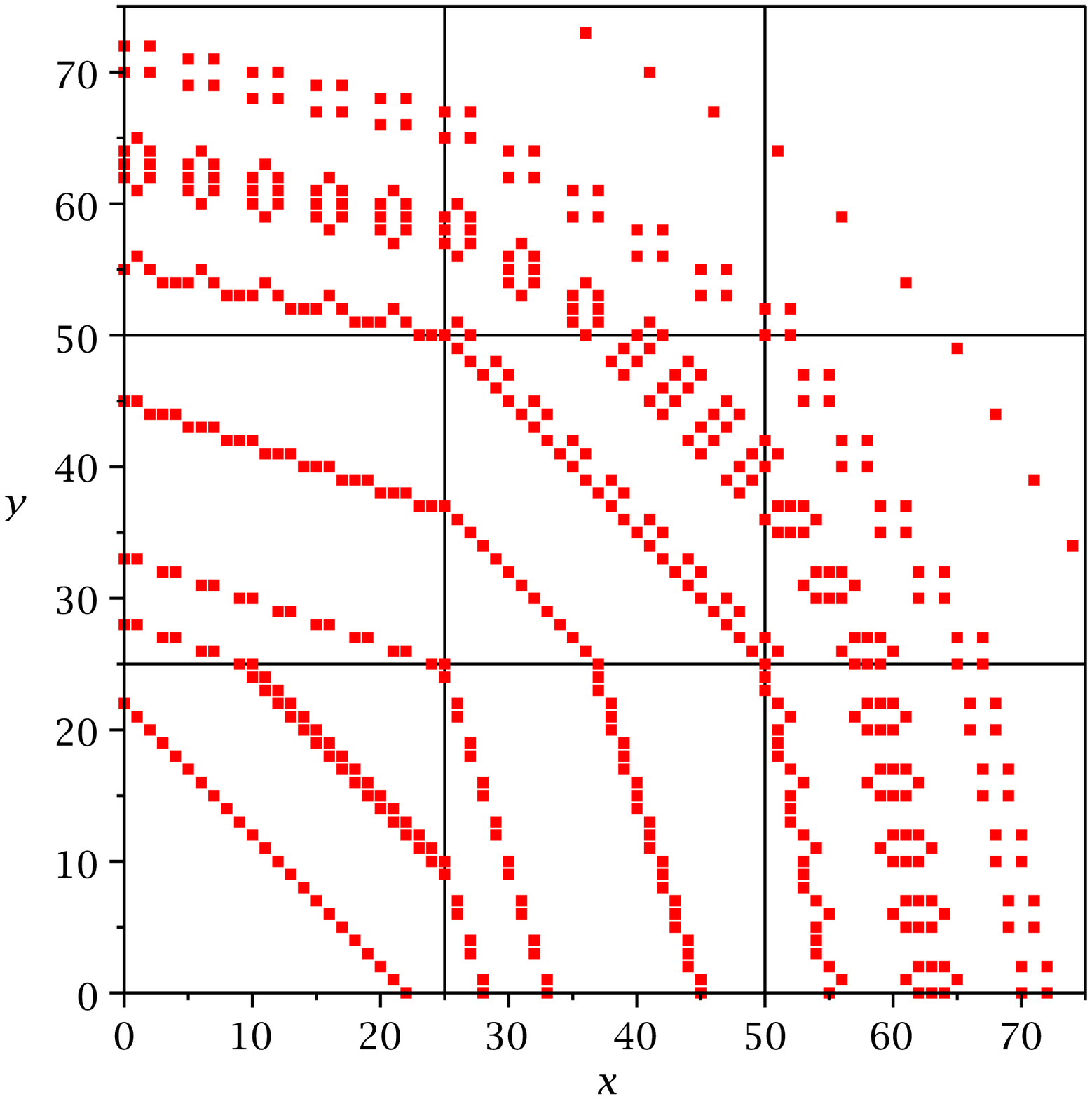,scale=0.35}
\quad
\epsfig{file=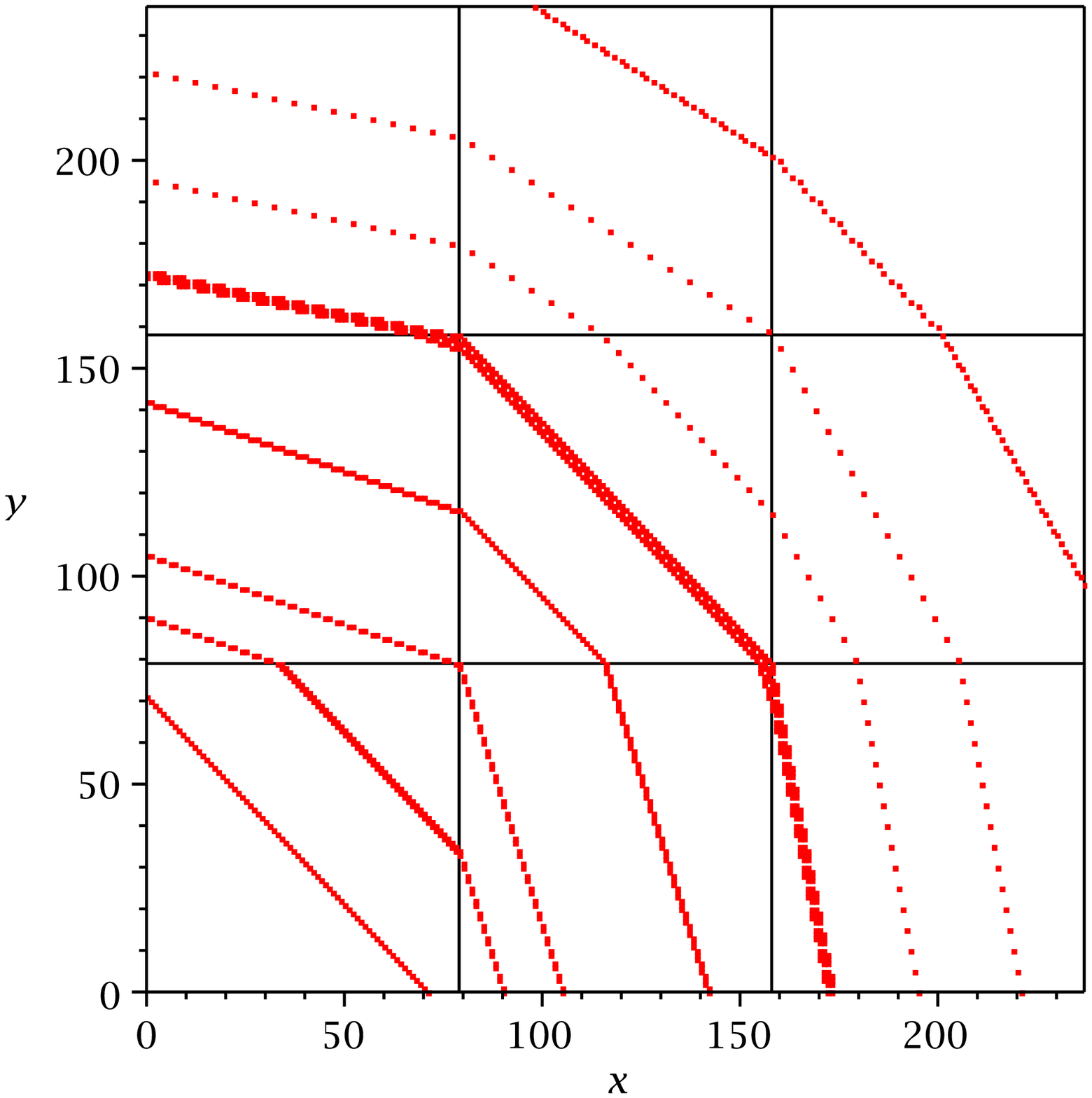,scale=0.35}
\hfil
\caption{\label{fig:PolygonalOrbits}\rm\small
Some periodic orbits of the map $F$, for small values of 
the parameter $\lambda$. 
Left: $\lambda=1/25$; right: $\lambda=1/79$.
The boxes have side $\lambda^{-1}$.
}
\end{figure}
%%%%%%%%%%%%%%%%%%%%%%%%%%%%%%%%%%%%%%%%%%%%%%%%%%%%%%%%

More precisely, if $\lambda$ is small, then the orbits lie approximately on convex 
polygons; the smaller $\lambda$, the closer the approximation (figure \ref{fig:PolygonalOrbits}). The number of sides 
of these polygons increases with the distance from the origin; near the origin they
are squares, while at infinity they approach circles.

In the spirit of Takens' theorem \cite[section 6.2.2]{ArrowsmithPlace}, we 
introduce a piecewise smooth integrable Hamiltonian function 
(equation (\ref{eq:Hamiltonian})), whose invariant curves are polygons,
representing the limit foliation of the plane for the system (\ref{eq:Map}).
To match Hamiltonian flow and lattice map, we exploit the fact that, for small 
$\lambda$, the composite map $F^4$ is close to the identity. After scaling, it is 
possible to identify the action of $F^4$ with the unit time-advance map of the flow.
The two actions agree along the sides of the polygons, but they differ in vanishingly
small regions near the vertices. This discrepancy provides the perturbation mechanism. 

All integrable orbits must close after one revolution around the origin. By contrast, 
the lattice orbits need not do so, leading to a non-trivial clustering
of the periods around integer multiples of a basic rotational period ---see
figure \ref{fig:PeriodFunction}. 
%%%%%%%%%%%%%%%%%%%%%%%%%%%%%%%%%%%%%%%%%%%%%%%%%%%%%%%% FIGURE
\begin{figure}[!h]
        \centering
	\epsfig{file=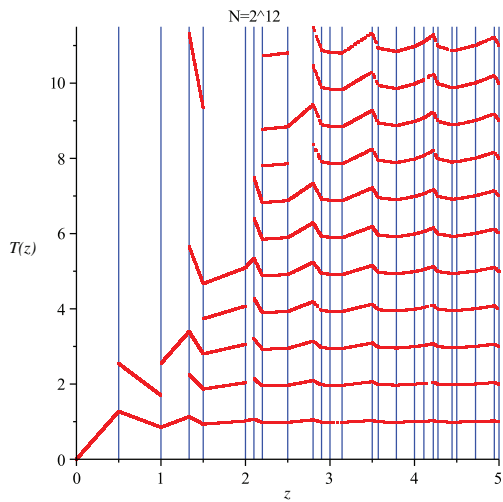,scale=0.35}
        \caption{The normalised period function $T_\lambda(z)$ for points 
         $z=(x,x)$, and $\lambda=2^{-12}$.
         The vertical lines mark the location of critical polygons, which 
         pass through lattice points.
         }
        \label{fig:PeriodFunction}
\end{figure}
%%%%%%%%%%%%%%%%%%%%%%%%%%%%%%%%%%%%%%%%%%%%%%%%%%%%%%%%
The lowest branch of the period function comprises the {\it minimal orbits,} 
which shadow the unperturbed orbits, and close after one revolution. For these 
orbits, the effects of the perturbation cancel out. (The role of cancellation 
of singularities in the existence of invariant tori of non-smooth systems 
was noted long ago \cite{HenonWisdom}.)
These are the orbits of the integrable system that survive the perturbation.
The other orbits mimic the divided phase space structure of a near-integrable 
area-preserving map, in embryonic form near the origin, and increasing in 
complexity for larger amplitudes.

The set of invariant polygons is partitioned by \textit{critical polygons}, 
which contain $\Z^2$ points, into infinitely many \textit{polygon classes}, 
which can be characterised arithmetically in terms of sums of squares.
The main result of this work is that for infinitely many classes, a positive 
fraction of the minimal orbits having time-reversal symmetry survive. 
The restriction to infinitely many classes ---as opposed to all classes--- 
stems from a coprimality condition we impose in order to achieve convergence
of the density. This density turns out to be a rational number smaller than 
unity, which depends only on the family of polygons being considered. 
As the number of sides of the polygons increases to infinity, the density 
tends to zero. As a corollary, we obtain a positive rational lower bound 
for the density of all minimal orbits ---symmetric or not.
These results appear as theorems 
A %\ref{thm:Phi_equivariance} 
and 
B, %\ref{thm:minimal_densities}, 
stated in section \ref{sec:MainTheorems} after a somewhat lengthy preparation.

Our analysis is based on the study of the first-return map to a thin strip placed
along the symmetry axis of the map (\ref{eq:Map}). The minimal orbits are fixed 
points of this map, and we study those having time reversal symmetry. 
The analysis of the return map requires tracking the return orbits, and this is done
through repeated applications of a \textit{strip map}, an acceleration device which 
exploits local integrability. This is a variant of a construct introduced for outer 
billiards of polygons (see \cite[chapter 7]{Schwartz}, and references therein), 
although in our case the strip map has an increasing number of components, 
providing a dynamics of increasing complexity. 
There is a symbolic dynamics associated with the strip map; its cylinder sets
in the return domain are congruence classes modulo a nested sequence of two-dimensional lattices.
The key result is that, within a polygon class, this lattice structure becomes 
independent of $\lambda$, provided that $\lambda$ is small enough. 
This fact gives a `non archimedean' character to the dynamics; 
the rationality of the density of minimal orbits, and their locally 
uniform distribution then follow.

The plan of this paper is the following.
In section \ref{sec:IntegrableLimit} we introduce the integrable Hamiltonian, 
characterise its invariant curves via a symbolic dynamics, and connect them to 
the arithmetical problem of sums of two squares (theorem \ref{thm:Polygons}).
In section \ref{sec:Recurrence} we switch on the round-off perturbation,
and show that all orbits recur to a small neighbourhood of the symmetry axis. 
Accordingly, we construct a return map of this neighbourhood, and show that
the return orbits shadow the integrable orbits (theorem \ref{thm:Hausdorff}).
Most proofs for this section are postponed to section \ref{sec:Proofs}.
Matching the symbolic dynamics of integrable and perturbed orbits is more 
delicate, requiring the exclusion of certain anomalous domains, and establishing 
that the size of these domains is negligible in the limit. 
This is done in section \ref{sec:MainTheorems}, where we also state the main 
results of this work, theorems 
A %\ref{thm:Phi_equivariance} 
and 
B. %\ref{thm:minimal_densities}. 
The first theorem states that, for all sufficiently small $\lambda$, the return map 
commutes with translations 
by the elements of a two-dimensional lattice, which depends only on the polygonal 
class being considered.
The second theorem states that, if the symbolic dynamics of a polygonal class
satisfies certain coprimality conditions, then, as $\lambda\to 0$ the density 
of symmetric minimal orbits among all symmetric orbits converges to a positive
rational number, which is computed explicitly.
An immediate corollary of theorem B is the existence of a positive
rational lower bound for the density of minimal orbits ---symmetric or otherwise---
among all orbits (corollary \ref{cor:Density}).
At the end of section \ref{sec:MainTheorems} we also briefly discuss some experiments 
on the conditions of the statement of theorem B. 
In section \ref{sec:StripMap} we introduce the strip map, and establish
some of its properties (propositions \ref{thm:regularity} and \ref{thm:epsilon_j}).
In the final section we demonstrate the link between the symbolic dynamics 
of the perturbed orbits and the aforementioned group of lattice translations. 
This result leads to the conclusion of the proof of the main theorems. 

There are many issues we haven't considered.
The nature of the phase portrait at infinity, stability, the role played by 
non-symmetric orbits, the distribution of periods among the branches of the 
period function. More generally, one may consider the construct of strip maps 
for the purpose of developing a Hamiltonian perturbation theory over discrete 
spaces.

These questions deserve further investigation.

\bigskip\noindent
{\sc Acknowledgements.} We are grateful to J A G Roberts for engaging discussions, 
and to the referees, whose comments helped us improve the clarity and correctness of the paper.

%%%%%%%%%%%%%%%%%%%%%%%%%%%%%%%%%%%%%%%%%%%%%%%%%%%%%%%%%%%%%%%%%%%%%%%%%%
\section{The integrable limit}\label{sec:IntegrableLimit}
%%%%%%%%%%%%%%%%%%%%%%%%%%%%%%%%%%%%%%%%%%%%%%%%%%%%%%%%%%%%%%%%%%%%%%%%%%

Figure \ref{fig:PolygonalOrbits} suggests that the analysis of the limit 
$\lambda\to 0$ requires some scaling; equation (\ref{eq:Map}) suggests
that the quantity to be held constant should be $\lambda x$.
Accordingly, we normalise the metric by introducing 
the \textbf{scaled lattice map} $F_{\lambda}$, which is conjugate to $F$,
and acts on points $z=\lambda(x,y)$ of the scaled lattice $(\lambda\Z)^2$:
\begin{equation*}%\label{eq:F_lambda}
F_{\lambda}: (\lambda\Z)^2 \rightarrow (\lambda\Z)^2 
 \hskip 40pt
 F_{\lambda}(z)=\lambda F(z/\lambda).
\end{equation*}
The discretisation length of $F_\lambda$ is $\lambda$.
Then we define the \textbf{discrete vector field}, which measures the 
deviation of $F_{\lambda}^4$ from the identity:
\begin{equation}\label{eq:v}
 \mathbf{v}: \; (\lambda\Z)^2 \rightarrow (\lambda\Z)^2
\hskip 40pt
\mathbf{v}(z) = F_{\lambda}^4(z)-z.
\end{equation}

To capture the main features of $\mathbf{v}$ on the scaled lattice, 
we introduce an \textbf{auxiliary vector field} $\mathbf{w}$ on the plane, 
given by
\begin{equation}\label{eq:w}
 \mathbf{w}: \; \R^2 \rightarrow (\lambda\Z)^2
\hskip 40pt
\mathbf{w}(x,y)=\lambda(2\lfloor y \rfloor+1,-(2\lfloor x \rfloor+1)).
\end{equation}
The field $\mathbf{w}$ is constant on every translated unit square 
(called a \textit{box})
\begin{equation}\label{eq:B_mn}
 B_{m,n} = \{ (x,y)\in\R^2 : \lfloor x \rfloor =m, 
   \lfloor y \rfloor = n \}, \quad m,n\in\Z
\end{equation}
and we denote the value of $\mathbf{w}$ on $B_{m,n}$ as
\begin{equation} \label{eq:w_mn}
 \mathbf{w}_{m,n}=\lambda(2n+1,-(2m+1)). 
\end{equation}

The following proposition, whose proof we defer until section 
\ref{sec:Proofs}, states that if we ignore a set of points of zero 
density, then the functions ${\bf v}$ and ${\bf w}$ agree on
the lattice $(\lambda\Z)^2$.

\begin{proposition} \label{prop:mu_1}
 Let $r$ be a positive real number. We define the set
\begin{equation}\label{eq:A}
 A(r,\lambda) = \{ z\in(\lambda\Z)^2 : \| z \|_{\infty} < r \},
\end{equation}
(with $\| (u,v) \|_{\infty} = \max (|u|,|v|)$), 
and the ratio
 \begin{displaymath}
  \mu_1(r,\lambda) = \frac{ \# \{z\in A(r,\lambda) : \mathbf{v}(z) = \mathbf{w}(z) \} }{\# A(r,\lambda)} .
 \end{displaymath}
Then we have
 \begin{displaymath}
  \lim_{\lambda\rightarrow 0} \mu_1(r,\lambda) = 1 .
 \end{displaymath}
\end{proposition}

The asymptotic regime that results from replacing ${\bf v}$ by ${\bf w}$ will 
be referred to as the \textbf{integrable limit} of the system \eqref{eq:Map}, 
as $\lambda\to 0$. The points where the two vector fields differ have
the property that $\lambda x$ or $\lambda y$ is close to an integer. 
The perturbation of the integrable orbits will take place in these small domains.

%---------------------------------------------------------------------
\subsection{The integrable Hamiltonian}\label{sec:Hamiltonian}

We define the real function 
\begin{equation}\label{eq:P}
P:\R\to \R \hskip 40pt P(x)=\fl{x}^2+(2\fl{x}+1)\{x\}
\end{equation}
where $\{x\}$ denotes the fractional part of $x$.
The function $P$ is piecewise affine, and coincides with the function 
$x\mapsto x^2$ on the integers.
Thus:
\begin{equation}\label{eq:SqrtP}
 P(\lfloor x \rfloor) = \lfloor x \rfloor^2 \hskip 40pt \lfloor \sqrt{P(x)} \rfloor = \lfloor x \rfloor.
\end{equation}
Using this fact, we can invert $P$ up to sign by defining
\begin{equation}\label{eq:Pinv}
 P^{-1}: \R_{\geq 0} \rightarrow \R_{\geq 0} 
\hskip 40pt 
 P^{-1}(x) = \frac{x + \lfloor \sqrt{x} \rfloor (1+\lfloor \sqrt{x} 
\rfloor)}{2\lfloor \sqrt{x} \rfloor+1} ,
\end{equation}
so that $(P^{-1}\circ P)(x) = |x|$.

We define the following Hamiltonian
\begin{equation}\label{eq:Hamiltonian}
 \cP: \; \R^2 \; \rightarrow \R
 \hskip 40pt
 \cP(x,y) = P(x)+P(y).
\end{equation}
The function $\cP$ is continuous and piecewise affine. 
It is differentiable in $\R^2\setminus \Delta$, where $\Delta$ is
a set of orthogonal lines given by
\begin{equation}\label{eq:Delta}
\Delta=\{(x,y)\in\R^2\,:\,(x-\lfloor x\rfloor)(y-\lfloor y\rfloor)=0\}.
\end{equation}
The set $\Delta$ is the boundary of the boxes $B_{m,n}$, defined in (\ref{eq:B_mn}).
The associated (scaled) Hamiltonian vector field, defined for all points 
$(x,y)\in\R^2\setminus \Delta$, is parallel to the vector field 
$\mathbf{w}$ given in (\ref{eq:w}):
\begin{equation}\label{eq:HamiltonianVectorField}
 \lambda \left( \frac{ \partial\cP(x,y)}{\partial y}, 
   -\frac{\partial\cP(x,y)}{\partial x} \right) = \mathbf{w}(x,y)
 \hskip 40pt 
(x,y)\in \R^2\setminus \Delta.
\end{equation}
The parameter $\lambda$ merely rescales the time.

For a point $z\in\R^2$, we write $\Pi(z)$ for the level set of 
$\cP$ passing through $z$:
\begin{displaymath}
 \Pi(z) = \{ w\in\R^2 : \cP(w) = \cP(z) \}.
\end{displaymath}                                     
Below (theorem \ref{thm:Polygons}) we shall see that these sets are polygons.
The {\bf value} of a polygon $\Pi(z)$ is the real number $\cP(z)$, and if 
$\Pi(z)$ contains a lattice point, then we speak of a \textbf{critical polygon}.
The critical polygons form a distinguished subset of the plane:
\begin{equation*}%\label{eq:Gamma}
\Gamma:=\bigcup_{z\in\Z^2}\Pi(z).
\end{equation*}
All topological information concerning the Hamiltonian $\cP$ is
encoded in the partition of the plane generated by $\Gamma\cup\Delta$.
The elements of $\Gamma$ act as separatrices, whose vertices belong
to $\Delta$.

To characterise $\cP$ arithmetically, we consider the Hamiltonian 
\begin{displaymath}
 \cQ(x,y) = x^2 + y^2
\end{displaymath}
which represents the unperturbed rotations (no round-off) in the limit $\lambda\to 0$. 
Its level sets are circles, and the circles containing lattice points will be
called \textbf{critical circles}. 
By construction, the functions $\cP$ and $\cQ$ coincide over $\Z^2$, and 
hence the value of every critical polygon belongs to $\cQ(\Z^2)$, the set 
of non-negative integers which are representable as the sum of two squares. 
We denote this set by $\cE$.

A classical result, due to Fermat and Euler, states that
a natural number $n$ is a sum of two squares if and only if
any prime congruent to 3 modulo 4 which divides $n$ occurs 
with an even exponent in the prime factorisation of $n$  
\cite[theorem 366]{HardyWright}). 
We refer to $\cE$ as the set of \textbf{critical numbers},
and use the notation
\begin{displaymath}
 \cE = \{e_i: i\geq 0\} = \{0,1,2,4,5,8,9,10,13,16,17,\dots \} .
\end{displaymath}
There is an associated family of \textbf{critical intervals}, defined as
\begin{equation}\label{eq:EulerInterval}
 I^{e_i} = (e_i,e_{i+1}).
\end{equation}
Let us define
$$
\cE(x)=\#\{e\in\cE\,:\,e\leq x\}.
$$
The following result, due to Landau and Ramanujan, gives the asymptotic 
behaviour of $\cE(x)$ (see, e.g., \cite{MoreeKazaran})
\begin{equation}\label{eq:LandauRamanujan}
\lim_{x\to\infty}\frac{\sqrt{\ln x}}{x}\,\cE(x)=K,
\end{equation}
where $K$ is the Landau-Ramanujan constant
$$
K=\frac{1}{\sqrt{2}}\prod_{p\,\,\mathrm{prime}\atop {p\equiv 3 \; ({\rm mod} \; 4)}}
\left(1-\frac{1}{p^2}\right)^{-1/2}\,=\,0.764\ldots .
$$

Furthermore, let $r(n)$ be the number of representations of the integer $n$
as a sum of two squares. To compute $r(n)$, we first factor $n$ as follows
$$
n=2^a\prod p^b\prod q^c
$$
where $p$ and $q$ are primes congruent to 1 and 3 modulo 4, respectively.
(Each product is equal to 1 if there are no prime divisors of the corresponding type.)
Then we have \cite[theorem 278]{HardyWright}
\begin{equation}\label{eq:r}
r(n)=4\prod(b+1)\prod\left(\frac{1+(-1)^c}{2}\right).
\end{equation}
Note that this product is zero whenever $n$ is not a critical number, 
i.e., $r(n)=0$ if $n\notin\cE$.

We now have the following characterisation of the invariant curves of the Hamiltonian $\cP$.
\begin{theorem}\label{thm:Polygons}
The level sets $\Pi(z)$ of $\cP$ are convex polygons, invariant under 
the dihedral group $D_4$, generated by the two orientation-reversing involutions
\begin{equation}\label{eq:Dihedral}
 G:\quad (x,y) \mapsto (y,x) 
\hskip 40pt
 G':\quad (x,y) \mapsto (x,-y).
\end{equation}
The polygon $\Pi(z)$ is critical if and only if
$
 \cP(z)\in\cE.
$
The number of sides of $\Pi(z)$ is equal to
\begin{equation}\label{eq:NumberOfSides}
4(2\left\lfloor\sqrt{\cP(z)}\right\rfloor+1)-r(\cP(z))
\end{equation}
where the function $r$ is given in (\ref{eq:r}).
For every $e\in\cE$, the critical polygon with value $e$ intersects 
one and only one critical circle, namely that with the same value.
The intersection consists of $r(e)$ lattice points, and the polygon lies
inside the circle.
\end{theorem}

\begin{proof}
The symmetry properties follow from the fact that the Hamiltonian $\cP$ is 
invariant under the interchange of its arguments, and the function $P$ is even:
\begin{align*} 
P(-x) &= \lfloor -x \rfloor^2 + \{-x\}(2\lfloor -x \rfloor+1) \\
 &= \left\{ \begin{array}{ll} (-\lfloor x \rfloor-1)^2 - (1-\{x\})(2\lfloor x \rfloor+1) & x\notin\Z\\ 
(-\lfloor x \rfloor)^2 & x\in\Z
\end{array} \right. \\
 &= \lfloor x \rfloor^2 + \{x\}(2\lfloor x \rfloor+1) = P(x).
\end{align*}

The vector field (\ref{eq:HamiltonianVectorField}) is piecewise-constant, 
and equal to $\mathbf{w}_{m,n}$ in the box $B_{m,n}$ (cf.~equations
(\ref{eq:B_mn}) and (\ref{eq:w_mn})). 
Hence a level set $\Pi(z)$ is the union of segments. It is easy to 
verify that no three segments can have an end-point in common 
(considering end-points in the first octant will suffice). 
Thus $\Pi(z)$ is a polygonal curve. Equally, segments cannot intersect
inside boxes, because they are parallel there. But a non self-intersecting 
symmetric polygonal must be a polygon. 

Next we prove convexity. Due to dihedral symmetry, if $\Pi(z)$ is convex 
within the open first octant $0<y<x$, then it is piecewise convex. 
Thus we suppose that $\Pi(z)$ has an edge in the box $B_{m,n}$, where 
$0< n\leq m$. The adjacent edge in the direction of the flow must be 
in one of the boxes
\begin{displaymath}
B_{m,n-1},\quad B_{m+1,n-1},\quad B_{m+1,n}.
\end{displaymath}
Using (\ref{eq:w_mn}) one verifies that the three determinants
\begin{displaymath}
\det(\mathbf{w}_{m,n},\mathbf{w}_{m,n-1})
\qquad
\det(\mathbf{w}_{m,n},\mathbf{w}_{m+1,n-1})
\qquad
\det(\mathbf{w}_{m,n},\mathbf{w}_{m+1,n})
\end{displaymath}
are negative. This means that, in each case, at the boundary between adjacent boxes, 
the integral curve turns clockwise. So $\Pi(z)$ is piecewise convex.
It remains to prove that convexity is preserved across the boundaries of the 
first octant, which belong to the fixed sets $\Fix{G}$ (the line $x=y$)
and $\Fix{G}'$ (the line $y=0$) of the involutions (\ref{eq:Dihedral}). 
Indeed, $\Pi(z)$ is either orthogonal to $\mathrm{Fix}\,G$ (in which case
convexity is clearly preserved), or has a vertex $(m,m)$ on it; in the latter case, 
the relevant determinant is $\det(\mathbf{w}_{m-1,m},\mathbf{w}_{m,m-1})=-8m<0$.
The preservation of convexity across $\Fix{G}'$ is proved similarly, and thus
$\Pi(z)$ is convex.

The statement on the criticality of $\cP(z)$ follows from 
the fact that, on $\Z^2$, we have $\cP=\cQ$.

Consider now the edges of $\Pi(z)$. The intersections of $\Pi(z)$ 
with the $x$-axis have abscissas $\pm P^{-1} (\cP(z))$. 
Using (\ref{eq:SqrtP}) we have that there are 
$2\lfloor\sqrt{\cP(z)}\rfloor+1$ integer points between them, hence as 
many lines orthogonal to the $x$-axis with integer abscissa. 
The same holds for the $y$-axis. If $\Pi(z)$ is non-critical, 
it follows that $\Pi(z)$ intersects $\Delta$ in exactly 
$4(2\left\lfloor\sqrt{\cP(z)}\right\rfloor+1)$ points, 
each line being intersected twice. 
Because the vector field changes across each line, 
the polygon has $4(2\left\lfloor\sqrt{\cP(z)}\right\rfloor+1)$ vertices. 
If the polygon is critical, then we have $\cP(z)=e\in\cE$. At each 
of the $r(e)$ vertices that belong to $\Z^2$, two lines in $\Delta$ 
intersect, resulting in one fewer vertex. So $r(e)$ vertices must be removed 
from the count.

Next we deal with intersections of critical curves.
Let us consider two arbitrary critical curves
$$
\cP(x,y)=e
\qquad
\cQ(x,y)=e+f
\qquad
e, e+f\in\cE.
$$
This system of equations yields
$$
\{x\}^2 + \{y\}^2 -\{x\} -\{y\}=f
$$
which is a circle with centre at $(1/2,1/2)$, and radius $\rho$, where
\begin{equation}\label{eq:rho}
\rho^2=f+\frac{1}{2}.
\end{equation}
Since we must have $0\leq\{x\},\{y\}<1$, we find $\rho^2\leq 1/2$, and
since $f$ is an integer, we obtain $\{x\}=\{y\}=f=0$. So critical
polygons and circles intersect only if they have the same value, 
and their intersection consists of lattice points. Then the 
number of these lattice points is necessarily equal to $r(e)$.

Finally, let the lattice point $(m,n)$ belong to the intersection 
of two critical curves, and let $\mathbf{w}'_{m,n}$ be the
vector field of the Hamiltonian $\cQ$ at that point.
Without loss of generality, we assume that $(m,n)$ lies within 
the first octant. If $n>0$, then the vector field of $\cP$
before and after the vertex in the direction of the flow
is equal to $\mathbf{w}_{m-1,n}$ and $\mathbf{w}_{m,n-1}$, 
respectively. One verifies that the determinants
\begin{displaymath}
\det(\mathbf{w}_{m-1,n},\mathbf{w}_{m,n}')
\qquad
\det(\mathbf{w}_{m,n}',\mathbf{w}_{m,n-1})
\end{displaymath}
are negative. This means that, near this vertex, the polygon lies 
inside the circle. If $n=0$, the field after the vertex is 
$\mathbf{w}_{m-1,-1}$, and the same result holds.

The proof is complete.
\end{proof}

From this theorem it follows that the set $\Gamma$ of critical polygons
partitions the plane into concentric domains, which we call \textbf{polygon classes}. 
Each domain contains a single critical circle, and has no lattice points in its interior.
The values of all the polygons in a class is a critical interval (\ref{eq:EulerInterval}).
There is a dual arrangement for critical circles. 
Because counting critical polygons is the same as counting critical circles,
the number of critical polygons (or, equivalently, of polygon classes) contained 
in a circle of radius $\sqrt{x}$ is equal to $\cE(x)$, with asymptotic 
formula (\ref{eq:LandauRamanujan}).
From equation (\ref{eq:rho}), one can show that the total variation 
$\Delta\cQ(a)$ of $\cQ$ along the polygon $\cP(z)=a$ satisfies the bound
$$
\Delta\cQ(a)\leq \frac{1}{2}
$$
which is strict (e.g., for $a=1$). 

%Each box $B_{m,n}$ decomposes into the disjoint union of
%\textit{finitely many} convex sets  $B_{m,n}^{e}$, given by
%\begin{equation*}%\label{eq:B_mne}
% B_{m,n}^{e}=\{z\in B_{m,n}, \cP(z)\in\cE\}
%\end{equation*}
%we have that each $B_{m,n}^{e}$ is connected and convex. 
%
%Finally, we consider the variation of the function $\cQ$ along 
%the level sets of $\cP$.
%The Poisson brackets of the two functions is given by
%%%%%%%%%%%%%%\remark{remove?}
%$$
%\cS=\{\cP,\cQ\}=\frac{\partial \cP}{\partial x}\frac{\partial \cQ}{\partial y}
%          -\frac{\partial \cP}{\partial y}\frac{\partial \cQ}{\partial x}
%          =2(2\fl{x}+1)y-2(2\fl{y}+1)x.
%$$
%The solutions of the equation $\cS(x,y)=0$ are the points where polygons 
%and circles are tangent.
%The corresponding Hamiltonian vector field is given by
%%\begin{equation}\label{eq:VectorFieldR}
%\left(-\frac{\partial \cS}{\partial y},\frac{\partial \cS}{\partial x}\right)
%\,=\,-2(2\lfloor x\rfloor+1,\,2\lfloor y\rfloor+1),
%\end{equation}
%and it is orthogonal to ${\bf w}$.

%\begin{lemma}\label{eq:LocalMaxima}
%All elements of the sequence $\cE$ are local maxima of $\Delta\cQ$.
%\end{lemma}

%---------------------------------------------------------------------------------
\subsection{Symbolic dynamics of polygon classes} \label{sec:Coding}

In theorem \ref{thm:Polygons} we classified the invariant curves of 
the Hamiltonian $\cP$ in terms of critical numbers. We found that the
set $\Gamma$ of critical polygons partitions the plane into concentric annular 
domains ---the polygon classes. 
In this section we define a symbolic dynamics on the set of classes, 
which specifies the common itinerary of all orbits in a class,
taken with respect to the lattice $\Z^2$. 

Suppose that the polygon $\Pi(z)$ is non-critical. Then all vertices of 
$\Pi(z)$ belong to $\Delta\setminus \Z^2$, where $\Delta$ was 
defined in (\ref{eq:Delta}). Let $\xi$ be a vertex.
Then $\xi$ has one integer and one non-integer coordinate, and we let $u$
be the value of the non-integer coordinate. 
We say that the vertex $\xi$ is of \textbf{type} $v$ if $\lfloor |u| \rfloor =v$.
Then we write $v_j$ for the type of the $j$th vertex, where the vertices of $\Pi(z)$ 
are enumerated according to their position in the plane, 
starting from the positive half of symmetry line $\Fix{G}$ and proceeding clockwise.

As the type of a vertex is defined using the modulus of the non-integer coordinate, 
the sequence of vertex types $v_j$ reflects the eight-fold symmetry of $\Pi(z)$. 
Hence if the $k$th vertex lies on the $x$-axis, then there are $2k-1$ vertices 
belonging to each quarter-turn, and the vertex types satisfy
\begin{equation} \label{eq:v_symmetry}
 v_j = v_{2k-j} = v_{(2k-1)i+j}, \hskip 20pt 
1\leq j \leq k, \quad 0\leq i \leq 3. 
\end{equation}
Thus it suffices to consider the vertices in the first octant, and 
the \textbf{vertex list} of $\Pi(z)$ is the sequence of vertex types
\begin{displaymath}
V=(v_1,\dots,v_{k}).
\end{displaymath}
We note that the vertex list can be decomposed into two disjoint subsequences; 
those entries belonging to a vertex with integer $x$-coordinate and those 
belonging to a vertex with integer $y$-coordinate. 
These subsequences are non-decreasing and non-increasing, respectively.

From theorem \ref{thm:Polygons}, it follows that for 
every $e\in\cE$, the set of polygons $\Pi(z)$ with
$
 \cP(z)\in I^e
$
have the same vertex list.
Let $k$ be the number of entries in the vertex list. Since the polygon
$\Pi(z)$ is non-critical, equation (\ref{eq:NumberOfSides}) 
gives us that 
$4( 2\lfloor \sqrt{e} \rfloor +1) = 4(2k-1)$, and hence
\begin{equation*}%\label{eq:LengthOfVertexList}
k=\#V=\lfloor \sqrt{e}\rfloor +1.
\end{equation*}
Any two polygons with the same vertex list have not only the same number of edges, 
but intersect the same collection of boxes, and have the same collection of tangent 
vectors. The critical polygons which intersect the lattice $\Z^2$, where 
the vertex list is multiply defined, form the boundaries between classes. 
The symbolic dynamics of these polygons is ambiguous, but this item will not
be required in our analysis.

Thus the vertex list is a function on classes, hence on $\cE$. 
For example, the polygon class identified with the interval $I^9=(9,10)$ 
(see figure \ref{fig:V(9)}) has vertex list
\begin{displaymath}
V(9)=( 2, 2, 0, 3).
\end{displaymath}
See figure \ref{fig:V_table} for further examples of V(e). 
For each class, there are two vertex types which we can calculate explicitly; 
the first and the last. If $\cP(z)\in I^e$, and the polygon $\Pi(z)$ intersects 
the symmetry line $\Fix{G}$ at some point $(x,x)\in\R^2$, then by the definition 
(\ref{eq:Hamiltonian}) of the Hamiltonian $\cP$
 $$ \cP(z) = 2 P(x). $$
Thus inverting $P$ and using (\ref{eq:SqrtP}), it is straightforward to show that
the first vertex type is given by
\begin{equation}\label{eq:v1}
 v_1 = \lfloor |x| \rfloor = \lfloor P^{-1}(\cP(z)/2) \rfloor = \lfloor \sqrt{e/2} \rfloor. 
\end{equation}
Similarly the last vertex type, corresponding to the vertex on the $x$-axis is given by
\begin{equation} \label{eq:vk}
 v_k = \lfloor P^{-1}(\cP(z)) \rfloor = \lfloor \sqrt{e} \rfloor. 
\end{equation}

%%%%%%%%%%%%%%%%%%%%%%%%%%%%%%%%%%%%%%%%%%%%%%%%%%%%%%%% FIGURE
\begin{figure}[!h]
        \centering
        \vspace*{-100pt} %%%%%
	\epsfig{file=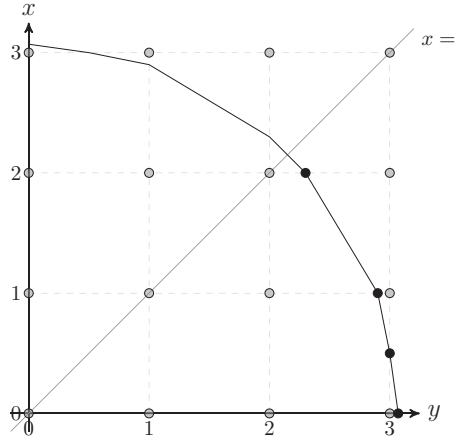,scale=0.8}
        \vspace*{-370pt} %%%%%
        \caption{A polygon with $\cP(z)$ in the interval $(9,10)$ and its vertices in the first octant.}
        \label{fig:V(9)}
\end{figure}

%%%%%%%%%%%%%%%%%%%%%%%%%%%%%%%%%%%%%%%%%%%%%%%%%%%%%%%% FIGURE
\begin{figure}[!h]
        \centering
                \begin{tabular}{|l|l|}
                \hline
                e & V(e) \\
                \hline
                9 & $ (2,2,0,3) $\\
                10 & $(2,1,3,3)$ \\
                18 &  $(3,3,1,4,4)$ \\
                29 & $(3,4,2,5,5,5)$ \\
                49 & $(4,5,3,6,6,6,0,7)$ \\
                52 & $(5,4,6,6,6,1,7,7)$ \\
                \hline
                \end{tabular}
        \caption{A table showing the vertex list $V(e)$ for a selection of 
 critical numbers $e$. Notice that the last entry in the vertex list is 
 always $\lfloor \sqrt{e}\rfloor$, and the
 number of entries in the list is $k=\lfloor \sqrt{e}\rfloor+1$.}
        \label{fig:V_table}
\end{figure}

%%%%%%%%%%%%%%%%%%%%%%%%%%%%%%%%%%%%%%%%%%%%%%%%%%%%%%%%%%%%%%%%%%%%%%%%%%
\section{Recurrence and return map}\label{sec:Recurrence}
%%%%%%%%%%%%%%%%%%%%%%%%%%%%%%%%%%%%%%%%%%%%%%%%%%%%%%%%%%%%%%%%%%%%%%%%%%

(The proofs of all statements in this section are deferred until section 
\ref{sec:Proofs}.)

The lattice map $F$ has fewer symmetries than the Hamiltonian $\cP$,
but it is easy to verify that $F$ is still reversible, being conjugate 
to its inverse via the involution $G$ given in equation (\ref{eq:Dihedral}):
\begin{equation*}%\label{eq:Reversibility}
F^{-1}=G\circ F\circ G\hskip 40pt G^2=\mathrm{Id}.
\end{equation*}
The scaled map $F_\lambda$ has the same property, and all orbits of 
$F_\lambda$ return repeatedly to a neighbourhood of the symmetry line 
$\Fix{G}$. In this section we consider the return map associated to 
this recurrence, and identify some asymptotic properties of the first-return
orbits.

From equation (\ref{eq:Map}), the rotation number $\nu$ has the asymptotic form
\begin{displaymath}
 \nu = \frac{1}{2\pi}\arccos\left(\frac{\lambda}{2}\right) 
= \frac{1}{4}- \frac{\lambda}{4\pi} + O(\lambda^3) \hskip 40pt \lambda\rightarrow 0.
\end{displaymath}
The integer $t=4$ is the \textbf{zeroth-order recurrence time} of orbits under 
$F_{\lambda}$, that is, the number of iterations needed for a point to return 
to an $O(\lambda)$-neighbourhood of its starting point. It turns out (see proof of 
proposition \ref{prop:mu_1}) that the field ${\bf v}(z)$ (equation (\ref{eq:v})) is non-zero 
for all non-zero points $z$, so no orbit has period four.
Accordingly, for small $\lambda>0$, we define the \textbf{first-order recurrence time} 
$t^*$ of the rotation to be the next time of closest approach:
\begin{equation}\label{eq:tstar}
 t^*(\lambda) = \inf \left\{k\in\N \,: \,d_H(k\nu,\N) < d_H(4\nu,\N)\right\}
              = \frac{\pi}{\lambda} + O(1)
\hskip 40pt \lambda\to 0,
\end{equation}
where $d_H$ is the Hausdorff distance, and the expression $d_H(x,A)$, with $x\in\R$, 
is to be understood as the Hausdorff distance between the sets $\{x\}$ and $A$. 
(Throughout this paper we use $\N$ to denote the set of positive integers.)

The integer $t^*$ provides a natural recurrence timescale for $F_{\lambda}$. 
Let $T(z)$ be the minimal period of the orbit under $F$ of the point $z\in\Z^2$ 
, so that $T(z/\lambda)$ is the corresponding function for 
points $z\in(\lambda\Z)^2$ under $F_{\lambda}$.
(In accordance with the periodicity conjecture, we assume that this 
function is well-defined.) 
Since, as $\lambda\to 0$, the recurrence time $t^*$ diverges,
the periods of the orbits will cluster around integer multiples 
of $t^*$, giving rise to branches of the period function. 
The lowest branch corresponds to orbits that perform a single revolution 
around the origin, and their period is approximately equal to $t^*$.
The period function $T$ has a normalised counterpart, given by (cf.~(\ref{eq:tstar}))
\begin{equation*}%\label{eq:T_lambda}
T_{\lambda}: (\lambda\Z)^2 \rightarrow \frac{\lambda}{\pi}\,\N
 \hskip 40pt
 T_\lambda(z)=\frac{\lambda}{\pi}T(z/\lambda).
\end{equation*}
The values of $T_\lambda$ oscillate about the integers (figure \ref{fig:PeriodFunction}).

We construct a Poincar\'e return map $\Phi$ on a neighbourhood of the 
positive half of the symmetry line $\Fix{G}$. 
Let $d(z)$ be the perpendicular distance between a point $z$ and 
$\Fix{G}$:
\begin{equation*}%\label{eq:Phi}
 d(z) = d_H(z,\Fix{G}).
\end{equation*}

We define the domain $X$ of the return map $\Phi$ to be the set of points 
$z\in(\lambda\Z_{\geq0})^2$ which are closer to $\Fix{G}$ than their 
pre-images under ${F}_\lambda^4$, and at least as close as their images:
\begin{equation}\label{eq:X}
 X = \{z\in(\lambda\Z_{\geq0})^2: d(z) \leq d(F_{\lambda}^4(z)), \; d(z) < d(F_{\lambda}^{-4}(z)) \}.
\end{equation}
Asymptotically, for every non-negative integer $m$, the set $X$ has non-empty intersection
with the boxes $B_{m,m},\,B_{m+1,m},\,B_{m,m+1}$. The main component is in $B_{m,m}$, a 
thin strip of width $\|\mathbf{w}_{m,m}\|$ lying parallel to the symmetry line $\Fix{G}$ (figure \ref{fig:X}).
To ensure that this component is non-empty, we require $\lambda<\lambda_m$, 
where the critical parameter $\lambda_m$ is given by
\begin{equation}\label{eq:lambda_m}
 \lambda_m=\frac{1}{6(m+1)}.
\end{equation}

%%%%%%%%%%%%%%%%%%%%%%%%%%%%%%%%%%%%%%%%%%%%%%%%%%%%%%%% FIGURE
\begin{figure}[!h]
        \centering
\epsfig{file=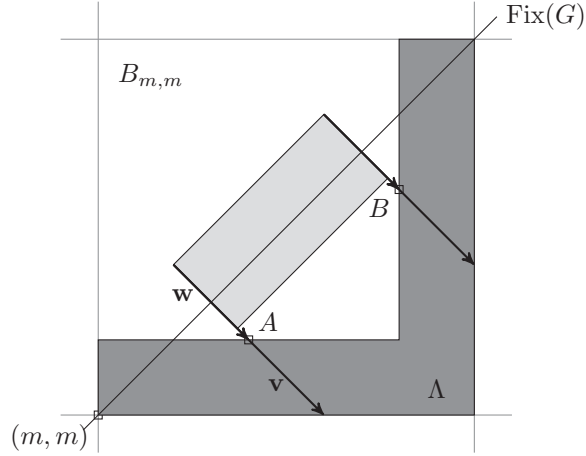,scale=1.0}
        \caption{Structure of the box $B_{m,m}$. The the strip containing the main component 
        of $X$ is shown in light grey; the intersection of $B_{m,m}$ with the set $\Lambda$ 
        of equation (\ref{eq:Lambda}) is shown in dark grey. 
        The set $X$ comprises the lattice points which are closer to $\Fix{G}$ than 
        their pre-images under translation by $\mathbf{v}$, and at least as close as 
        their images. In $B_{m,m}\setminus\Lambda$ we have $\mathbf{v}=\mathbf{w}_{m,m}$  
        (lemma \ref{thm:Lambda}), so that the width of the strip is $\|\mathbf{w}_{m,m}\|$, 
        and the points $A$ and $B$ lie on the line $y=x - 2m\lambda$. 
        If $A\in\Lambda$, then the $y$-coordinate of $A$ cannot be larger than $m+\lambda(2m+3)$,
        and if $B\in\Lambda$, then the $x$-coordinate of $B$ cannot be smaller 
        than $m+1 - \lambda(2m+3)$ (see proof of lemma \ref{thm:Lambda}). 
        Thus the strip is non-empty as long as $m+\lambda(2m+3)<m+1-\lambda(2m+3) - 2m\lambda$, 
        giving the critical parameter $\lambda_m$ of equation (\ref{eq:lambda_m}).}
        \label{fig:X}
\end{figure}
%%%%%%%%%%%%%%%%%%%%%%%%%%%%%%%%%%%%%%%%%%%%%%%%%%%%%%%%

The transit time $\tau$ to the set $X$ is well-defined for all $z\in(\lambda\Z)^2$: 
\begin{equation} \label{eq:tau}
 \tau : (\lambda\Z)^2 \rightarrow \N
 \hskip 40pt \tau(z) = \min \{ k\in\N :F_{\lambda}^k(z) \in X \}.
\end{equation}
Thus the first return map $\Phi$ is the function
\begin{displaymath}
  \Phi : X \rightarrow X
 \hskip 40pt
   \Phi(z) = F_{\lambda}^{\tau(z)}(z).
\end{displaymath}
We refer to the orbit of $z\in X$ up to the return time $\tau(z)$ as the 
\textbf{return orbit} of $z$:
\begin{equation*}%\label{eq:ReturnOrbit_X}
 {\mathcal O}_{\tau}(z) = \{ F_{\lambda}^k(z)\,:\,0\leq k \leq \tau(z) \}, \hskip 20pt z\in X.
\end{equation*}
We let $\tau_{-}$ be the transit time to $X$ under $F_{\lambda}^{-1}$: 
\begin{displaymath}
 \tau_{-} : (\lambda\Z)^2 \rightarrow \Z_{\geq0}
 \hskip 40pt
  \tau_{-}(z) = \min \{ k\in\Z_{\geq0} :F_{\lambda}^{-k}(z) \in X \},
\end{displaymath}
so that the return orbit for a general $z\in(\lambda\Z)^2$ is given by: 
\begin{equation*}%\label{eq:ReturnOrbit}
 {\mathcal O}_{\tau}(z) = \{ F_{\lambda}^k(z)\,:\, \tau_{-}(z)\leq k \leq \tau(z) \}, \hskip 20pt z\in (\lambda\Z)^2.
\end{equation*}

To associate a unique return orbit to an integrable orbit, we define the rescaled 
round-off function $R_{\lambda}$, which rounds points on the plane down to 
the next lattice point:
\begin{equation*}% \label{eq:R_lambda}
R_{\lambda}: \R^2 \rightarrow (\lambda\Z)^2 
 \hskip 40pt
 R_{\lambda}(w)=\lambda R(w/\lambda),
\end{equation*}
where $R$ is the integer round-off function
\begin{equation*}%\label{eq:R}
 R: \R^2 \rightarrow \Z^2 
 \hskip 40pt
 R(u,v) = (\lfloor u \rfloor, \lfloor v \rfloor ).
\end{equation*}
For every point $w\in\R^2$ and every $\delta>0$, the set of points
$$
\{z\in(\lambda\Z)^2: \, z= R_{\lambda}(w), \,0< \lambda<\delta\}
$$
that represent $w$ on the lattice as $\lambda\to0$ is countably infinite.
The corresponding set of points on $\Z^2$, before rescaling, is unbounded.

In the rest of this section, we state some metric properties of the
return orbits, and then show that the return orbits shadow the integrable orbits. 
The proofs will be found in section \ref{sec:Proofs}. 

According to proposition \ref{prop:mu_1} on page \pageref{prop:mu_1}, the points of the scaled lattice 
$(\lambda\Z)^2$ at which integrable and discrete vector fields have 
different values are rare, as a proportion of lattice points. The following 
result shows that these points are also rare within each return orbit. 

\begin{proposition} \label{prop:mu_2}
 For any $w\in\R^2$, if we define the ratio
 \begin{displaymath}
  \mu_2(w,\lambda) = \frac{\#\{z\in{\mathcal O}_{\tau}(R_{\lambda}(w))\,:\,
   \mathbf{v}(z)=\mathbf{w}(z)\}}{\# {\mathcal O}_{\tau}(R_{\lambda}(w))} ,
 \end{displaymath}
 then we have
 \begin{displaymath}
  \lim_{\lambda\rightarrow 0} \mu_2(w,\lambda) = 1 .
 \end{displaymath}
\end{proposition} 

To establish propositions \ref{prop:mu_1} and \ref{prop:mu_2},
we seek to isolate the lattice points $z\in(\lambda\Z)^2$ where 
the discrete vector field $\mathbf{v}(z)$ deviates from the integrable 
vector field $\mathbf{w}(z)$. We say that a point $z\in(\lambda\Z)^2$ is a 
\textbf{transition point} if $z$ and its image under $F_\lambda^4$
do not belong to the same box, namely if
\begin{displaymath}
 R(F_{\lambda}^4(z))\not=R(z).
\end{displaymath}
Let $\Lambda$ be the set of transition points. Then
\begin{equation} \label{eq:Lambda}
 \Lambda = \bigcup_{m,n\in\Z} \Lambda_{m,n},
\end{equation}
where
\begin{displaymath}
 \Lambda_{m,n} =  F_{\lambda}^{-4}(B_{m,n}\cap(\lambda\Z)^2)\setminus B_{m,n}.
\end{displaymath}

%%%%%%%%%%%%%%%%%%%%%%%%%%%%%%%%%%%%%%%%%%%%%%%%%%%%%%%% FIGURE
\begin{figure}[!h]
        \centering
\epsfig{file=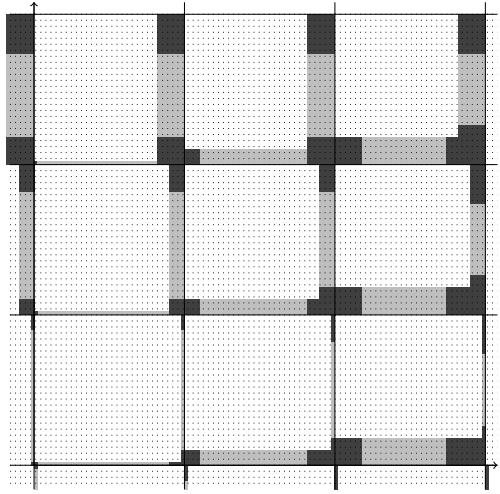,scale=0.8}
        \caption{Structure of phase space. The boxes $B_{m,n}$, bounded by the
         set $\Delta$, include regular domains (white), where the motion is integrable. 
         The darker regions comprise the set $\Lambda$ of transition points, given in (\ref{eq:Lambda}). 
         The darkest domains are the components $\Sigma_{m,n}$ of the set $\Sigma\subset \Lambda$, 
         defined in (\ref{eq:Sigma}).}
        \label{fig:LambdaSigma_plot}
\end{figure}
%%%%%%%%%%%%%%%%%%%%%%%%%%%%%%%%%%%%%%%%%%%%%%%%%%%%%%%%

For small $\lambda$, the set of transition points $\Lambda$ consists of thin 
strips of lattice points arranged along the lines $\Delta$ (see figure 
\ref{fig:LambdaSigma_plot}). The following key lemma states that, for sufficiently 
small $\lambda$, all points $z\not=(0,0)$ where $\mathbf{v}(z)\neq\mathbf{w}(z)$ 
are transition points. 

\begin{lemma} \label{thm:Lambda}
Let $A(r,\lambda)$ be as in equation (\ref{eq:A}).
Then for all $r>0$ there exists $\lambda^*>0$ such that, for all $\lambda<\lambda^*$ 
and $z\in A(r,\lambda)$, we have
\begin{displaymath}
 \mathbf{v}(z)\neq \mathbf{w}(z) \; \Rightarrow \; z \in \Lambda \cup \{(0,0)\}.
\end{displaymath}
\end{lemma}

We conclude this section by formulating a shadowing theorem, which states that
for time scales corresponding to a first return to the domain $X$, every 
integrable orbit has a scaled return orbit that shadows it. 
Furthermore, this scaled orbit of the round-off map 
converges to the integrable orbit in the Hausdorff metric as $\lambda\to 0$.

\begin{theorem} \label{thm:Hausdorff}
For any $w\in\R^2$, let $\Pi(w)$ be the orbit of $\cP$, and let
${\mathcal O}_{\tau}(R_{\lambda}(w))$ be the return orbit at the rounded
lattice point. Then
\begin{displaymath}
 \lim_{\lambda\rightarrow 0} d_H 
   \left(\Pi(w),{\mathcal O}_{\tau}(R_{\lambda}(w))\right)=0,
\end{displaymath}
where $d_H$ is the Hausdorff distance on $\R^2$.
\end{theorem}

This result justifies the term `integrable limit' assigned to the flow $\cP$.

%%%%%%%%%%%%%%%%%%%%%%%%%%%%%%%%%%%%%%%%%%%%%%%%%%%%%%%%%%%%%%%%%%%%%%%%%%
\section{Proofs for section \ref{sec:Recurrence}}\label{sec:Proofs}
%%%%%%%%%%%%%%%%%%%%%%%%%%%%%%%%%%%%%%%%%%%%%%%%%%%%%%%%%%%%%%%%%%%%%%%%%%

\noindent{\sc Proof of lemma \ref{thm:Lambda}}. \quad
Let $r>0$ be given, and let $A(r,\lambda)$ be as in equation (\ref{eq:A}).
We show that if $\lambda$ and $\|z\|_\infty$ are sufficiently small (and $z\not=0$), 
then
\begin{displaymath}
\mathbf{v}(z)\neq \mathbf{w}(z) 
\quad \Rightarrow \quad 
 R(F_{\lambda}^4(z))\not=R(z).
\end{displaymath}
%More precisely, we shall prove that there is a critical value 
%$\lambda^*=\lambda^*(r)$, such that if $\lambda<\lambda^*$, then any point 
%$z=\lambda(x,y)\in A(r,\lambda)\setminus \{(0,0)\}$
%with $\mathbf{v}(z)\neq\mathbf{w}(z)$ must belong to $\Lambda$.

Since $z\in A(r,\lambda)$, we have $z\in B_{m,n}$ for some $|m|,|n|\leq\lceil r \rceil$, 
where $\lceil \cdot \rceil$ is the ceiling function, defined by the 
identity $\lceil x \rceil=-\lfloor -x \rfloor$. 
Through repeated applications of $F_{\lambda}$, we have
\begin{equation}\label{eq:Fabcd}
\begin{array}{llll}
 F_{\lambda}(z)   &= \lambda(-y+m,x)              &R(F_\lambda(z))&=(-(a+1),m),\\
 \noalign{\vspace*{2pt}}
 F^2_{\lambda}(z) &= \lambda(-x-a-1,-y+m)         &R(F_\lambda^2(z))&=(-(b+1),-(a+1)),\\
 \noalign{\vspace*{2pt}}
 F^3_{\lambda}(z) &= \lambda(y-m-b-1,-x-a-1)      &R(F_\lambda^3(z))&=(c,-(b+1)), \\
 \noalign{\vspace*{2pt}}
 F^4_{\lambda}(z) &= \lambda(x+a+c+1,y-m-b-1)     &R(F_\lambda^4(z))&=(d,c),
\end{array}
\end{equation}
where the integers $a,b,c,d$ are given by
\begin{equation}\label{eq:abcd}
\begin{array}{rl}
  a+1 &= \lceil \lambda (y-m) \rceil,\\
 \noalign{\vspace*{2pt}}
  b+1 &= \lceil \lambda (x+a+1) \rceil,  \\
 \noalign{\vspace*{2pt}}
  c &= \lfloor \lambda (y-m-b-1) \rfloor \\
 \noalign{\vspace*{2pt}}
  d &= \lfloor \lambda (x+a+c+1) \rfloor.
\end{array}
\end{equation}
The integers $a$, $b$, $c$ and $d$ label the boxes in which each iterate 
occurs (figure \ref{fig:F_Iterates}),
and also give an explicit expression for the round-off term 
$\lfloor \lambda x \rfloor$ at each step.
Thus reading from the last of these equations, the discrete vector field $\mathbf{v}$ is given by
\begin{equation} \label{eq:v_abcd}
\mathbf{v}(z)= F^4_{\lambda}(z) - z = 
\lambda( a+c+1, -(m+b+1)),
\end{equation}
and $z$ is a transition point whenever at least one of the equalities $d=m$ and $c=n$ on the box labels fails.

%%%%%%%%%%%%%%%%%%%%%%%%%%%%%%%%%%%%%%%%%%%%%%%%%%%%%%%% FIGURE
\begin{figure}[!h]
        \centering
\vspace*{-100pt}
\epsfig{file=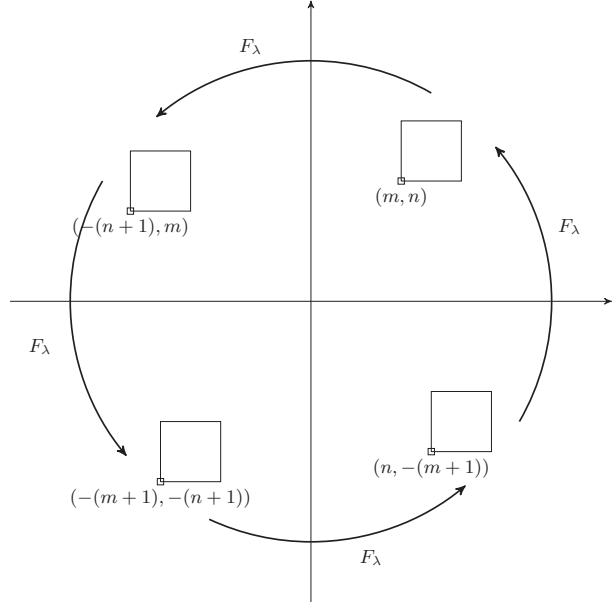,scale=0.8}
\vspace*{-320pt}
        \caption{The boxes in which successive iterates of $z$ under 
$F_\lambda$ occur, for the typical case where $m=b=d$, $n=a=c$.}
        \label{fig:F_Iterates}
\end{figure}
%%%%%%%%%%%%%%%%%%%%%%%%%%%%%%%%%%%%%%%%%%%%%%%%%%%%%%%%

If the integers $m$, $a$, $b$, $c$ are sufficiently small relative to the 
number of lattice points per unit length, i.e., if
\begin{equation} \label{eq:abc_ineq}
\max (|m|,|a+1|,|m+b+1|,|a+c+1|)<1/\lambda, 
\end{equation}
then the map $F^4_{\lambda}$ moves the point $z$ at most one box in each of the $x$ and $y$ directions, 
so that the labels $a$, $b$, $c$ and $d$ satisfy
\begin{equation} \label{eq:bd_ac_sets}
b,d\in\{m-1,m,m+1\}, \hskip 20pt a,c\in\{n-1,n,n+1\}. 
\end{equation}
Similarly, (\ref{eq:abc_ineq}) dictates that the discrepancy between each of the pairs $(b,d)$, $(a,c)$ cannot be too large:
\begin{equation} \label{eq:bd_ac_ineq}
 |b-d|, |a-c| \leq 1.
\end{equation}

Letting $\lambda^* = 1/(2\lceil r \rceil +3)$, we obtain 
\begin{align*}
\max (|m|,|a+1|,|m+b+1|,|a+c+1|) 
        &\leq \max (|m|+|b+1|,|a+1|+|c|) \\
	&\leq \max (2|m|+|b-m|,2|n|+|a-n|+|c-n|)+1 \\
	&\leq 2\lceil r \rceil+3 \leq 1/\lambda^*,
\end{align*}
so that (\ref{eq:bd_ac_sets}) and (\ref{eq:bd_ac_ineq}) hold for all $\lambda<\lambda^*$. 
Then the expression (\ref{eq:v_abcd}) for $\mathbf{v}$, combined with the inequality (\ref{eq:bd_ac_ineq}),
gives that $\mathbf{v}(z)= \mathbf{w}(z)$ if and only if
\begin{equation} \label{eq:v=w}
m=b \hskip 40pt n = a = c. 
\end{equation}

Suppose now that $z$ is not a transition point, so that $c=n$ and $d=m$, 
but that $\mathbf{v}(z)\neq \mathbf{w}(z)$, so that at least one of the 
equalities (\ref{eq:v=w}) fails. 
If $a\neq n$, straightforward manipulation of inequalities 
shows that the only combination of values which satisfies (\ref{eq:bd_ac_sets}) is
\begin{displaymath}
 a=n-1, \, m=0, \, b=-1, \, \lambda y = n.
\end{displaymath}
In turn, using also the inequality (\ref{eq:bd_ac_ineq}), $b\neq m$ implies that 
\begin{displaymath}
 b=m-1, \, n=0, \, a=-1, \, \lambda x = m.
\end{displaymath}
Combining these, we have $m=a+1=b+1=c=0$, 
which corresponds to the unique point $z=(0,0)$.
\endproof

We now use this result to prove propositions \ref{prop:mu_1} and \ref{prop:mu_2}, 
given in sections \ref{sec:IntegrableLimit} and \ref{sec:Recurrence}, respectively.

\medskip

\noindent{\sc Proof of proposition \ref{prop:mu_1}}. \quad
%\begin{proof}[Proof of Proposition \ref{prop:mu_1}] \label{pf:mu_1}
From equation (\ref{eq:A}), we have that the number of lattice points 
in the set $A(r,\lambda)$ is given by 
\begin{displaymath}
  \# A(r,\lambda) = \left( 2 \Big\lceil \frac{r}{\lambda} \Big\rceil -1\right)^2.
\end{displaymath}
By lemma \ref{thm:Lambda}, for sufficiently small $\lambda$, 
every non-zero point $z\in A(r,\lambda)$ 
satisfying $\mathbf{v}(z)\neq \mathbf{w}(z)$ is a transition point, so has 
$z\in \Lambda_{m,n}$ for some $m,n\in\Z$ with $|m|,|n|\leq \lceil r \rceil+1$. 
Furthermore, every set $\Lambda_{m,n}$ is composed of two strips,
each of unit length, and width approximately equal to $\lambda(2m+1)$ 
and $\lambda(2n+1)$, respectively.
We can bound the number of lattice points in the set $\Lambda_{m,n}$ explicitly by
\begin{displaymath}
  \# \Lambda_{m,n} \leqslant \frac{|2m+1|+|2n+1|+c}{\lambda},
\end{displaymath}
for some positive constant $c$, independent of $m$ and $n$. (Indeed $c=3$ is sufficient -- 
cf.~the methods used in the proof of proposition \ref{prop:Xe_tilde}.)
It follows that for fixed $r>0$, as $\lambda\to 0$ we have the estimate
\begin{align*}
  \mu_1(r,\lambda) &= 1 - \frac{ \# \{z\in A(r,\lambda) : 
    \mathbf{v}(z) \neq \mathbf{w}(z) \} }{\# A(r,\lambda)} \\
 &\geq 1 - \frac{ \# \left(A(r,\lambda) \cap \Lambda \right)}{\# A(r,\lambda)} \\
 &\geq 1 - \frac{1}{\# A(r,\lambda)} \sum_{|m|,|n|\leq \lceil r \rceil+1} \#  \Lambda_{m,n} \\
 &\geq 1 - \left( 2 \Big\lceil \frac{r}{\lambda} \Big\rceil -1\right)^{-2} \sum_{|m|,|n|\leq \lceil r \rceil+1} 
          \frac{|2m+1|+|2n+1|+c}{\lambda} \\
 &= 1 - O(\lambda).
\end{align*}
Since $\mu_1(r,\lambda)\leq 1$, the proof is complete.
\endproof%\end{proof} 

\medskip

By construction, the piecewise-constant vector field $\mathbf{w}$ is 
parallel to the Hamiltonian vector field associated 
with $\cP$ (cf.~equation (\ref{eq:HamiltonianVectorField})). It follows that 
$\cP$ is invariant under translation by $\mathbf{w}$ everywhere except across 
the discontinuities of $\mathbf{w}$, i.e., except at transition points, 
which occur near the vertices of integrable orbits.

In the above proof of proposition \ref{prop:mu_1}, we show that transition points $z\in\Lambda$ 
have zero density in the limit $\lambda\rightarrow 0$ in the following sense:
\begin{displaymath}
 \forall r>0: \hskip 20pt \frac{\# \left(A(r,\lambda) \cap \Lambda \right)}{\# A(r,\lambda)} \rightarrow 0,
\end{displaymath}
so that in any bounded region, and for sufficiently small $\lambda$, 
$\cP$ is invariant under translation by $\mathbf{w}$ almost everywhere. 
In turn, lemma \ref{thm:Lambda} implies that $\cP$ is invariant under translation 
by $\mathbf{v}$, and hence under iterates of $F_{\lambda}^4$, almost everywhere. 
Again the exceptions to this rule occur near the vertices of integrable orbits. 

To prove proposition \ref{prop:mu_2} and theorem \ref{thm:Hausdorff}, 
we bound the variation in the Hamiltonian function $\cP$ along perturbed 
orbits ${\mathcal O}_{\tau}(z)$. Since we know $\cP$ is invariant under 
$F_{\lambda}^4$ at all points $z\notin\Lambda$, it remains to consider 
the change in $\cP$ under $F_{\lambda}$, and under $F_{\lambda}^4$ at 
points $z\in\Lambda$.

\medskip

\noindent{\sc Proof of proposition \ref{prop:mu_2}}. \quad

Let $r>0$. Using lemma \ref{thm:Lambda}, we observe that for sufficiently small $\lambda$, 
$\cP$ is invariant under $F_{\lambda}^4$ on $A(r,\lambda)\setminus \Lambda$. 
We begin by bounding the change in $\cP$ under $F_{\lambda}^4$ on the set $A(r,\lambda)\cap\Lambda$.

For any $z,v\in\R^2$ we have:
\begin{align*}
 \left| \cP(z+v)- \cP(z) \right| &= \left| \; \int_{[z,z+v]} \nabla\cP(\xi)\cdot d\mathbf{\xi} \; \right| \\
&\leq \max_{\xi\in[z,z+v]} \left( \|\nabla\cP(\xi)\| \right) \, \|v\|,
\end{align*}
where $[z,z+v]$ denotes the line segment joining the points $z$ and $z+v$, $d\mathbf{\xi}$ is 
the line element tangent to this segment, and $\nabla\cP$ is the gradient of $\cP$, given by:
 $$ \nabla\cP(x,y) = (2\lfloor x \rfloor +1,2\lfloor y \rfloor +1) \hskip 40pt 
(x,y)\in \R^2\setminus \Delta. $$
If $z=\lambda(x,y)$ and $v=\mathbf{v}(z)$ is the discrete vector field, then for sufficiently small $\lambda$, 
equations (\ref{eq:abcd}) and (\ref{eq:bd_ac_sets}) can be combined to give:
\begin{align*}
 \| \mathbf{v}(z) \| &\leq \lambda \sqrt{(|2\lfloor \lambda y \rfloor +1| +2)^2 + (|2\lfloor \lambda x \rfloor +1| +1)^2} \\
 &\leq \lambda \sqrt{(2 |\lambda y| +3)^2 + (2 |\lambda x| +2)^2} \\
 &\leq \lambda (2\|z \| + \sqrt{13}).
\end{align*}
%It follows that for sufficiently large $\|z\|$:
 %$$ \| \mathbf{v}(z) \| \leq \lambda \sqrt{(3 \lambda x)^2 + (3 \lambda y)^2} = 3 \lambda \| z \|. $$
This inequality ensures that the length of the line segment $[z,z+v]$ goes to zero with $\lambda$, 
so that for sufficiently small $\lambda$, the piecewise-constant form of the gradient $\nabla\cP$ gives:
\begin{align*}
 \max_{\xi\in[z,z+v]} (\|\nabla\cP(\xi)\|) 
 &\leq \sqrt{(|2\lfloor \lambda x \rfloor +1| +2)^2 + (|2\lfloor \lambda y \rfloor +1| +2)^2} \\
  &\leq \sqrt{(2 |\lambda x| +3)^2 + (2 |\lambda y| +3)^2} \\
 &\leq 2\|z \| + 3\sqrt{2}.
\end{align*}
%Again, for sufficiently large $\|z\|$, this can be simplified to:
% $$ \|\nabla\cP(z + \|\mathbf{v}\| \hat{\nabla}\cP(z)) \} \leq 3 \| z \|.$$
Substituting these into the first inequality, we have that for sufficiently small 
$\lambda$: %and sufficiently large $\|z\|$:
 $$ \left| \cP(F_{\lambda}^4(z))- \cP(z) \right| = \left|\cP(z+\mathbf{v}(z))- \cP(z) \right| 
 \leq \lambda (2\|z \| + 3\sqrt{2})^2. $$
For the change in $\cP$ under $F_{\lambda}$, if $z=\lambda(x,y)$, then by the same sort of analysis we have:
\begin{align*}
 \left| \cP(F_{\lambda}(z))- \cP(z) \right|  &= \left| P(\lambda(\lfloor\lambda x\rfloor -y)) - P(\lambda y) \right|\\
&= \left| P(\lambda(y-\lfloor\lambda x\rfloor)) - P(\lambda y) \right| \\
&\leq \lambda |\lfloor\lambda x\rfloor| \, (|P^{\prime}(\lambda y)|+2) \\
&\leq \lambda |\lfloor\lambda x\rfloor| \, (|2\lfloor \lambda y \rfloor +1| +2) \\
&\leq \lambda (|\lambda x| +1) \, (2|\lambda y| +3) \\
&\leq 2 \lambda (\|z\| +2) ^2,
\end{align*}
where $P$ is the piecewise-affine function defined in equation (\ref{eq:P}). 
(See page \pageref{thm:Polygons} for the proof that $P$ is even.) 
It follows that for any orbit contained in $A(r,\lambda)$, if $k\in\Z_{\geq0}$ and $0\leq l<4$:
 $$ \left| \cP(F_{\lambda}^{4k+l}(z)) - \cP(z) \right| 
 \leq m \lambda (2\|z \| + 3\sqrt{2})^2 + 2l \lambda (\|z\| +2) ^2 $$
where $m$ is the number of transition points in the orbit of $z$ under $F_{\lambda}^4$:
 $$ m = \# \left( \{ z, F_{\lambda}^4(z), \dots, F_{\lambda}^{4k}(z) \} \cap \Lambda \right). $$
Similar expressions hold in backwards time, for iterates of $F_{\lambda}^{-4}$ and $F_{\lambda}^{-1}$. 
For fixed $\lambda$, this estimate bounds the perturbed orbit of a point $z\in(\lambda\Z)^2$ to a polygonal annulus 
around the polygon $\Pi(z)$, which grows in thickness as the number of transition points in the orbit increases. 
Conversely, by taking a sufficiently small value of $\lambda$, we can ensure that the perturbed orbit of $z$ 
stays arbitrarily close to the polygon $\Pi(z)$ for any finite number of transition points. 
Specifically, given any $\epsilon>0$ and $m\in\N$, there is a critical value of $\lambda$ below 
which any point $F_{\lambda}^{4k+l}(z)$ along the perturbed orbit satisfies 
 $$ \left| \cP(F_{\lambda}^{4k+l}(z)) - \cP(z) \right| < \epsilon, $$
 as long as the number of transition points encountered is not more than $m$:
  $$ \# \left( \{ z, F_{\lambda}^4(z), \dots, F_{\lambda}^{4k}(z) \} \cap \Lambda \right) \leq m. $$

Now let $w\in \R^2$ be given, and let $z=R_{\lambda}(w)$, so that ${\mathcal O}_{\tau}(z)$ 
is the return orbit which shadows the integrable orbit $\Pi(w)$. 
Furthermore let $r^{\pm}=P^{-1}(\cP(w))\pm1$, where $P^{-1}(\cP(w))$ is the 
abscissa of the intersection of $\Pi(w)$ with the positive $x$-axis. 
Since $\| z-w \|<\sqrt{2}\lambda$, using the same bounds as above, we have:
\begin{align*}
 \left| \cP(z)- \cP(w) \right| &\leq \max_{\xi\in[w,z]} (\|\nabla\cP(\xi)\|) \, \|z-w\| \\
&\leq \sqrt{2}\lambda(2\|w \| + 3\sqrt{2}),
\end{align*}
so that for small $\lambda$, the polygons $\Pi(z)$ and $\Pi(w)$ are close.

Pick $0<\epsilon<1$ and consider the polygons $\Pi(w)^{\pm}$ given by
 $$ \Pi(w)^{\pm} = \{ \xi : \; \cP(\xi) = \cP(w)\pm\epsilon \}.  $$
Without loss of generality, we may assume that neither of these polygons 
is critical. Thus each of these integrable orbits intersects as many boxes 
as it has sides. For the larger polygon $\Pi(w)^+$, the number of sides 
(see theorem \ref{thm:Polygons}) is given by
 $$ 4\left(2\Big\lfloor \sqrt{\cP(w)+ \epsilon} \Big\rfloor+1 \right), $$
and we let this number be $m$. 
Furthermore, we note that all lattice points in the interior of $\Pi(w)^+$ are 
elements of the set $A(r^+,\lambda)$, since:
 $$ \max \{ \, \|\xi\|_{\infty} : \; \xi\in\Pi(w)^{+} \, \} = P^{-1}(\cP(w)+\epsilon) < r^+. $$

Then, as we showed above, for sufficiently small $\lambda$, 
the perturbed orbit ${\mathcal O}_{\tau}(z)$ is bounded between 
$\Pi(w)^{\pm}$ as long as the number of transition points 
encountered does not exceed $m$. 
By construction (cf.~equation (\ref{eq:Lambda})), the return orbit ${\mathcal O}_{\tau}(z)$ 
contains at most one transition point for every box it intersects. However, as long as the perturbed orbit 
is bounded above by the integrable orbit $\Pi(w)^{+}$, it cannot intersect more boxes than this  
integrable orbit in any one revolution around the origin. 
Hence the perturbed orbit cannot intersect more than $m$ boxes in any one revolution, 
and its number of transition points is bounded above as follows:
 $$ \# \left({\mathcal O}_{\tau}(z)\,\cap\,\Lambda \right) \leq m. $$

Now we consider the total number of points in the return orbit ${\mathcal O}_{\tau}(z)$. 
Since the perturbed orbit is bounded below by the integrable orbit $\Pi(w)^{-}$, it must contain a point $\xi$,  
close to the positive $x$-axis, with $x$-coordinate not less than:
 $$ P^{-1}(\cP(w) - \epsilon) > r^-. $$
Similarly it must contain a point $\xi$, close to the negative $x$-axis, 
with $x$-coordinate not greater than $-r^-$.
The return orbit moves between these points via the action of $F_{\lambda}^4$, 
i.e., by translations of the vector field $\mathbf{v}$. 
Hence, for sufficiently small $\lambda$, the number of steps required for the orbit 
to move from one point to the other in, say, the upper half-plane, is bounded below by 
the distance $2r^-$ divided by the maximal length of $\mathbf{v}$ along the orbit:
\begin{displaymath}
 \{ \lambda(x,y)\in {\mathcal O}_{\tau}(z) : y> 0 \} 
  \geq \frac{2r^-}{\lambda (2\|z \| + \sqrt{13})}
  \geq \frac{2r^-}{\lambda (2\|w \| + 4)}.
\end{displaymath}

Thus, as $\lambda\rightarrow 0$, we have the estimate
 \begin{align*}
  \mu_2(w,\lambda) &= 1 - \frac{ \# \{\xi\in  {\mathcal O}_{\tau}(z) : 
   \mathbf{v}(\xi) \neq \mathbf{w}(\xi) \} }{\# {\mathcal O}_{\tau}(z)}, \\
 &\geq 1 - \frac{ \# \left({\mathcal O}_{\tau}(z)\,\cap\,\Lambda \right) }{\# {\mathcal O}_{\tau}(z)} , \\
 &\geq 1 - \lambda \; \frac{ m (2\|w \| + 4)}{4r^-}, \\
 &= 1 - O(\lambda).
 \end{align*}
Since $\mu_2(w,\lambda)\leq 1$, the proof is complete.
\endproof%\end{proof}

\medskip

\noindent{\sc Proof of theorem \ref{thm:Hausdorff}}. \quad

%\begin{proof}[Proof of theorem \ref{thm:Hausdorff}] \label{pf:hausdorff}
Let $w\in \R^2$ be given, and let $z=R_{\lambda}(w)$, 
so that ${\mathcal O}_{\tau}(z)$ 
is the return orbit which shadows the integrable orbit $\Pi(w)$. 
We have already seen, in the proof of proposition \ref{prop:mu_2}, 
that by taking a sufficiently small value of $\lambda$, we can ensure that the perturbed orbit of $z$ 
stays arbitrarily close to the polygon $\Pi(z)$.
It follows that:
 $$ \forall \xi\in {\mathcal O}_{\tau}(z): \hskip 20pt d_H(\xi,\Pi(w)) = O(\lambda) $$
as $\lambda\rightarrow 0$.

Neighbouring points $\xi,\xi+\mathbf{v}(\xi)$ in the return orbit ${\mathcal O}_{\tau}(z)$ are also 
$O(\lambda)$-close as $\lambda\rightarrow 0$, so the result follows.
\endproof%\end{proof}

%%%%%%%%%%%%%%%%%%%%%%%%%%%%%%%%%%%%%%%%%%%%%%%%%%%%%%%%%%%%%%%%%%%%%%%%%%
\section{Regular domains and statement of the main theorems}\label{sec:MainTheorems}
%%%%%%%%%%%%%%%%%%%%%%%%%%%%%%%%%%%%%%%%%%%%%%%%%%%%%%%%%%%%%%%%%%%%%%%%%%

In this section we state theorems A and B, after the necessary preparation.
Then we briefly discuss the conditions of theorem B.

In section \ref{sec:IntegrableLimit}, the positive real line was partitioned 
into critical intervals $I^e, e\in\cE$, each corresponding to a distinct polygon 
class of the integrable system; accordingly, the set $\Gamma$ of critical polygons
partitioned the plane.
In section \ref{sec:Recurrence} we reduced the near-integrable round-off map 
$F_\lambda$ to a return map $\Phi$ of a thin domain $X$, placed along the 
symmetry axis. 
In this section we partition the set $X$ into sub-domains which play the same 
role for the perturbed orbits as the intervals $I^e$ for the integrable orbits.

By theorem \ref{thm:Hausdorff}, the return orbit ${\mathcal O}_{\tau}(z)$ of 
a point $z\in X$ shadows the orbit $\Pi(z)$ of the integrable Hamiltonian $\cP$. 
However, the quantity $\cP$ is not constant along perturbed orbits.
To deal with this problem, we introduce the following sequence of sets
\begin{equation*}%\label{eq:X^e}
X^e= \{ z\in X\, :\,\forall w\in {\mathcal O}_{\tau}(z),\, \cP(w) \in I^e \}
\qquad e\in\cE.
\end{equation*}
It remains to match the vertex list associated with $I^e$. 
To this end, it is necessary to replace the sets $X^e$ by smaller 
\textbf{regular domains}, and then prove that, in the limit, these 
domains have full density in $X$.

We start by defining the \textbf{edges} of ${\mathcal O}_{\tau}(z)$, 
as the non-empty sets of the form
\begin{displaymath}
 B_{m,n}\cap {\mathcal O}_{\tau}(z),\qquad m,n\in\Z.
\end{displaymath}
For sufficiently small $\lambda$, consecutive edges 
of ${\mathcal O}_{\tau}(z)$ must lie in adjacent boxes, 
and transitions between edges occur when the orbit meets the set 
$\Lambda$ given in equation (\ref{eq:Lambda}). 
Thus we call the set ${\mathcal O}_{\tau}(z)\,\cap\,\Lambda$ the set of \textbf{vertices} of ${\mathcal O}_{\tau}(z)$.
By analogy with the vertices of the polygons, we say that the return orbit 
${\mathcal O}_{\tau}(z)$ has a vertex on $x=m$ of \textbf{type} $v$ 
if there exists a point $w\in{\mathcal O}_{\tau}(z)$ such that
%%%%%\remark{is relation to original defn of vertex type clear?}
\begin{displaymath}
 w\in B_{m,v} \cap F_{\lambda}^{-4}(B_{m-1,v}) 
  \hskip 20pt \mbox{or} \hskip 20pt w\in B_{m-1,v} \cap F_{\lambda}^{-4}(B_{m,v}).
\end{displaymath}
Similarly for a vertex on $y=n$ of type $v$.
A perturbed orbit is \textbf{critical} if it has a vertex whose type is undefined, 
i.e., if there exists $w\in{\mathcal O}_{\tau}(z)$ such that
\begin{displaymath}
 w\in B_{m,n} \cap F_{\lambda}^{-4}(B_{m\pm1,n\pm1})
\end{displaymath}
for some $m,n\in\Z$ (see figure \ref{fig:critical_vertex}).

%%%%%%%%%%%%%%%%%%%%%%%%%%%%%%%%%%%%%%%%%%%%%%%%%%%%%%%% FIGURE
\begin{figure}[ht]
        \centering
        \vspace*{-80pt}
	\epsfig{file=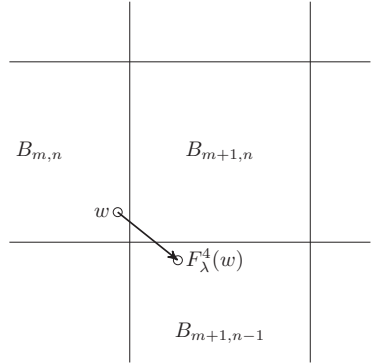,scale=0.8}
        \vspace*{-400pt}
        \caption{A critical vertex $w\in B_{m,n} \cap F_{\lambda}^{-4}(B_{m\pm1,n\pm1})$.}
        \label{fig:critical_vertex}
\end{figure}

By excluding points whose perturbed orbit is critical, we'll construct a subset 
$\Xe$ of $X^e$ with the following properties: \/ $i)$\/ for all 
$z\in\Xe$ the orbit ${\mathcal O}_{\tau}(z)$ has the same sequence 
of vertex types as $\Pi(z)$; \/ $ii)$\/ the union of all sets $\Xe$ still 
has full density in $X$, as $\lambda\to 0$. 

We now give the construction of $\Xe$. 
Let the set $\Sigma\subset\Lambda$ be given by
\begin{equation}\label{eq:Sigma}
 \Sigma = \bigcup_{m,n\in\Z}\Sigma_{m,n},
\end{equation}
where
\begin{displaymath}
 \Sigma_{m,n} = \{ z\in\Lambda: \|z - (m,n)\|_{\infty} \leq \lambda( \, \|(2m+1,2n+1)\|_{\infty}+2) \, \}
\end{displaymath}
and $\|(u,v)\|_{\infty} = \max(|u|,|v|)$. The set $\Sigma_{m,n}$ is a small domain, adjacent
to the integer point $(m,n)$ (see figure \ref{fig:LambdaSigma_plot}).

If $z\in X^e$ for some $e\in\cE$, we say that $z$ is \textbf{regular} if two properties hold: 
firstly, $z$ itself is not a vertex of ${\mathcal O}_{\tau}(z)$, i.e.,
$
 z\in X\setminus \Lambda,
$
and secondly the orbit ${\mathcal O}_{\tau}(z)$ does not intersect the set $\Sigma$.
Points which are not regular are called \textbf{irregular}. Then the set $\Xe$ is defined as
\begin{equation}\label{eq:tildeX^e}
 \Xe = \{ z\in X^e : \cP(z) \in \tilde{I}^e(\lambda) \},
\end{equation}
where $\tilde{I}^e(\lambda)\subset I^e$ is the largest interval such that all points in 
$\Xe$ are regular.

In principle, the interval $\tilde{I}^e(\lambda)$ need not be uniquely defined, and may be empty. 
However, the following proposition ensures that $\tilde{I}^e(\lambda)$ is well-defined for all 
sufficiently small $\lambda$, and indeed that the irregular points have zero density in $X$ 
as $\lambda\rightarrow 0$.

\begin{proposition} \label{thm:I^e_tilde}
 Let $\Xe$ and $\tilde I^e$ be as above. 
 Then, for all $e\in\cE$, we have
 \begin{displaymath}
  \lim_{\lambda\rightarrow 0}\frac{ |\tilde{I}^e(\lambda)| }{|I^e|} = 1.
 \end{displaymath}
\end{proposition}

\begin{proof}
Consider $z\in X$ such that $\cP(z)\in I^e$ for some $e\in\cE$. 
If the orbit of $z$ strays between polygon classes, i.e., if $z\notin X^e$, 
then we have
\begin{displaymath}
 \exists w\in {\mathcal O}_{\tau}(z): \hskip 20pt \cP(w) \notin I^e.
\end{displaymath}
However, in the proof of theorem \ref{thm:Hausdorff} in section \ref{sec:Proofs}, 
we showed that the maximum variation in $\cP$ along an orbit 
${\mathcal O}_{\tau}(z)$ is of order $\lambda$, as $\lambda\rightarrow 0$:
\begin{displaymath}
 \forall z\in X,\, \forall w\in {\mathcal O}_{\tau}(z): 
\hskip 20pt \cP(w) - \cP(z) = O(\lambda).
\end{displaymath}
Hence we have
\begin{equation}\label{eq:Pvariation}
 \cP(z)=\cP(w)+O(\lambda)=
\left\{\begin{array}{ll} 
 e+O(\lambda) & \cP(w)\leq e \\
 e'+O(\lambda) & \cP(w)\geq e' \\
\end{array}\right.
\end{equation}
where $e'$ is the successor of $e$ in the sequence $\cE$. 
In both cases, $\cP(z)$ is near the boundary of $I^e$.

If $z\in X^e$ but $z$ is irregular, then either $z\in\Lambda$ or the orbit of 
$z$ intersects the set $\Sigma$. If $z\in\Lambda$, then one of 
its coordinates must be nearly integer:
$$ 
 d_H(z,\,\Delta) = O(\lambda),
$$
where the set $\Delta$ was defined in (\ref{eq:Delta}).
However, as the domain $X^e$ lies in an $O(\lambda)$-neighbourhood of the 
symmetry line $\Fix{G}$, it follows that both coordinates must be nearly 
integer, giving
\begin{displaymath}
 d_H(z,\,\Z^2) = O(\lambda).
\end{displaymath}
Again, $\cP(z)$ lies in a $O(\lambda)$-neighbourhood of the boundary of $I^e$.

Similarly, if there is a point $w\in{\mathcal O}_{\tau}(z)\cap \Sigma$,
then, by construction,
\begin{displaymath}
 d_H(w,\,\Z^2) = O(\lambda).
\end{displaymath}
Hence we have again the estimate (\ref{eq:Pvariation}).

Combining these observations, we have
$$
  \frac{ |\tilde{I}^e(\lambda)| }{|I^e|} = 1 - \frac{ |I^e\setminus\tilde{I}^e(\lambda)| }{|I^e|}
 = 1 - \frac{ O(\lambda) }{|I^e|}
$$
and the result follows.
\end{proof}

We now give an explicit representation of $\Xe$. Let $v_1$ be 
the first entry in the vertex list of the critical number $e$. 
By the construction (\ref{eq:tildeX^e}) of $\Xe$, we have 
$\Xe\subset B_{v_1,v_1}\setminus\Lambda$. 
Hence, by lemma \ref{thm:Lambda}, the discrete vector field in $\Xe$ satisfies
\begin{displaymath}
 \mathbf{v}(z) = \mathbf{w}(z) = \mathbf{w}_{v_1,v_1} \hskip 40pt z\in \Xe.
\end{displaymath}
Consequently by the definition (\ref{eq:X}) of the Poincare section $X$, if 
$z=\lambda(x,y)\in\Xe$ then 
\begin{displaymath}
 -(2v_1+1) \leq x-y < 2v_1+1.
\end{displaymath}
Hence the set $\Xe$ is given by:
\begin{equation}\label{eq:tildeX^eII}
 \Xe = \{ z=\lambda(x,y)\in(\lambda\Z)^2 : -(2v_1+1) \leq x-y < 2v_1+1, \; \cP(z)\in\tilde{I}_e \}.
\end{equation}

Now we show that the sequence of sets $\Xe$ fulfil their objective, 
which was to exclude all points $z\in X$ whose perturbed orbit is critical 
in the sense defined above.

\begin{proposition} \label{prop:Xe_tilde}
 If $z\in\Xe$ for some $e\in\cE$, then the perturbed orbit of $z$ is not critical.
\end{proposition}
\begin{proof}
Suppose that $z\in X$ with $\cP(z)\in I^e$ is critical, i.e., 
there exists $w\in{\mathcal O}_{\tau}(z)$ such that
\begin{displaymath}
 w\in B_{m,n} \hskip 20pt \mbox{and} 
  \hskip 20pt F_{\lambda}^{4}(w)=w+\mathbf{v}(w)\in B_{m\pm1,n\pm1}
\end{displaymath}
for some $m,n\in\Z$. We will show that $z\notin\Xe$. 
For simplicity, we assume that $m$ and $n$ are both non-negative, so that 
by the orientation of the vector field in the first quadrant:
\begin{displaymath}
 F_{\lambda}^{4}(w)=w+\mathbf{v}(w)\in B_{m+1,n-1}.
\end{displaymath}

Recalling the proof of lemma \ref{thm:Lambda}, the expression (\ref{eq:v_abcd}) 
for the perturbed vector field $\mathbf{v}$ at the point $w=\lambda(x,y)$ 
implies that
\begin{equation}\label{eq:vw}
\mathbf{v}(w)=
\lambda( a+c+1, -(m+b+1)),
\end{equation}
where the integers $a,b,c$ are given by (\ref{eq:abcd}).
By assumption, $w+\mathbf{v}(w)\in B_{m+1,n-1}$, which implies that
\begin{displaymath}
 c=n-1 \hskip 20pt \mbox{and} \hskip 20pt d =m+1.
\end{displaymath}
It follows that the difference between the values $\lambda x$, $\lambda y$ and the integers 
$m+1$, $n$, respectively, is bounded according to
\begin{align*}
-\lambda(a+n) = -\lambda(a+c+1) &\leq \lambda x - (m+1) < 0 \\
0 &\leq \lambda y - n < \lambda(m+b+1).
\end{align*}
Combining this observation with the bounds (\ref{eq:bd_ac_sets}) on $a$ and $b$ gives
\begin{align*}
\| w-(m+1,n) \|_{\infty} &\leq \lambda\max( |a+n|, |m+b+1|) \\
 &\leq \lambda\max( |2n+1| + |n-a|+1, |2m+1| + |m-b|) \\
 &\leq \lambda( \|2m+1,2n+1\|_{\infty} + 2).
\end{align*}
Hence $w\in\Sigma$ and $z$ is irregular, so $z\notin \Xe$. The cases where $m$ or $n$ are negative proceed similarly.
\end{proof}

\medskip

Below, in equation (\ref{eq:Le}) of section \ref{sec:lattice}, we shall define 
a sequence of lattices $\Le\subset(\lambda\Z)^2$, $e\in\cE$,
independent of $\lambda$ up to scaling, such that within the domain 
$\Xe$, the return map $\Phi$ is equivariant under the group of 
translations generated by $\Le$. Formally,

\begin{maintheorem} \label{thm:Phi_equivariance}
For every $e\in\cE$, and all sufficiently small $\lambda$, 
the map $\Phi$ commutes with translations by the elements 
of $\Le$ on the domain $\Xe$
\begin{equation} \label{eq:Phi_equivariance}
 \forall l\in\Le, \; \forall z,z+l\in \Xe,
 \hskip 10pt 
 \Phi(z + l) \equiv \Phi(z) + l \mod{\mathbf{w}_{v_1,v_1}}.
\end{equation}
\end{maintheorem}

There is a critical value of $\lambda$, depending on $e$, above which the statement 
of the theorem is empty, as $\Xe$ is insufficiently populated for a pair 
of points $z,z+l\in \Xe$ to exist. 
We write (mod $\mathbf{a}$) for some vector $\mathbf{a}$ to 
denote congruence modulo the one-dimensional module $\langle\mathbf{a}\rangle$ 
generated by $\mathbf{a}$.
The congruence under the local integrable vector field 
$\mathbf{w}_{v_1,v_1}$ in equation (\ref{eq:Phi_equivariance}) is necessary for 
the case that $\Phi(z) + l\notin X$.

Furthermore, we define the fraction of symmetric, minimal orbits in $\Xe$
\begin{displaymath}
 \delta(e,\lambda) = \frac{\# \{ z\in \Xe : G({\mathcal O}(z))={\mathcal O}(z), \Phi(z)=z \}}{\# \Xe},
\end{displaymath}
and prove the following result on the persistence of such orbits in the limit 
$\lambda\rightarrow 0$.

\begin{maintheorem} \label{thm:minimal_densities}
Let $e\in\cE$, and let $(v_1,\dots,v_k)$ be the vertex list of the 
corresponding polygon class.
If $2v_1+1$ or $2v_k+1$ is coprime to $2v_j+1$ for all other vertex types 
$v_j$, i.e., if
\begin{equation} \label{eq:v1_k_coprimality}
\exists \; i\in\{1,k\}, \quad \forall j\in\{1,\ldots\,k\},\quad
v_j\neq v_i \,\Rightarrow \, \gcd(2v_i+1,2v_j+1) =1,
\end{equation}
then, for sufficiently small $\lambda$, the number of symmetric fixed points
of $\Phi$ in $\Xe$ modulo $\Le$ is independent of $\lambda$. 
Thus the asymptotic density of symmetric fixed points in $\Xe$ converges,
and its value is given by
\begin{equation}\label{eq:Density}
 \lim_{\lambda\rightarrow 0} \delta(e,\lambda) = \frac{1}{(2\lfloor \sqrt{e} \rfloor+1)(2\lfloor \sqrt{e/2} \rfloor+1)}.
\end{equation}
\end{maintheorem}

As in theorem A, the smallness of $\lambda$ serves only to ensure that $\Xe$ 
is sufficiently populated for all congruence classes modulo $\Le$ to be 
represented. 

The condition (\ref{eq:v1_k_coprimality}) on the orbit code is clearly satisfied 
for infinitely many critical numbers $e$, e.g., those for which 
either $2\lfloor\sqrt{e}\rfloor+1$ or
$2\lfloor\sqrt{e/2}\rfloor+1$ is a prime number. 
The first violation occurs at $e=49$ (see Figure \ref{fig:V_table}), where $2v_1+1 = 9$ 
and $2v_k+1=15$ have a common factor.
We contrast this to the case of $e=52$, where $2v_k+1=15$ and $2v_2+1 = 9$ have a 
common factor, but $2v_1+1=11$ is prime, so the condition (\ref{eq:v1_k_coprimality}) holds.
Numerical experiments show that the density
of values of $e$ for which (\ref{eq:v1_k_coprimality}) holds decays very slowly,
reaching 1/2 for $e\approx 500,000$.

%The first violation occurs at $e=49$. Numerical experiments show that the density 
%of values of $e$ for which (\ref{eq:v1_k_coprimality}) holds decays very slowly, 
%reaching 1/2 for $e\approx 500,000$.

The stated condition on the orbit code is actually stronger than that 
we require in the proof. This was done to simplify the formulation of the
theorem. We remark that the weaker condition is still not necessary for 
the validity of the density expression (\ref{eq:Density}). 
At the same time, there are values of $e$ for which the density of symmetric minimal 
orbits deviates from the given formula, and convergence is not guaranteed. 
Our numerical experiments show that these deviations are small, and don't seem 
connected to new dynamical phenomena. 
More significant are the fluctuations in the density of non-symmetric orbits.
Its dependence on $e$ is considerably less predictable than for symmetric orbits, 
see figure \ref{fig:Density}.

%%%%%%%%%%%%%%%%%%%%%%%%%%%%%%%%%%%%%%%%%%%%%%%%%%%%%%%% FIGURE
\begin{figure}[!h]
        \centering
\epsfig{file=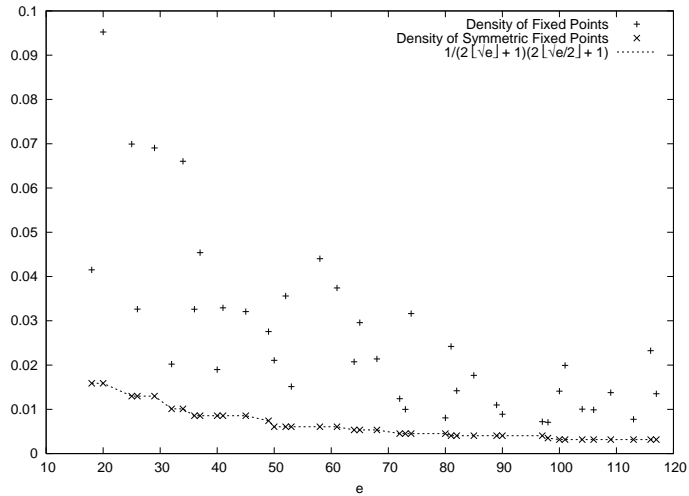,scale=0.75}
        \caption{The density of symmetric minimal orbits, as a function 
of the critical number $e$ (calculated for suitably small values of the parameter $\lambda$). 
The solid line represents the estimate (\ref{eq:Density}).
The scattered points correspond to the density of all minimal orbits, symmetric
and non-symmetric.
}
        \label{fig:Density}
\end{figure}
%%%%%%%%%%%%%%%%%%%%%%%%%%%%%%%%%%%%%%%%%%%%%%%%%%%%%%%%

The asymptotic density of symmetric fixed points in $\Xe$ provides an obvious lower bound 
for the overall density of fixed points, which we denote $\eta(e,\lambda)$:
\begin{displaymath}
 \eta(e,\lambda) = \frac{\# \{ z\in \Xe : \Phi(z)=z \}}{\# \Xe}.
\end{displaymath}

\begin{corollary} \label{cor:Density}
Let $e\in\cE$ satisfy the condition (\ref{eq:v1_k_coprimality}) of theorem \ref{thm:minimal_densities}.
Then the asymptotic density of fixed points in $\Xe$ is bounded below as follows:
\begin{equation}\label{eq:Density_II}
 \liminf_{\lambda\rightarrow 0} \eta(e,\lambda) = \frac{1}{(2\lfloor \sqrt{e} \rfloor+1)(2\lfloor \sqrt{e/2} \rfloor+1)}.
\end{equation}
\end{corollary}

Note that we do not suggest that the density $\eta(e,\lambda)$ converges as $\lambda\rightarrow0$, 
regardless of whether the condition (\ref{eq:v1_k_coprimality}) is satisfied or not.

%%%%%%%%%%%%%%%%%%%%%%%%%%%%%%%%%%%%%%%%%%%%%%%%%%%%%%%%%%%%%%%%%%%%%%%%%%
\section{The strip map $\Psi$}\label{sec:StripMap}
%%%%%%%%%%%%%%%%%%%%%%%%%%%%%%%%%%%%%%%%%%%%%%%%%%%%%%%%%%%%%%%%%%%%%%%%%%

In section \ref{sec:Recurrence}, we saw that all non-zero points $z\in(\lambda\Z)^2$ 
where the discrete vector field $\mathbf{v}(z)$ deviates from the Hamiltonian 
vector field $\mathbf{w}(z)$ lie in the set of transition points $\Lambda$, 
defined in (\ref{eq:Lambda}). 
In order to study the dynamics at these points, where the perturbations from the 
integrable limit occur, we define a transit map $\Psi$ to $\Lambda$ which we 
call the \textbf{strip map}:
\begin{displaymath}
  \Psi : (\lambda\Z)^2 \rightarrow \Lambda
 \hskip 40pt
  \Psi(z) = F_{\lambda}^{4t(z)}(z),
\end{displaymath}
where the transit time $t$ to $\Lambda$ is well-defined for all points excluding the origin:
\begin{displaymath}
 t(z)= \min \{ i\in\N : F_{\lambda}^{4i}(z) \in \Lambda \} \hskip 40pt z\neq(0,0).
\end{displaymath}
(Since the origin plays no role in the present construction, to simplify 
notation we shall write $(\lambda\Z)^2$ for $(\lambda\Z)^2\setminus\{(0,0)\}$ 
in the rest of the paper, where appropriate.) 
By abuse of notation, we define $\Psi^{-1}$ to be the transit map to $\Lambda$ 
under $F_\lambda^{-1}$. Note that $\Psi^{-1}$ is the inverse of $\Psi$ only on $\Lambda$.

If $z\in B_{m,n}\setminus\Lambda$ for some $m,n\in\Z$, then lemma 
\ref{thm:Lambda} (page \pageref{thm:Lambda}) implies that $\Psi(z)$ 
satisfies \begin{equation} \label{eq:Psi_congruence2}
 \Psi(z) = z + t(z)\mathbf{w}_{m,n},
\end{equation}
where $\mathbf{w}_{m,n}$ is the value of the Hamiltonian vector field $\mathbf{w}$ in the box $B_{m,n}$. 
If $z\in\Lambda_{m,n}$, then we may have $\mathbf{v}(z)\neq \mathbf{w}(z)$, so the expression becomes
\begin{equation} \label{eq:Psi_congruence1}
 \Psi(z) = z + \mathbf{v}(z) + (t(z)-1)\mathbf{w}_{m,n}. 
\end{equation}

In the previous section, we identified the set ${\mathcal O}_{\tau}(z)\,\cap\,\Lambda$ as the set of 
vertices of the perturbed orbit ${\mathcal O}_{\tau}(z)$. Thus, within each quarter-turn, the 
strip map $\Psi$ represents transit to the next vertex. For $1\leq j\leq k$, where $k$ is the 
length of the vertex list at $z$, we say that the orbit ${\mathcal O}_{\tau}(z)$ \textbf{meets the $j$th vertex} 
at the point $\Psi^j(z)\in\Lambda$. 
For $z\in X$ regular, the polygon $\Pi(z)$ and the return orbit ${\mathcal O}_{\tau}(z)$ are non-critical,
by construction, and the number of sides of each is given by equation (\ref{eq:NumberOfSides}). 
Thus the full set of vertices of ${\mathcal O}_{\tau}(z)$ is given by
 \begin{displaymath}
 {\mathcal O}_{\tau}(z)\,\cap\,\Lambda = \bigcup_{i=0}^3 \; \bigcup_{j=1}^{2k-1} \; \{ (\Psi^j\circ F_{\lambda}^i)(z) \}.
\end{displaymath}
Recall that the vertices of a polygon (or orbit) are numbered in the clockwise 
direction ---the orientation of the integrable vector field $\mathbf{w}$. 
Hence the first $2k-1$ vertices (those lying in the first quarter-turn) are 
given by $(\Psi^j(z)),\,{1\leq j \leq 2k-1}$. 
The action of $F_{\lambda}$ moves points from one quadrant to the next in the 
opposing (anti-clockwise) direction, so that the vertices 
$((\Psi^j\circ F)(z)),\,{1\leq j \leq 2k-1}$ are the last $2k-1$ vertices. 
Thus the following proposition is a simple consequence of the number of 
vertices of a given polygon class.

\begin{proposition} \label{thm:regularity} 
Let $e$ be a critical number, and let $k$ be the length of the
vertex list of the corresponding polygon class.
Then the return map $\Phi$ on $\Xe$ is related to $\Psi$ via
\begin{equation} \label{eq:Phi_Psi}
\Phi(z) \equiv (\Psi^{2k}\circ F_{\lambda})(z) \mod{\mathbf{w}_{v_1,v_1}}, 
\hskip 40pt z\in\Xe
\end{equation}
where $v_1$ is the type of the first vertex and $\mathbf{w}_{v_1,v_1}$ 
is the value of the integrable vector field $\mathbf{w}$ at $z$.
\end{proposition}

\begin{proof} Let $z\in \Xe$.  By the preceding discussion,
the last vertex in the orbit ${\mathcal O}_{\tau}(z)$ is given by
\begin{displaymath}
 (\Psi^{2k-1}\circ F_{\lambda})(z) \in \Lambda_{v_1,v_1}.
\end{displaymath}
The point $\Phi(z)$ satisfies 
$\Phi(z)\in B_{v_1,v_1}\setminus \Lambda$.
Using the expression (\ref{eq:Psi_congruence2}) for $\Psi$ applied to $\Phi(z)$, we have
\begin{displaymath}
 (\Psi^{2k}\circ F_{\lambda})(z) \equiv \Phi(z) \mod{\mathbf{w}_{v_1,v_1}}
\end{displaymath}
as required.
\end{proof}

For $z\in X$ regular, we use the vertices 
$(\Psi^j(z))_{1\leq j \leq 2k-1}$ in the first quarter-turn to define a 
sequence of natural numbers $\sigma(z)$ called the \textbf{orbit code} of $z$, 
which encapsulates how the asymptotic orbit ${\mathcal O}_{\tau}(z)$ 
deviates from $\Pi(z)$.

Suppose the $j$th vertex of $\Pi(z)$ is a vertex of type $v_j$ lying on $y=n$,
and the orbit ${\mathcal O}_{\tau}(z)$ meets its corresponding vertex at $\Psi^j(z)$.
We define the pair $(x_j,y_j)$ via
\begin{equation}\label{eq:(xj,yj)}
 \Psi^j(z) = \lambda\left(\left\lceil\frac{v_j}{\lambda}\right\rceil + x_j,
 \left\lceil\frac{n}{\lambda}\right\rceil + y_j\right),
\end{equation}
where $x_j\geq 0$, and $|y_j|$, which is (essentially) the number of lattice points 
between $\Psi^j(z)$ and the line $y=n$, is small relative to $1/\lambda$. 
Using similar arguments to those in the proof of proposition 
\ref{prop:Xe_tilde} one can show that $y_j$ satisfies
\begin{displaymath}
 -(2v_j+1) \leq y_j <0 \hskip 20pt \mbox{or} \hskip 20pt 0\leq y_j < 2v_j+1,
\end{displaymath}
depending whether the integrable vector field is oriented in the positive or 
negative $y$-direction.
In both cases, the possible values of $y_j$ form a complete set of residues 
modulo $2v_j+1$. Hence the $j$th element $\sigma_j$ of the orbit 
code $\sigma(z)$ is defined to be the unique residue 
$\sigma_j\in\{0,1,\dots, 2v_j\}$ which is congruent to $y_j$:
\begin{equation} \label{eq:sigma_j}
 \sigma_j \equiv y_j \mod{2v_j+1}.
\end{equation}
We call $y$ the \textbf{integer coordinate} of the vertex and $x$ the 
\textbf{non-integer coordinate}.
Similarly, if the $j$th vertex lies on $x=m$, then the $j$th element $\sigma_j$ of 
the orbit code is defined to be the residue congruent to $x_j$ modulo $2v_j+1$. 
In this case $x$ is the integer coordinate and $y$ is the non-integer coordinate.

For all vertices in the first quadrant, the fact that orbits progress clockwise 
under the action of $F^4_{\lambda}$ means that $y_j$ will be non-negative 
wherever $y$ is the integer coordinate, and $x_j$ will be negative 
wherever $x$ is the integer coordinate:
\begin{equation} \label{eq:xj<0,yj>0}
 -(2v_j+1) \leq x_j <0 \hskip 20pt \mbox{or} \hskip 20pt 0\leq y_j < 2v_j+1. 
\end{equation}
Thus the value of $\sigma_j$ is given explicitly by
\begin{equation} \label{eq:sigma_j_2}
 \sigma_j = x_j + 2v_j+1 \hskip 20pt \mbox{or} \hskip 20pt \sigma_j = y_j, 
\end{equation}
respectively.

In addition to the values $\sigma_{j}$ for $1\leq j \leq 2k-1$ we consider $\sigma_{-1}$, 
which corresponds to the last vertex \textit{before} the symmetry line, i.e., to the point $\Psi^{-1}(z)$. 
Thus the orbit code of $z$ is a sequence $\sigma(z)=(\sigma_{-1},\sigma_1,\dots,\sigma_{2k-1})$, such that
\begin{align*}
 & 0\leq \sigma_{-1} < 2v_1+1, \\
 & 0\leq \sigma_j < 2v_j+1, \hskip 20pt 1\leq j \leq 2k-1,
\end{align*}
where the $v_j$ are the vertex types.

%%%%%%%%%%%%%%%%%%%%%%%%%%%%%%%%%%%%%%%%%%%%%%%%%%%%%%%%%%%%%%%%%%%%%%%%%%
In the next proposition we consider how a perturbed orbit behaves at
its vertices. We find that 
the regularity of $z$ ensures that the discrete vector field $\mathbf{v}$ 
matches the Hamiltonian vector field $\mathbf{w}$ in the integer coordinate 
at $\Psi^j(z)$. The possible discrepancy in the non-integer coordinate is 
determined by the value of $\sigma_j$.
\begin{proposition} \label{thm:epsilon_j}
Let $e$ be a critical number and let $k$ be the length of the
vertex list of the corresponding polygon class. 
For any $z\in \Xe$ and any $j\in\{-1,1,2,\ldots,2k-1\}$, 
let $m,n$ be such that $\Psi^j(z)\in\Lambda_{m,n}$.
Then the discrete vector field at the $j$th vertex of 
${\mathcal O}_{\tau}(z)$ is given by
\begin{displaymath}
\mathbf{v}(\Psi^j(z)) = \mathbf{w}_{m,n} + \lambda \epsilon_j(\sigma_j) \mathbf{e}, 
\end{displaymath}
where $\epsilon_j$ is a function of the $j$th entry $\sigma_j$ of the orbit code $\sigma(z)$ 
and $\mathbf{e}$ is the unit vector in the direction of the non-integer coordinate of the vertex.
\end{proposition}

\begin{proof}
If $z\in X$ is regular, then by proposition \ref{prop:Xe_tilde} the perturbed orbit 
${\mathcal O}(z)$ is not critical. 
Thus for any vertex $w$, which, by construction, satisfies
\begin{displaymath}
 w\in\Lambda_{m,n}
 \qquad
 F_{\lambda}^4(w) = w+\mathbf{v}(w)\in B_{m,n} 
\end{displaymath}
for some $m,n\in\Z$, we must have either $w\in B_{m,n\pm1}$ or $w\in B_{m\pm1,n}$. 
For definiteness we suppose that $w\in B_{m,n+1}$, so that the vertex $w$ lies on $y=n+1$. 
The cases where $w\in B_{m,n-1}$ or $w\in B_{m\pm1, n}$ are similar.

Now the proof proceeds very much as that of proposition \ref{prop:Xe_tilde}. 
The perturbed vector field $\mathbf{v}(w)$, with $R(w)=(m,n+1)$,  
is given by equation (\ref{eq:vw}), with $a,b,c,d$ as in (\ref{eq:abcd}). 
In this case, $R(F_{\lambda}^4(w))=(m,n)$ implies that
\begin{displaymath}
 c= n \hskip 20pt \mbox{and} \hskip 20pt d =m,
\end{displaymath}
and according to (\ref{eq:bd_ac_sets}), the remaining integers $a$ and $b$ satisfy
\begin{displaymath}
 a\in \{n,n+1\}, \hskip 40pt b=m.
\end{displaymath}
Thus we have
\begin{align}
 \mathbf{v}(w) &= \lambda(n+a+1,-(2m+1)) \label{eq:v(Psi^j)} \\
 &= \mathbf{w}_{m,n} + \lambda(a-n)\mathbf{e}, \nonumber
\end{align}
where $\mathbf{e}=(1,0)$ is the unit vector in the $x$-direction, the non-integer 
coordinate direction of the vertex.

If $w=\Psi^j(z)$, then the coefficient of the difference between 
$\mathbf{v}(w)$ and $\mathbf{w}_{m,n}$ in the $x$-direction is given by
\begin{align*}
 \epsilon_j &= a-n \\
&= \lceil \lambda(y-m)\rceil -(n+1) \\
&= \left\{ \begin{array}{ll} 1 \; \; & \lambda y-(n+1) >\lambda m \\
            0 \; \; & \mbox{otherwise}.
           \end{array} \right.
\end{align*}
As in equation (\ref{eq:(xj,yj)}), we write
\begin{displaymath}
 y = \left\lceil \frac{n+1}{\lambda} \right\rceil + y_j,
\end{displaymath}
where $y_j$ satisfies the second inequality in (\ref{eq:xj<0,yj>0}).
Hence, by (\ref{eq:sigma_j_2}), we have $y_j=\sigma_j$, and 
the function $\epsilon_j=\epsilon_j(\sigma_j)$ is given by
\begin{displaymath}
 \epsilon_j = \left\{ \begin{array}{ll} 1 \; \; & \sigma_j > m \\
            0 \; \; & \mbox{otherwise}
           \end{array} \right.
\end{displaymath}
which completes the proof.
\end{proof}

Note that the function $\epsilon_j$ depends on $j$ via $m$.
In what follows we shall write $\epsilon_j$, omitting the argument.

Applying proposition \ref{thm:epsilon_j} to equation (\ref{eq:Psi_congruence1}), 
we have that if $\Psi^j(z)\in\Lambda_{m,n}$, then the transit between vertices satisfies
\begin{equation}\label{eq:Psi^jp1}
 \Psi^{j+1}(z) = \Psi^j(z) + \lambda \epsilon_j 
\mathbf{e} + t\mathbf{w}_{m,n}, 
\end{equation}
where $t=t(\Psi^j(z))$ is the transit time. 
Hence we think of an orbit as moving according to the integrable vector 
field at all points except the vertices, where there is a mismatch between
integrable and non-integrable dynamics, and points are given a small 
`kick' in the non-integer coordinate direction.

%%%%%%%%%%%%%%%%%%%%%%%%%%%%%%%%%%%%%%%%%%%%%%%%%%%%%%%%%%%%%%%%%%%%%%%%%%
\section{Lattice structure and proof of main theorems} \label{sec:lattice}
%%%%%%%%%%%%%%%%%%%%%%%%%%%%%%%%%%%%%%%%%%%%%%%%%%%%%%%%%%%%%%%%%%%%%%%%%%

In this section we prove theorems 
A %\ref{thm:Phi_equivariance} 
and B, %\ref{thm:minimal_densities}, 
stated in section \ref{sec:MainTheorems}.

For $e\in\cE$, suppose the vertex list $V(e)=(v_1,\dots,v_k)$ 
contains $l$ distinct entries. We define the sequence 
$(\iota(j))_{1\leq j\leq l}$ such that the $\iota(j)$th entry in the vertex 
list is the $j$th distinct entry. Since all repeated entries are consecutive, 
it follows that the vertex list has the form
\begin{equation} \label{eq:v_iota}
 V(e) = (v_{\iota(1)},\dots,v_{\iota(1)},v_{\iota(2)},\dots, v_{\iota(2)}, \, 
     \dots \, ,v_{\iota(l)},\dots, v_{\iota(l)}), 
\end{equation}
with $v_{\iota(1)} = v_1$ and $v_{\iota(l)} = v_k$.
% so the list
%\begin{equation}
% (v_{\iota(1)},\dots,v_{\iota(l)}) \label{eq:v_iota}
%\end{equation}
%is obtained from the vertex list by simply removing repeated entries. 
We define the vector $\mathbf{L}=\mathbf{L}(e,\lambda)$ as:
\begin{equation}
 \mathbf{L} = \frac{\lambda q}{2v_1+1} \; (1,1),
\end{equation}
where the natural number $q=q(e)$ is defined as follows
\begin{equation}\label{eq:q}
 q = \mbox{\rm lcm}((2v_{\iota(1)}+1)^2,(2v_{\iota(1)}+1)(2v_{\iota(2)}+1),
      \ldots ,(2v_{\iota(l-1)}+1)(2v_{\iota(l)}+1)).
\end{equation}
Here the least common multiple runs over $(2v_1+1)^2$ and all products 
of the form $(2v_j+1)(2v_{j+1}+1)$, where $v_j$ and 
$v_{j+1}$ are consecutive, distinct vertex types. 
Finally, the lattice $\Le=\Le(\lambda)\subset(\lambda\Z)^2$ of 
theorem A %\ref{thm:Phi_equivariance} 
is given by
\begin{equation}\label{eq:Le}
 \Le = \left\langle \mathbf{L}, \frac{1}{2}\left(\mathbf{L}-\mathbf{w}_{v_1,v_1}\right) \right\rangle,
\end{equation}
where $\langle \cdots \rangle$ denotes the $\Z$-module generated by a set 
of vectors, and the vector $\mathbf{w}_{v_1,v_1}$ given by (\ref{eq:w_mn}) 
is the Hamiltonian vector field $\mathbf{w}$ in the domain $\Xe$ (figure \ref{fig:lattice_Le}).

%%%%%%%%%%%%%%%%%%%%%%%%%%%%%%%%%%%%%%%%%%%%%%%%%%%%%%%% FIGURE
\begin{figure}[ht]
	\centering
        \vspace*{-80pt}
        \includegraphics[scale=0.7]{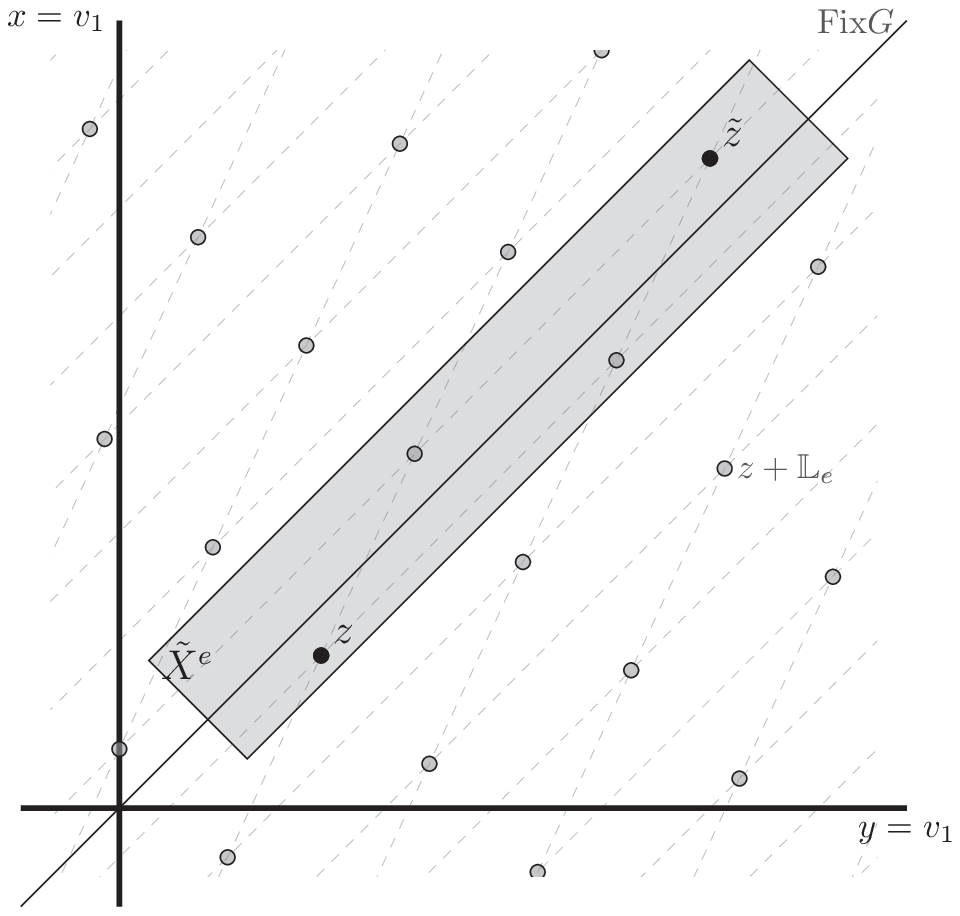}
        \vspace*{-280pt}
	\caption{The lattice $\Le$.}
	\label{fig:lattice_Le}
\end{figure}
%%%%%%%%%%%%%%%%%%%%%%%%%%%%%%%%%%%%%%%%%%%%%%%%%%%%%%%% 

Any $z,\tZ \in \Xe$ which are congruent modulo $\Le$ are related by:
\begin{displaymath}
 \tZ = z + \frac{1}{2}\left( (2a+b) \mathbf{L} -b\mathbf{w}_{v_1,v_1}\right),
\end{displaymath}
where $a,b\in\Z$ are the coordinates of $\tZ-z$ relative to the module basis. 
We note that the vector $\mathbf{L}$ is parallel to the symmetry line $\Fix{G}$, and 
hence parallel to the strip $\Xe$, whereas the vector $\mathbf{w}_{v_1,v_1}$ is perpendicular to it. 
It follows that both $a$ and $b$ are determined uniquely by the value of the coefficient $2a+b$, 
because if $z=\lambda(x,y)$, then
\begin{align*}
 x\geq y \; \; &\Rightarrow \; \; b\in\{0,1\}, \\
 x<y \; \; &\Rightarrow \; \; b\in\{-1,0\}.
\end{align*}
The point $z$ itself corresponds to $a=b=0$.

We prove theorems A %\ref{thm:Phi_equivariance} 
\& B %\ref{thm:minimal_densities} 
via several lemmata. 
The first and most significant step is to show that the orbit codes $\sigma(z)$ of points $z\in \Xe$ are in 
one-to-one correspondence with the equivalence classes of $\Xe$ modulo $\Le$. 
We do this by constructing a sequence of nested lattices whose congruence classes are the 
cylinder sets of the orbit code.

We define recursively a finite integer sequence $(q_j)$, 
$j=1,\ldots,2k-1$, as follows:
\begin{align}
   q_1&=(2v_1+1)^2\nonumber \\
   q_j&=\begin{cases}
         q_{j-1}&\mbox{if}\,\, v_j=v_{j-1}\\
         \mbox{\rm lcm}((2v_j+1)(2v_{j-1}+1),q_{j-1}) &\mbox{if}\,\,\, v_j\neq v_{j-1}
        \end{cases}
      &j>1. \label{eq:q_j}
\end{align}
Then we let
\begin{equation}\label{eq:p_j}
p_j=q_j/(2v_j+1)\qquad j=1,\ldots,2k-1.
\end{equation}
By construction, $p_j$ is also an integer.
After defining the associated sequence of vectors 
\begin{equation} \label{eq:L_j}
 \mathbf{L}_j = \frac{\lambda q_j}{2v_1+1} \; (1,1), 
\end{equation}
we let the lattices $\Le_j$ be the $\Z$-modules with basis
\begin{equation} \label{eq:lattice_j}
 \Le_j =\left\langle \mathbf{L}_j, 
   \frac{1}{2}\left(\mathbf{L}_j -\mathbf{w}_{v_1,v_1}\right) \right\rangle. 
\end{equation}
By construction
\begin{displaymath}
 \Le_{2k-1} \subseteq \Le_{2k-2} \subseteq \dots \subseteq \Le_1 \subset (\lambda\Z)^2.
\end{displaymath}

We claim that for all $1\leq j \leq 2k-1$, the closed form expression for $q_j$ is given by
\begin{equation}
 q_j = \mbox{\rm lcm}((2v_{\iota(1)}+1)^2,(2v_{\iota(1)}+1)(2v_{\iota(2)}+1),
  \dots,(2v_{\iota(i-1)}+1)(2v_{\iota(i)}+1)), \label{eq:q_j_closed_form}
\end{equation}
where $i$ is the number of distinct entries in the list $(v_1,v_2,\dots,v_j)$. 
That the lowest common multiple (\ref{eq:q_j_closed_form}) runs over all products 
$(2v_j+1)(2v_{j+1}+1)$ of consecutive, distinct vertex types follows from the form 
(\ref{eq:v_iota}) of the vertex list and the symmetry (\ref{eq:v_symmetry}) of the vertex types. 
Furthermore, since all distinct vertex types occur within the first $k$ vertex types, 
the expression (\ref{eq:q_j_closed_form}) implies that the sequence $(q_j)$ is eventually 
stationary:
\begin{equation}
 q_j = q, \hskip 20pt \Le_j = \Le, \hskip 20pt k\leq j \leq 2k-1, \label{eq:q_j=q}
\end{equation}
where $q$ and $\Le$ are given by equations (\ref{eq:q}) and (\ref{eq:Le}). 

For given $e$, the following result details the role of the $\Le_j$ as cylinder sets of the orbit code. 
Applying the result for $j=2k-1$, along with the observation (\ref{eq:q_j=q}), implies that two points 
share the same orbit code if and only if they are congruent modulo $\Le$.

\begin{lemma} \label{thm:sigma_lattices}
Let $e$ be a critical number, let $k$ be the length of the vertex list of the 
corresponding polygon class, and let $p_j$ and $\Le_j$ be as above.
For any $1\leq j \leq 2k-1$ and all $z,\tZ\in \Xe$, the following three statements are equivalent:
\begin{enumerate}[(i)]
 \item the orbit codes of $z$ and $\tZ$ match up to the $j$th entry,
 \item $z$ and $\tZ$ are congruent modulo $\Le_j$,
 \item the points $\Psi^j(z)$ and $\Psi^j(\tZ)$ are congruent modulo 
$\lambda p_j\mathbf{e}$, where $\mathbf{e}$ is the unit vector in the 
direction of the non-integer coordinate of the $j$th vertex.
\end{enumerate}
\end{lemma}

\begin{proof} Let $e\in\mathcal{E}$, let $\Le_j$ be as above, and let 
$z,\tilde z\in \Xe$ with orbit codes 
$(\sigma_{-1},\sigma_1,\dots,\sigma_{2k-1})$, and
$(\tilde\sigma_{-1},\tilde\sigma_1,\dots,\tilde\sigma_{2k-1})$, respectively. 
We proceed by induction on $j$, with two induction hypotheses. 
Firstly we suppose that $(i)$ is equivalent to $(ii)$, so that for any 
$1\leq j\leq 2k-1$: 
\begin{equation}
(\sigma_{-1},\sigma_1,\dots,\sigma_j)=
(\tilde\sigma_{-1},\tilde\sigma_1,\dots,\tilde\sigma_j)
\quad\Leftrightarrow\quad
 \tZ \equiv z \mod{\Le_j}. \tag{H1}\label{eq:H1}
\end{equation}
Thus any such $\tZ$ is related to $z$ via
\begin{equation}
 \tZ = z + \frac{1}{2}\left( (2a+b) \mathbf{L}_j -b\mathbf{w}_{v_1,v_1}\right), \label{eq:z_tilde}
\end{equation}
for $a\in\Z$ and $b\in\{0,1\}$ or $b\in\{0,-1\}$, as appropriate. 
Secondly, we suppose that $(ii)$ is equivalent to $(iii)$. 
In particular,
\begin{equation} \tag{H2}\label{eq:H2}
\tZ = z + \frac{1}{2}\left( (2a+b) \mathbf{L}_j -b\mathbf{w}_{v_1,v_1}\right)
\quad\Leftrightarrow\quad
\Psi^j(\tZ) = \Psi^j(z) + \lambda(2a +b) p_j \mathbf{e}, 
\end{equation}
where $\mathbf{e}$ is the unit vector in the direction of the non-integer coordinate of that vertex.

We begin with the base case $j=1$. 
Suppose that the first vertex of a polygon in class $e$ lies on $y=v_1$, 
so that $y$ is its integer coordinate (if $x$ is the integer coordinate, 
then the analysis is identical). 
By symmetry, the previous vertex lies on $x=v_1$ and its integer coordinate is $x$.
Using the properties of $\Psi$ given in equations (\ref{eq:Psi_congruence2}) and 
(\ref{eq:Psi_congruence1}), applied to $z\in B_{v_1,v_1}\setminus\Lambda$ and 
$\Psi^{-1}(z)\in\Lambda_{v_1,v_1}$ respectively, we have
\begin{align}
  \Psi(z) &\equiv z \mod{\mathbf{w}_{v_1,v_1}}, \label{eq:Psi_congruence4} \\
  \Psi^{-1}(z) + \mathbf{v}(\Psi^{-1}(z)) &\equiv  z \mod{\mathbf{w}_{v_1,v_1}}. \label{eq:Psi_congruence3}
\end{align}
Furthermore by proposition \ref{thm:epsilon_j}:
\begin{align*}
 \mathbf{v}(\Psi^{-1}(z)) = \mathbf{w}_{v_1,v_1} +\lambda\epsilon_{-1}\mathbf{e},
\end{align*}
where $\mathbf{e}=(0,1)$ is the non-integer coordinate vector for the $(-1)$th vertex. 
Thus if $z=\lambda(\lceil v_1/\lambda\rceil+x,\lceil v_1/\lambda\rceil+y)$, 
by the definition (\ref{eq:sigma_j}) of the orbit code,
the $x$- and $y$-components of equations (\ref{eq:Psi_congruence3}) and (\ref{eq:Psi_congruence4}), 
respectively, give us that the first two entries in the orbit code $\sigma(z)$ satisfy
\begin{align*}
 x \equiv \sigma_{-1} \mod{2v_1+1}, \\
 y \equiv \sigma_1 \mod{2v_1+1}.
\end{align*}
It follows that $z,\tZ\in \Xe$ share the partial code $(\sigma_{-1},\sigma_1)$ if and only if
\begin{displaymath}
 \tZ \equiv z \mod{(\lambda(2v_1+1)\Z)^2}.
\end{displaymath}
The lattice $\Le_1$ is given by (cf.~(\ref{eq:lattice_j}))
\begin{displaymath}
\Le_1 = \left\langle \mathbf{L}_1, \frac{1}{2} \left(\mathbf{L}_1 -\mathbf{w}_{v_1,v_1}\right) \right\rangle, 
\end{displaymath}
where $\mathbf{L}_1=\lambda p_1(1,1)$, $p_1=2v_1+1$ and
\begin{displaymath}
 \frac{1}{2} \left(\mathbf{L}_1 -\mathbf{w}_{v_1,v_1}\right) = \frac{\lambda}{2} (p_1-p_1,p_1+p_1) = \lambda p_1\mathbf{e}.
\end{displaymath}
Thus $\Le_1=(\lambda(2v_1+1)\Z)^2$ and the first hypothesis holds.

Now let $z,\tZ\in \Xe$ satisfy (\ref{eq:z_tilde}) with $j=1$.
If $\Psi(z) = F_{\lambda}^{4t}(z) = z +t \mathbf{w}_{v_1,v_1}$, where $t\in\N$ is the 
transit time to $\Lambda$, then the identities
\begin{align*}
\tZ + (t+ a +b)\mathbf{w}_{v_1,v_1}
&= \Psi(z) + \frac{1}{2} (2a+b) \left(\mathbf{L}_1 +\mathbf{w}_{v_1,v_1}\right)  \\
&= \Psi(z) + \lambda(2a +b)p_1 \mathbf{e}
\end{align*}
with $\mathbf{e}=(1,0)$ show that $\tZ$ has transit time $t+a+b$ ,
and therefore $\Psi(\tZ) = \Psi(z) + \lambda(2a +b)p_1 \mathbf{e}$, as required
(see figure \ref{fig:Psi_diagram}).
This completes the basis for induction.

%%%%%%%%%%%%%%%%%%%%%%%%%%%%%%%%%%%%%%%%%%%%%%%%%%%%%%%% FIGURE
\begin{figure}[ht]
        \centering
        \vspace*{-80pt}
        \includegraphics[scale=0.7]{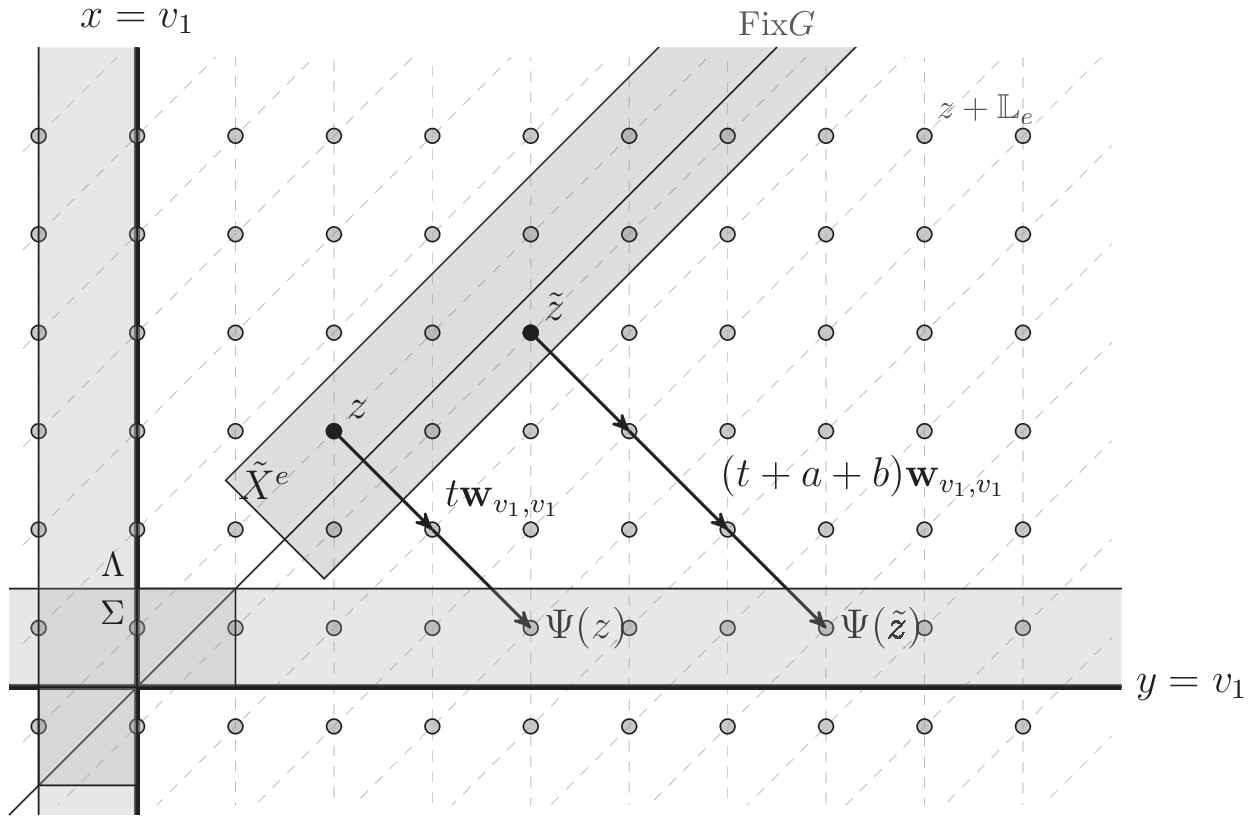}
        \vspace*{-300pt}
        \caption{The points $z$, $\tZ$, $\Psi(z)$ and $\Psi(\tZ)$. }
        \label{fig:Psi_diagram}
\end{figure}

Now we proceed with the inductive step. 
Let (\ref{eq:H1}) and (\ref{eq:H2}) hold for some $j\geq 1$. 
Then $z$ and $\tZ$ are related as in equation (\ref{eq:z_tilde}), for some $a,b$. 
We think of $\tilde{\sigma}_{j+1}$, the $(j+1)$th entry of the orbit code of $\tZ$,
as a function of $(a,b)$. We suppose that the $j$th vertex lies 
on $y=n$ for some $n\in\Z$ (again the case in which the vertex lies on $x=m$ is identical). 
Let the pair $(x_j,y_j)$ be defined from $\Psi^j(z)$ via equation (\ref{eq:(xj,yj)}). 
Similarly, $\Psi^j(\tZ)$ defines the pair $(\tilde{x}_j,\tilde{y}_j)$.

By (\ref{eq:H2}), $\Psi^j(\tZ)$ satisfies
\begin{displaymath}
 \Psi^j(\tZ) = \Psi^j(z) + \lambda(2a +b) p_j \mathbf{x},
\end{displaymath}
where $\mathbf{x}=\mathbf{e}=(1,0)$. 
Combining this expression with equation (\ref{eq:Psi^jp1}), 
applied to $\Psi^j(\tZ)\in\Lambda_{v_j,n-1}$, we obtain
\begin{align}
 \Psi^{j+1}(\tZ) &= \Psi^j(\tZ) + \lambda\epsilon_j \mathbf{x} + \tilde t \mathbf{w}_{v_j,n-1} \nonumber \\
 &= \Psi^j(z) + \lambda(2a +b) p_j\mathbf{x} + \lambda\epsilon_j \mathbf{x} + \tilde t  \mathbf{w}_{v_j,n-1}. 
 \label{eq:Psi^jp1tilde}
\end{align}
where $\tilde t$ is the transit time of $\tilde z$ to $\Lambda$.

There are now two cases to consider. \\

\noindent
\textit{Case 1: $v_j=v_{j+1}$.}

In this case the $j$th and $(j+1)$th vertices lie on parallel lines, 
which we take to be $y=n$ and $y=n-1$, so $\Psi^{j+1}(z)$ is given by
\begin{align*}
 \Psi^{j+1}(z) = \lambda\left(\left\lceil\frac{v_j}{\lambda}\right\rceil + x_{j+1},
       \left\lceil\frac{n-1}{\lambda}\right\rceil + y_{j+1}\right),
\end{align*}
and similarly for $\Psi^{j+1}(\tZ)$. 
According to the definitions (\ref{eq:p_j}) and (\ref{eq:lattice_j}), 
we have $p_j=p_{j+1}$ and $\Le_j=\Le_{j+1}$. 
Thus, to show that (\ref{eq:H1}) continues to hold, we need to show that $\tilde{\sigma}_{j+1}=\sigma_{j+1}$ for all $(a,b)$. 
Similarly we need to show that the vector $\Psi^{j+1}(\tZ)-\Psi^{j+1}(z)$ 
is equal to the vector $\Psi^{j}(\tZ)-\Psi^{j}(z)$ of hypothesis (\ref{eq:H2}).

Because $y$ is the integer coordinate of both $j$th and $(j+1)$th vertices,
the transit time is the same for the orbits of $z$ and $\tilde z$. 
Therefore equations (\ref{eq:Psi^jp1}) 
and (\ref{eq:Psi^jp1tilde}) with $\tilde t =t$ give us that
\begin{align*}
 \Psi^{j+1}(\tZ) &= \Psi^j(z) + \lambda(2a +b) p_j\mathbf{x} 
     + \lambda\epsilon_j \mathbf{x} + t  \mathbf{w}_{v_j,n-1}, \\
 &= \Psi^{j+1}(z) + \lambda(2a +b) p_j\mathbf{x},
\end{align*}
and the second hypothesis (\ref{eq:H2}) remains satisfied. 
Furthermore, $\Psi^{j+1}(\tZ)$ and $\Psi^{j+1}(z)$ have the same integer 
($y$) coordinate. It follows that, by the definition (\ref{eq:sigma_j}) 
of the orbit code, $\tilde \sigma_{j+1}=\sigma_{j+1}$ and (\ref{eq:H1}) 
is also satisfied.

By the $y$-component of (\ref{eq:Psi^jp1tilde}), the value of ${\sigma}_{j+1}$ 
is determined explicitly by the congruences
\begin{equation} \label{eq:sigma_j+1_case1}
 \left\lceil\frac{n-1}{\lambda}\right\rceil + {\sigma}_{j+1} \equiv
     \left\lceil\frac{n}{\lambda}\right\rceil + \sigma_j \mod{2v_j+1}.  
\end{equation}
Equation (\ref{eq:sigma_j+1_case1}) shows that, if $v_j=v_{j+1}$, then there is a map
$\sigma_{j}\mapsto\sigma_{j+1}$, which is a bijection of a set of congruence classes.

\bigskip

\noindent
\textit{Case 2: $v_j\neq v_{j+1}$.}

In this case the $j$th and $(j+1)$th vertices lie on perpendicular lines.
We take these to be the lines $y=n$ and $x=v_j+1$, respectively, so that 
$v_{j+1}=n-1$ and $\Psi^{j+1}(z)$ is given by
\begin{align*}
 \Psi^{j+1}(z) = \lambda\left(\left\lceil\frac{v_j+1}{\lambda}\right\rceil 
    + x_{j+1},\left\lceil\frac{n-1}{\lambda}\right\rceil + y_{j+1}\right).
\end{align*}
(If $x$ is the integer co-ordinate, then the analysis is identical.)
We shall demonstrate the form of $\Le_{j+1}$ by identifying those pairs 
$(a,b)$ for which $\tilde{\sigma}_{j+1}=\sigma_{j+1}$.

Taking the $x$-coordinate of equation (\ref{eq:Psi^jp1tilde}), and recalling the explicit form 
(\ref{eq:sigma_j_2}) of the orbit code, we see that $\tilde{\sigma}_{j+1}$ is determined by
\begin{equation} \label{eq:t,2a+b_eqn}
 \left\lceil\frac{v_j}{\lambda}\right\rceil + x_j + (2a+b)p_j +\epsilon_j +\tilde t(2v_{j+1}+1) 
     = \left\lceil\frac{v_j+1}{\lambda}\right\rceil +\tilde{\sigma}_{j+1} - (2v_{j+1}+1). 
\end{equation}
We think of this as an integer equation of the form $A(2a+b)+B\tilde t=C$, which has solutions 
$2a+b\in\Z$ and $\tilde t\in\N$ for some given value of $\tilde{\sigma}_{j+1}$ if and only if
\begin{displaymath}
 C = \left\lceil\frac{v_j+1}{\lambda}\right\rceil +\tilde{\sigma}_{j+1} 
   - (2v_{j+1}+1) - \left\lceil\frac{v_j}{\lambda}\right\rceil - x_j - \epsilon_j
\end{displaymath}
is sufficiently large and $C\equiv 0 \; ({\rm mod} \; \gcd(A,B))$, i.e., if $\lambda$ is sufficiently small and $\tilde{\sigma}_{j+1}$ satisfies the congruence
\begin{equation}
   \tilde{\sigma}_{j+1} \equiv \left\lceil\frac{v_j}{\lambda}\right\rceil 
 + x_j+\epsilon_j-\left\lceil\frac{v_j+1}{\lambda}\right\rceil \mod{\gcd( p_j, 2v_{j+1}+1)}. \label{eq:sigma_j+1_case2}
\end{equation}
To find the lattice $\Le_{j+1}$, we need to solve this equation in the case $\tilde{\sigma}_{j+1}=\sigma_{j+1}$.

By assumption, the point $z$, given by the module coordinates $a=b=0$, corresponds to the solution 
$2a+b=0$, $\tilde t=t$, for some transit time $t\in\N$. 
Hence the general solution of (\ref{eq:t,2a+b_eqn}) is given by
\begin{align}
 \tilde t&= t -  s\; \frac{p_j}{\gcd( p_j, 2v_{j+1}+1)}, \label{eq:GeneralSolution_t}\\
 2a+b&= s \; \frac{2v_{j+1}+1}{\gcd( p_j, 2v_{j+1}+1)}, \label{eq:GeneralSolution_2a+b}
\end{align}
for $s\in\Z$. The second of these equations implies that $s$ must have the same parity 
as $2a+b$, so we can write $s=2\tilde{a}+b$, where $\tilde{a}\in\Z$ and 
$b\in\{0,\pm 1\}$ for an appropriate choice of sign.
Substituting this expression into equation (\ref{eq:z_tilde}), the points $\tZ$ 
for which $\tilde{\sigma}_{j+1}=\sigma_{j+1}$ are given by
\begin{align*}
 \tZ &= z + \frac{1}{2}\left( s \; \frac{2v_{j+1}+1}{\gcd( p_j, 2v_{j+1}+1)} \mathbf{L}_j -b\mathbf{w}_{v_1,v_1}\right) \\
 &= z  + \frac{1}{2}\left( (2\tilde{a}+b) \mathbf{L}_{j+1} -b\mathbf{w}_{v_1,v_1}\right).
\end{align*}
The last equality is justified by the identities
\begin{align*}
 \frac{2v_{j+1}+1}{\gcd( p_j, 2v_{j+1}+1)} \; \mathbf{L}_j 
&= \frac{(2v_{j+1}+1)(2v_j+1)}{\gcd( (2v_j+1)p_j, (2v_{j+1}+1)(2v_j+1))} 
   \; \frac{\lambda q_j}{2v_1+1} \; (1,1) \\
&= \frac{\lambda q_{j+1}}{2v_1+1} \; (1,1) = \mathbf{L}_{j+1}
\end{align*}
where we have used the relationship $\mbox{\rm lcm}(a,b) = a b/\gcd(a,b)$.
Therefore the first hypothesis (\ref{eq:H1}) remains satisfied.

Substituting the general solution (\ref{eq:GeneralSolution_t}) and (\ref{eq:GeneralSolution_2a+b})  
into equation (\ref{eq:Psi^jp1tilde}) with $\mathbf{x}=\mathbf{e}=(1,0)$, 
and using equation (\ref{eq:Psi^jp1}), we find
\begin{align*}
\Psi^{j+1}(\tZ) &= \Psi^j(z) + \lambda(2a +b) p_j\mathbf{x} + 
      \lambda\epsilon_j \mathbf{x} + \tilde t \mathbf{w}_{v_j,v_{j+1}} \\
&= \Psi^{j+1}(z) + \frac{s p_j}{\gcd( p_j, 2v_{j+1}+1)} \left( \lambda(2v_{j+1}+1)\mathbf{x} - \mathbf{w}_{v_j,v_{j+1}}\right) \\
&= \Psi^{j+1}(z) + \frac{\lambda(2\tilde{a}+b)p_j}{\gcd( p_j, 2v_{j+1}+1)} (2v_j+1)\mathbf{y} \\
&= \Psi^{j+1}(z) + \lambda(2\tilde{a}+b)p_{j+1}\mathbf{y}
\end{align*}
where $\mathbf{y}=(0,1)$.
Thus the points where these $\tZ$ meet the $(j+1)$th vertex share the
same integer coordinate. So hypothesis (\ref{eq:H2}) also remains satisfied, 
completing the induction.

Thus hypotheses (\ref{eq:H1}) and (\ref{eq:H2}) hold for all $1\leq j\leq 2k-1$,
and the equivalence of the three statements follows.
\end{proof}

\begin{corollary} \label{thm:sigma_lattice}
Let $e$ be a critical number, let $k$ be the length of the vertex list of the 
corresponding polygon class, and let $j$ be in the range $1\leq j \leq 2k-1$.
Then two points $z$ and $\tZ$ in $\Xe$ have the same orbit code if and only if
the points $\Psi^j(z)$ and $\Psi^j(\tZ)$ are congruent modulo $\lambda q/(2v_j+1)\,\mathbf{e}$, 
where $\mathbf{e}$ is the unit vector in the direction of the non-integer coordinate 
of the $j$th vertex.
\end{corollary}

\begin{proof}
Recall that for all $1\leq j\leq 2k-1$:
\begin{displaymath}
 \Le \subseteq \Le_j.
\end{displaymath}
Thus any two points which are congruent modulo $\Le$ are also congruent modulo $\Le_j$. 
In particular, if 
\begin{align*}
 \tZ &= z + \frac{1}{2}\left( (2a+b) \mathbf{L} -b\mathbf{w}_{v_1,v_1}\right), \\
 &= z + \frac{1}{2}\left( (2a+b) \frac{q}{q_j} \mathbf{L}_j -b\mathbf{w}_{v_1,v_1}\right),
\end{align*}
then by the hypothesis (\ref{eq:H2}) of lemma \ref{thm:sigma_lattices} we have:
\begin{align*}
\Psi^j(\tZ) &= \Psi^j(z) + \lambda \left((2a+b) \frac{q}{q_j} \right)\, p_j\, \mathbf{e} \\
&= \Psi^j(z) + \lambda(2a+b) \frac{q}{2v_j+1} \, \mathbf{e} 
\end{align*}
as required.
\end{proof}

Lemma \ref{thm:sigma_lattices} shows the equivalence between orbit codes
and congruence classes of $\Le$.
To complete the proof of theorem A, %\ref{thm:Phi_equivariance}, 
we show that the orbit code $\sigma(z)$ 
determines uniquely the behaviour of $z$ under the return map $\Phi$.

\bigskip

%%%%%%%%%%%%%%%%%%%%%%%%%%%%%%%%%%%%%%%%%%%%%%%%%%%%%%%%%%%%%%%%%%%%
%\noindent{\sc Proof of theorem \ref{thm:Phi_equivariance}}\quad
\noindent{\sc Proof of theorem A}\quad

Consider two points $z,z+l\in\Xe$ for some $l\in\Le$ given by
\begin{displaymath}
 l= \frac{1}{2}\left( (2a+b) \mathbf{L} -b\mathbf{w}_{v_1,v_1}\right).
\end{displaymath}
These two points have the same orbit code and reach the $(2k-1)$th vertex 
at the points $\Psi^{2k-1}(z),\Psi^{2k-1}(z+l)\in\Lambda_{v_1,-(v_1+1)}$, which, 
by corollary \ref{thm:sigma_lattice}, are congruent modulo 
$\lambda q/(2v_{2k-1}+1)\mathbf{e}$, where $\mathbf{e}$ is the unit vector in the 
non-integer direction. In particular, $\Psi^{2k-1}(z)$ and $\Psi^{2k-1}(z+l)$ are related via
\begin{align}
\Psi^{2k-1}(z+l) &= \Psi^{2k-1}(z) + \lambda(2a+b)p_{2k-1}\mathbf{e}, \nonumber \\
&= \Psi^{2k-1}(z) + \lambda(2a+b)\frac{q}{(2v_1+1)} \; \mathbf{e}, \label{eq:vertex_2k-1}
\end{align}
where we have replaced $v_{2k-1}$ by $v_1$ using the symmetry (\ref{eq:v_symmetry}) of the vertex types. 
We will show that the points where they reach the last vertex are related by a similar equation:
\begin{equation}
(\Psi^{2k-1}\circ F_{\lambda})(z+l) = (\Psi^{2k-1}\circ F_{\lambda})(z) 
   + \lambda(2a+b)\frac{q}{(2v_1+1)} \; \mathbf{e}^{\perp}, \label{eq:vertex_8k-5}
\end{equation}
where the unit vector $\mathbf{e}^{\perp}$, the non-integer direction of the last vertex, is perpendicular to $\mathbf{e}$.

The last vertex of the return orbit ${\mathcal O}_{\tau}(z)$ lies in the set $\Lambda_{v_1,v_1}$, 
so must be close to the image under $F_{\lambda}$ of the $(2k-1)$th vertex (see figure \ref{fig:Phi_Psi}). 
If the $(2k-1)$th vertex lies on the line $x=v_1+1$, it is a simple exercise to show that these 
two points are in fact equal, i.e., that $(\Psi^{2k-1}\circ F_{\lambda})(z) = (F_{\lambda}\circ \Psi^{2k-1})(z)$ 
for any $z\in\Xe$. We consider the less obvious case in which the $(2k-1)$th 
vertex lies on the line $y=-v_1$ and the non-integer direction is $\mathbf{e}=(1,0)$. 
By the orientation of the vector field in the fourth quadrant, the orbit of 
the point $z$ reaches this vertex at the point $\Psi^{2k-1}(z)$ given by:
\begin{equation}
 \Psi^{2k-1}(z) = \lambda\left(\left\lceil\frac{v_1}{\lambda}\right\rceil + x_{2k-1},
 \left\lceil\frac{-v_1}{\lambda}\right\rceil + y_{2k-1}\right), \label{eq:Psi_2k-1}
\end{equation}
where $x_{2k-1}\geq 0$ and $0\leq y_{2k-1}< 2v_1+1$.

%%%%%%%%%%%%%%%%%%%%%%%%%%%%%%%%%%%%%%%%%%%%%%%%%%%%%%%% FIGURE
\begin{figure}[ht]
        \centering
        \vspace*{-80pt}
        \includegraphics[scale=0.7]{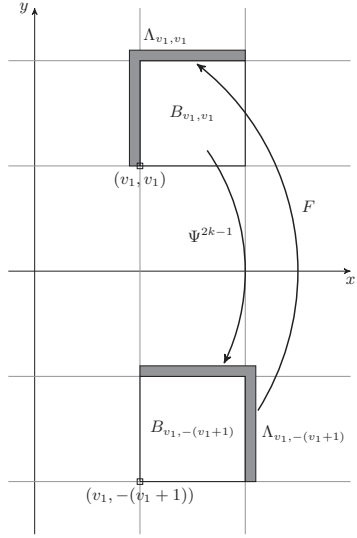}
        \vspace*{-270pt}
        \caption{The $(2k-1)$th and last vertices, joined by the action of $F_{\lambda}$. }
        \label{fig:Phi_Psi}
\end{figure}

Applying $F_{\lambda}$ to (\ref{eq:vertex_2k-1}) and substituting the expression (\ref{eq:Psi_2k-1}), we get:
\begin{align*}
 (F_{\lambda}\circ \Psi^{2k-1})(z+l) &=(F_{\lambda}\circ \Psi^{2k-1})(z) +\lambda(2a+b)\frac{(2v_k+1) \;p}{(2v_1+1)} \; 
 \mathbf{e}^{\perp} \\
&=\lambda\left( v_1-\left\lceil\frac{-v_1}{\lambda}\right\rceil - y_{2k-1}, 
    \left\lceil\frac{v_1}{\lambda}\right\rceil + x_{2k-1} \right) +\lambda(2a+b)\frac{(2v_k+1) \;p}{(2v_1+1)} \; 
 \mathbf{e}^{\perp},
\end{align*}
where the non-integer direction of the last vertex is $\mathbf{e}^{\perp}=(0,1)$.
By equation (\ref{eq:xj<0,yj>0}), if the first component of this point satisfies
$$ 
-(2v_1+1) \leq \left( v_1-\left\lceil\frac{-v_1}{\lambda}\right\rceil 
    - y_{2k-1} \right) - \left\lceil\frac{v_1}{\lambda}\right\rceil < 0,
$$
then the point is the last vertex on $x=v_1$ that we seek for all $(a,b)$:
\begin{equation} \label{eq:Phi_equivariance_case1}
 (\Psi^{2k-1}\circ F_{\lambda})(z+l) = (F_{\lambda}\circ \Psi^{2k-1})(z+l). 
\end{equation}
If the above inequality is not satisfied, then it must be the upper bound that fails; 
since $v_1$ is non-negative, $y_{2k-1}$ satisfies $0\leq y_{2k-1}< 2v_1+1$, and the absolute value
of the two ceiling functions differ by at most one, the lower bound must hold. 
In this case $(F_{\lambda}\circ \Psi^{2k-1})(z+l)\in B_{v_1,v_1}$, 
and we apply $F_{\lambda}^{-4}$ to find:
\begin{align}
 (\Psi^{2k-1}\circ F_{\lambda})(z+l) &= (F_{\lambda}^{-3}\circ \Psi^{2k-1})(z+l) \nonumber \\
&= (F_{\lambda}\circ \Psi^{2k-1})(z+l) - \mathbf{v}((F_{\lambda}^{-3}\circ \Psi^{2k-1})(z+l)) \nonumber \\
&= (F_{\lambda}\circ \Psi^{2k-1})(z+l) -\mathbf{w}_{v_1,v_1}- \lambda\epsilon \, \mathbf{e}^{\perp}, \label{eq:Phi_equivariance_case2}
\end{align}
where the error term $\epsilon$ is independent of $(a,b)$ by proposition \ref{thm:epsilon_j}.
In both cases (\ref{eq:Phi_equivariance_case1}) and (\ref{eq:Phi_equivariance_case2}) the relationship (\ref{eq:vertex_8k-5}) follows.

Using (\ref{eq:vertex_8k-5}), the property (\ref{eq:Psi_congruence1}) of $\Psi$, 
and the expression (\ref{eq:Phi_Psi}) for $\Phi$, we obtain
\begin{align*}
 \Phi(z+l) &\equiv (\Psi^{2k-1}\circ F_{\lambda})(z+l) + \mathbf{v}((\Psi^{2k-1}\circ F_{\lambda})(z+l)) 
                 \mod{\mathbf{w}_{v_1,v_1}} \\
&\equiv (\Psi^{2k-1}\circ F_{\lambda})(z) + \lambda(2a+b)\frac{q}{(2v_1+1)} \; 
        \mathbf{e}^{\perp} + \mathbf{v}((\Psi^{2k-1}\circ F_{\lambda})(z)) \mod{\mathbf{w}_{v_1,v_1}} \\
&\equiv \Phi(z) + \lambda(2a+b)\frac{q}{(2v_1+1)} \; 
             \mathbf{e}^{\perp} \mod{\mathbf{w}_{v_1,v_1}} \\
&\equiv \Phi(z) + \frac{\lambda}{2}(2a+b)\frac{q}{(2v_1+1)} \; 
    \left((\mathbf{e}^{\perp}+\mathbf{e})+(\mathbf{e}^{\perp}-\mathbf{e})\right) \mod{\mathbf{w}_{v_1,v_1}} \\
&\equiv \Phi(z) + \frac{1}{2}(2a+b)\left( \mathbf{L} \pm \frac{q}{(2v_1+1)^2} \; 
         \mathbf{w}_{v_1,v_1} \right) \mod{\mathbf{w}_{v_1,v_1}} \\
&\equiv \Phi(z) + \frac{1}{2}\left((2a+b)\mathbf{L} - b \mathbf{w}_{v_1,v_1}\right) \mod{\mathbf{w}_{v_1,v_1}} \\
&\equiv \Phi(z) + l \mod{\mathbf{w}_{v_1,v_1}},
\end{align*}
where we have also used the fact that $(\mathbf{e}^{\perp}+\mathbf{e})=(1,1)$, 
$(\mathbf{e}^{\perp}-\mathbf{e})=\pm(1,-1)$, and that $q/(2v_1+1)^2$ is odd.
This completes the proof of theorem A. %\ref{thm:Phi_equivariance}.
\endproof

%%%%%%%%%%%%%%%%%%%%%%%%%%%%%%%%%%%%%%%%%%%%%%%%%%%%%%%%%%%%%%%%%%%%%%%%%%%%%

The set $\Theta^e$ of possible orbit codes is a subset of the product space
\begin{displaymath}
 \Theta^e \subseteq \{ 0,1,\dots,2v_1 \} \times \prod_{j=1}^{2k-1} \{ 0,1,\dots,2v_j \}.
\end{displaymath}
Denoting by $\#\Le$ the scaled co-volume of $\Le$, namely
\begin{equation}\label{eq:HashLe}
\#\Le=|(\lambda\Z)^2/\Le|,
\end{equation}
the total number of possible orbit codes is given by 
\begin{equation} \label{eq:theta_e}
|\Theta^e| = \#\Le = -\frac{1}{2\lambda^2} \; \det\left(\mathbf{L}, \mathbf{w}_{v_1,v_1}\right)
 = q. 
\end{equation}
We note that although the lattice $\Le$ is independent of $\lambda$ (up to scaling), $\Theta^e$ is not. 

In the next lemma, we identify the orbit codes which correspond to symmetric 
fixed points of $\Phi$. 
Subsequently, in lemmas \ref{thm:sigma_j_I} and \ref{thm:sigma_j_II}, we identify values of $e$ 
for which the number of codes which satisfy the conditions of lemma 
\ref{thm:minimal_codes} is independent of $\lambda$.
The proof of theorem B will then follow.

\begin{lemma} \label{thm:minimal_codes}
For any $e\in\cE$ with vertex list $(v_1,\dots,v_k)$, $z\in \Xe$ and sufficiently small $\lambda$, 
the point $z$ is a symmetric fixed point of $\Phi$ if and only if its orbit code 
$\sigma(z)=(\sigma_{-1},\sigma_1,\dots,\sigma_{2k-1})$ satisfies:
\begin{center}
 \textit{(i)} \; $\sigma_{-1}=\sigma_1$, \hspace{2cm} \textit{(ii)} \; $2\sigma_k \equiv \lfloor \sqrt{e} \rfloor \mod{2\lfloor \sqrt{e} \rfloor+1}$.
\end{center}
\end{lemma}

\begin{proof}
It is a standard property of any reversible map $F$, with $F=G\circ F^{-1}\circ G$, that it can be written 
as the composition of two involutions:
$$ 
 F= H\circ G \hskip 40pt H= F\circ G \hskip 40pt G^2 = H^2 = \mathrm{Id}. 
$$
It follows that every symmetric periodic orbit of $F$ intersects the union of 
the fixed spaces of these involutions, $\Fix{G}\cup\Fix{H}$, at exactly two points 
(which coincide for period 1) \cite{DeVogelaere}.
Furthermore these two points are maximally separated in time, in the sense that if the minimal period 
of the orbit is $s$, then the transit time from one to the other is approximately $s/2$. 
More precisely, if $s$ is even, the orbit must intersect one of the following sets:
\begin{equation} \label{eq:s_even}
 \Fix{G} \cap F^{s/2}(\Fix{G}) \hskip 20pt \mbox{ or }  \hskip 20pt \Fix{H} \cap F^{s/2}(\Fix{H}),
\end{equation}
whereas if $s$ is odd, the orbit must intersect the set
\begin{equation} \label{eq:s_odd}
 \Fix{G} \cap F^{(s-1)/2}(\Fix{H}).
\end{equation}

For our map $F_{\lambda}$, we have already introduced the involution $G$ and its fixed space, 
considered now as a subset of the rescaled lattice $(\lambda\Z)^2$:
 $$ G(x,y) = (y,x) \hskip 20pt \Fix{G} = \{ \lambda(x,y)\in(\lambda\Z)^2 : x=y \}. $$
A simple calculation shows that the involution $H$ and its fixed space are given by:
\begin{equation} \label{eq:H}
H(x,y) = (\lfloor \lambda y\rfloor - x,  y)\hskip 20pt 
  \Fix{H} = \{ \lambda(x,y)\in(\lambda\Z)^2 : 2x = \lfloor \lambda y\rfloor \}.
\end{equation}

Take $e\in\cE$ and a point $z\in\Xe$. Suppose that $z$ is a symmetric fixed point of $\Phi$. 
If $z$ is non-zero, the orbit of $z$ intersects the set $\Fix{G}\cup\Fix{H}$ at exactly two points, 
and as $z$ is a fixed point of $\Phi$, these two points must occur within a single revolution. 
Hence we have:
 $$ | {\mathcal O}_{\tau}(z) \cap \left(\Fix{G}\cup\Fix{H}\right)|=2. $$

We begin by considering which points in the return orbit ${\mathcal O}_{\tau}(z)$ can lie in $\Fix{G}$. 
Since the domain $\Xe$ lies in an $O(\lambda)$-neighbourhood of the positive half of the 
symmetry line $\Fix{G}$, we may have:
 $$ z=\Phi(z)\in\Fix{G}. $$
Equally, the orbit may intersect the negative half of the symmetry line, which occurs if:
 $$ F^4(z) = G(z) \hskip 20pt \Leftrightarrow \hskip 20pt F^2(z)\in\Fix{G}. $$

Points in $\Fix{H}$ lie on disjoint vertical line segments of length one, in an 
$O(\lambda)$-neighbourhood of the $y$-axis. Recall that the polygon $\Pi(z)$ intersects 
the axes at vertices of type $v_k=\lfloor \sqrt{e} \rfloor$, and hence intersects the $y$-axis in the boxes:
 $$ B_{0,\lfloor \sqrt{e} \rfloor} \hskip 20pt \mbox{and} \hskip 20pt  B_{0,-(\lfloor \sqrt{e} \rfloor+1)}. $$
If $\lfloor \sqrt{e} \rfloor$ is even, it follows that the relevant segment $H^e_+$ of $\Fix{H}$ is given by:
$$ 
 H^e_+ = \{ \lambda(x,y)\in(\lambda\Z)^2 : \; 2x =\lfloor \lambda y \rfloor = \lfloor \sqrt{e} \rfloor \}, 
$$
which lies in the positive half-plane.
Similarly if $\lfloor \sqrt{e} \rfloor$ is odd, the relevant segment $H^e_-$ of $\Fix{H}$ is given by:
$$ 
 H^e_- = \{ \lambda(x,y)\in(\lambda\Z)^2 : \; 2x =\lfloor \lambda y \rfloor = -(\lfloor \sqrt{e} \rfloor+1) \}, 
$$
which lies in the negative half-plane.

Hence we see that the return orbit of $z$ cannot intersect $\Fix{H}$ twice in a single revolution, and 
intersects $\Fix{G}$ twice if and only if:
$$ 
 z \in \Fix{G} \cap F^{-2}(\Fix{G}). 
$$
By (\ref{eq:s_even}), the latter implies that $z$ is periodic with period 4, and we have already observed 
that there are no points with minimal period 4. Thus the only non-trivial possibility for a symmetric fixed 
point occurs when the return orbit of $z$ intersects both $\Fix{G}$ and $\Fix{H}$. 
Conversely, equation (\ref{eq:s_odd}) ensures that this is also a sufficient condition.

The proof now proceeds in two parts.  \\

\noindent
\textit{(i) ${\mathcal O}_{\tau}(z)$ intersects $\Fix{G}$ if and only if $\sigma_{-1}=\sigma_1$.}

If $z=\lambda(x,y)$, then the property $\sigma_{-1}=\sigma_1$ is satisfied if and only if:
\begin{displaymath}
 x \equiv y \mod{2v_1+1}. 
\end{displaymath}
If $y=x$ then clearly $z\in\Fix{G}$. 
The width of the strip $\Xe$, given by (\ref{eq:tildeX^eII}), ensures that the only other possibility is $y=x+(2v_1+1)$, in which case:
\begin{align*}
 F^4(z) &= z + \mathbf{w}_{v_1,v_1} \\
 &= \lambda(x+(2v_1+1),y-(2v_1+1)) \\
 &= \lambda(y,x) = G(z).
\end{align*}
This corresponds to $F^2(z)\in\Fix{G}$. \\

\noindent
\textit{(ii) ${\mathcal O}_{\tau}(z)$ intersects $\Fix{H}$ if and only if 
$2\sigma_k \equiv \lfloor \sqrt{e} \rfloor \mod{2\lfloor \sqrt{e} \rfloor+1}$.}

Instead of considering the sets $H^e_+$ and $H^e_-$ directly, we consider 
their images under $G$ and $F_{\lambda}$, respectively, which lie in a neighbourhood of the $x$-axis:
\begin{align}
 G(H^e_+) &= \{ \lambda(x,y)\in(\lambda\Z)^2 : \; 2y 
   =\lfloor \lambda x \rfloor = \lfloor \sqrt{e} \rfloor \},\label{eq:G(FixH)}\\
 F_{\lambda}(H^e_-) &= \{ \lambda(x,y)\in(\lambda\Z)^2 : \; 2y =\lfloor -\lambda (x+1) \rfloor
   = -(\lfloor \sqrt{e} \rfloor+1) \}. \label{eq:F(FixH)}
\end{align}
In (\ref{eq:F(FixH)}), we assume that $\lambda(\lfloor \sqrt{e} \rfloor+1)/2<1$, 
so that $F_{\lambda}(w) = \lambda(-1-y,x)$ for all $w=\lambda(x,y)\in H^e_-$. 
The orbit ${\mathcal O}_{\tau}(z)$ intersects $\Fix{H}$ if and only if it intersects the relevant 
one of these sets, according to the parity of $\lfloor \sqrt{e} \rfloor$.

The polygon $\Pi(z)$ intersects the $x$-axis at the $k$th vertex, where $k$ is the length of the vertex list $V(e)$. 
The return orbit ${\mathcal O}_{\tau}(z)$ reaches the $k$th vertex at the point $\Psi^k(z)$, given in the notation 
of (\ref{eq:(xj,yj)}) by
\begin{displaymath}
 \Psi^k(z) = \lambda\left(\left\lceil\frac{\lfloor \sqrt{e} \rfloor}{\lambda}\right\rceil + x_k, y_k\right),
\end{displaymath}
where, by (\ref{eq:sigma_j_2}), $y_k=\sigma_k$ is non-negative. 
Hence if $\lfloor \sqrt{e} \rfloor$ is even, ${\mathcal O}_{\tau}(z)$ intersects $\Fix{H}$ if and only if:
 $$ \Psi^k(z) \in G(H^e_+) \hskip 20pt \Leftrightarrow \hskip 20pt \sigma_k = \lfloor \sqrt{e} \rfloor/2. $$
If $\lfloor \sqrt{e} \rfloor$ is odd, then ${\mathcal O}_{\tau}(z)$ intersects $\Fix{H}$ if and only if:
\begin{align*}
 F^4_{\lambda}(\Psi^k(z)) \in F_{\lambda}(H^e_-) \hskip 20pt \Leftrightarrow \hskip 20pt 
\sigma_k &= -(\lfloor \sqrt{e} \rfloor+1)/2 + (2v_k+1) \\
&= (3\lfloor \sqrt{e} \rfloor+1)/2.
\end{align*}
The congruence $2\sigma_k \equiv \lfloor \sqrt{e} \rfloor \; ({\rm mod} \; 2\lfloor \sqrt{e} \rfloor+1)$ 
covers both of these cases, which completes the proof.
\end{proof}

For all $e\in\cE$ and sufficiently small $\lambda$, the set $\Xe$
---see equation (\ref{eq:tildeX^eII})---
is non-empty and contains at least one element from every congruence class 
modulo $\Le$. We now seek to identify the number of congruence classes 
whose orbit code satisfies the conditions of lemma \ref{thm:minimal_codes}.

As discussed in the proof of lemma \ref{thm:minimal_codes}, the points $z=\lambda(x,y)\in\Xe$ whose orbit code 
$\sigma(z) = (\sigma_{-1},\sigma_1,\dots,\sigma_{2k-1})$ satisfies $\sigma_{-1} = \sigma_1$ are precisely 
those satisfying
\begin{displaymath}
 y=x \hskip 20pt \mbox{or} \hskip 20pt y = x+(2v_1+1).
\end{displaymath}
All such points lie on one of two lines, parallel to the first generator $\mathbf{L}$ of the lattice $\Le$. 
Furthermore all points on one line are congruent to those on the other, 
as they are connected by the second generator $(\mathbf{L}-\mathbf{w}_{v_1,v_1})/2$.
Hence the number of points satisfying this condition modulo $\Le$ is
\begin{equation}
 |\mathbf{L}| = \frac{\#\Le}{2v_1+1} = \frac{q}{2v_1+1}, \label{eq:sigma_-1=sigma1}
\end{equation}
where we have used the expression (\ref{eq:theta_e}) for $\#\Le$.

It remains to determine what fraction of the orbit codes with $\sigma_{-1} = \sigma_1$ satisfy 
the second condition of lemma \ref{thm:minimal_codes}. 
We do this by identifying values of $e$ for which all possible values of 
$\sigma_k$ occur with equal frequency, independently of $\lambda$.

Clearly if $k=1$, i.e., if a polygon class has just one vertex in the first octant ($e=0$ -- a square), 
then the points $z=\lambda(x,y)\in\Xe$ with $\sigma_k=\sigma_1 = \sigma^*$ satisfy 
\begin{displaymath}
 x \equiv y \equiv \sigma^* \mod{2v_1+1},
\end{displaymath}
for any given $\sigma^*\in\{0,1,\dots,2v_1\}$. Such points form a fraction 
\begin{displaymath}
 \frac{q}{(2v_1+1)^2}
\end{displaymath}
of all points modulo $\Le$. Hence all possible values of $\sigma_k$ occur with equal frequency modulo $\Le$. 
More generally if $v_k=v_1$, i.e., if all vertices of the polygon class have the same type ($e=0,2,8$), 
then the same applies. This follows from the fact that, for any congruence class of $\Le$, 
the map $\sigma_j \mapsto \sigma_{j+1}$ is a permutation of the set $\{0,1,\dots,2v_j\}$ 
whenever $v_j=v_{j+1}$, as we saw in case 1 of the proof of lemma \ref{thm:sigma_lattices}. 

The following lemma deals with the case that a polygon class has two or more distinct vertex types.

\begin{lemma} \label{thm:sigma_j_I}
Let $e\in\cE$. Suppose that the vertex list $(v_1,v_2,\dots,v_k)$ of the 
associated polygon class has at least two distinct entries and satisfies
\begin{equation}
 \gcd(2v_{\iota(l)}+1,p_{\iota(l-1)}) =1, \label{eq:coprimality_I}
\end{equation}
where $(\iota(i))_{i=1}^l$ is the sequence of distinct vertex types defined in (\ref{eq:v_iota}). 
Then for every $z\in\Xe$, all $1\leq j < \iota(l) $, 
all $\sigma^*\in\{0,1,\dots,2v_k\}$,
and all sufficiently small $\lambda$, the number of points in the set $(z + \Le_j)/ \Le$ 
whose orbit code has $k$th entry $\sigma^*$ is
\begin{equation} \label{eq:Lj_sigma_star}
 \frac{1}{2v_k+1} \, |\Le_j/ \Le|. 
\end{equation}
\end{lemma}

\begin{proof}
Pick $e\in\cE$ such that the coprimality condition (\ref{eq:coprimality_I}) holds. 
Let $z\in\Xe$ have orbit code $\sigma(z) = (\sigma_{-1},\sigma_1,\dots,\sigma_{2k-1})$ 
and let the pair $(x_j,y_j)$ be defined as in equation (\ref{eq:(xj,yj)}), where 
$1\leq j < \iota(l) $.

We have to show that all possible values of $\sigma_k$ occur with equal frequency 
among points in $z + \Le_j$ modulo $\Le$. It suffices to prove that the expression 
(\ref{eq:Lj_sigma_star}) holds for $j=\iota(l)-1$, since all cylinder sets of $\Le_j$ 
with index $j<\iota(l)-1$ can be written as a union of cylinder sets of $\Le_{\iota(l)-1}$.

We let $j= \iota(l)-1$, so that
\begin{displaymath}
 v_{j+1} = v_{\iota(l)} = v_k,
\end{displaymath}
and consider the congruence class of $z$ modulo $\Le_j$. 
Since $z+\Le_j$ is a cylinder set in the sense of lemma \ref{thm:sigma_lattices}, 
the orbit codes of all points $\tZ\in z + \Le_j$ match up to the $j$th entry $\sigma_j$.
Let the $(j+1)$th entry of the orbit code of such a $\tZ$ be $\tilde{\sigma}_{j+1}$. 
We wish to show that all possible values of $\tilde{\sigma}_{j+1}$ occur with equal frequency.

By construction $v_j\neq v_{j+1}$, so the possible values of $\tilde{\sigma}_{j+1}$ are determined by 
case 2 of the proof of lemma \ref{thm:sigma_lattices}. 
In the course of the proof, we saw that the occurrence of points $\tZ$ with some fixed value of  
$\tilde{\sigma}_{j+1}$ correspond to solutions of an integer equation, given in the case where $y$ is the 
non-integer coordinate of the $j$th vertex by equation (\ref{eq:t,2a+b_eqn}). 
(A similar equation holds when $x$ is the non-integer coordinate.) 
Each solution $(2a+b,\tilde{t})\in \Z\times\N$ determines the module coordinates $(a,b)$ of 
$\tZ-z$ in $\Le_j$ and the transit time $\tilde{t}$ of $\tZ$ from the $j$th vertex to the $(j+1)$th.

Solutions of (\ref{eq:t,2a+b_eqn}) occur for all values of $\tilde{\sigma}_{j+1}$ 
satisfying the congruence (\ref{eq:sigma_j+1_case2}), and the condition that $\lambda$ 
be sufficiently small ensures that 
at least one such solution is realised by a point $\tZ\in \Xe$. 
By construction, each distinct value of $\tilde{\sigma}_{j+1}$ 
which has a solution defines a unique point in $z + \Le_j$ modulo $\Le_{j+1}$, which is 
isomorphic to the module $\Le_j/\Le_{j+1}$. 
However due to the coprimality condition (\ref{eq:coprimality_I}), the modulus of the congruence 
(\ref{eq:sigma_j+1_case2}) is unity. Hence solutions occur for all possible values of $
\tilde{\sigma}_{j+1}$, and each corresponds to a unique congruence class of 
$z + \Le_j$ modulo $\Le_{j+1}$. Furthermore, by (\ref{eq:q_j=q}), the lattices 
$\Le_{j+1}$ and $\Le$ are equal, hence all possible values of $\tilde{\sigma}_{j+1}$ 
occur with equal frequency in $z + \Le_j$ modulo $\Le$.

If $j+1=\iota(l)=k$ then this completes the proof. If $\iota(l)<k$, take $i$ in the range 
$\iota(l) \leq i <k$. By the definition of $\iota(l)$ as the index of the last distinct vertex type, 
we have $v_{i}=v_{i+1}=v_k$. As discussed above, the map $\sigma_i \mapsto \sigma_{i+1}$ 
is a permutation of the set $\{0,1,\dots,2v_i\}$ whenever $v_i = v_{i+1}$. 
Hence the equal frequency of the possible values of $\tilde{\sigma}_{i}$ implies that of $\tilde{\sigma}_{i+1}$ 
and the result follows.
\end{proof}

%%%%%%%%%%%%%%%%%%%%%%%%%%%%%%%%%%%%%%%%%%%%%%%%%%%%%%%%%%%%%%%%

In the previous section (equation (\ref{eq:sigma_j})), we defined 
the $j$th entry $\sigma_j$ of the orbit code via the congruence
$\sigma_j \equiv y_j \mod{2v_j+1}$, where $y$ is the integer coordinate 
of the relevant vertex, and the pair $(x_j,y_j)$ is defined by equation (\ref{eq:(xj,yj)}).
Similarly, we define the sequence $\gamma(z)$ such that its $j$th entry $\gamma_j$ satisfies:
\begin{equation}
 \gamma_j \equiv x_j \mod{\frac{q}{2v_j+1}}, \label{eq:gamma_j}
\end{equation}
where $x$ is the non-integer coordinate of the vertex. 
It follows from corollary \ref{thm:sigma_lattice} that, for any $1\leq j \leq 2k-1$, and any two points 
$z,\tZ\in\Xe$ :
\begin{displaymath} \label{eq:(sigma_j,gamma_j)}
 \tZ \equiv z \mod{\Le} \hskip 20pt \Leftrightarrow  \hskip 20pt (\sigma_j,\gamma_j) = (\tilde{\sigma}_j,\tilde{\gamma}_j).
\end{displaymath}

In the following lemma, we use $\gamma(z)$ to identify polygon classes where, 
among points with $\sigma_{-1}=\sigma_1$ and for each $j$ in $1\leq j \leq 2k-1$, 
all possible values of $\sigma_j\in\{0,1,\dots,2v_j\}$ occur with equal 
frequency modulo $\Le$, independently of $\lambda$.

\begin{lemma} \label{thm:sigma_j_II}
Let $e\in\cE$ and suppose that the vertex list $(v_1,v_2,\dots,v_k)$ of the 
associated polygon class is such that $2v_1+1$ is coprime to $2v_j+1$ for all other vertex types $v_j$:
\begin{equation}
 \gcd(2v_1+1,2v_j+1) =1 \hskip 40pt 2\leq j \leq k, \; v_j\neq v_1. \label{eq:v1_coprimality}
\end{equation}
Then for sufficiently small $\lambda$, for all $j$ in $1\leq j \leq 2k-1$,
and all $\sigma^*\in\{0,1,\dots,2v_j\}$, the number $n_j$ of points $z\in \Xe$ 
modulo $\Le$ whose orbit code $(\sigma_{-1},\sigma_1,\dots,\sigma_{2k-1})$ 
has $\sigma_{-1} = \sigma_1$ and $\sigma_j = \sigma^*$ is given by:
\begin{equation}\label{eq:n_j}
n_j=\frac{\#\Le}{(2v_1+1)(2v_j+1)}.
\end{equation}
\end{lemma}
\begin{proof}
We use induction on $j$. Consider points $z$ whose orbit code satisfies 
$\sigma_{-1}=\sigma_1$ and has $j$th value $\sigma_j$, for some arbitrary 
$\sigma_j\in\{0,1,\dots,2v_j\}$ and $j\in\{1,\dots,2k-1\}$. 
Let the sequence $\gamma(z)$ be denoted $(\gamma_{-1},\gamma_1,\dots,\gamma_{2k-1})$. 
Our induction hypotheses are that: \\
(i) equation (\ref{eq:n_j}) holds, where the coprimality condition (\ref{eq:v1_coprimality}) 
ensures that $n_j$ is a natural number; \\
(ii) for each residue $r \in \{0,1,\dots, n_j-1\}$ modulo $n_j$, there is a unique $z$ modulo $\Le$ satisfying
\begin{displaymath}
 \gamma_j \equiv r \mod{n_j}.
\end{displaymath}

The base case is $j=1$. The points $z=\lambda(x,y)\in\Xe$ with $\sigma_{-1}=\sigma_1$ for some fixed value of $\sigma_1$ satisfy:
\begin{displaymath}
 x \equiv y \equiv \sigma_1 \mod{2v_1+1}.
\end{displaymath}
Such points are congruent modulo $(\lambda(2v_1+1)\Z)^2$, hence the number of such points modulo $\Le$ is
\begin{displaymath}
 \frac{\#\Le}{(2v_1+1)^2} = \frac{q}{(2v_1+1)^2} = n_1.
\end{displaymath}
By lemma \ref{thm:sigma_lattices}, if $z$ is one such point, then any other point $\tZ$ reaches the first vertex at
\begin{displaymath}
 \Psi(\tZ) = \Psi(z) + \lambda s p_1 \mathbf{e}
\end{displaymath}
for some $s\in\Z$, where $p_1=2v_1+1$ and $\mathbf{e}$ is the unit vector in the non-integer coordinate direction. 
Then by the construction of $\gamma(z)$, if $\gamma(\tZ)=(\tilde{\gamma}_{-1},\tilde{\gamma}_1,\dots,\tilde{\gamma}_{2k-1})$, 
the value of $\tilde{\gamma}_1$ is related to $\gamma_1$ by
\begin{displaymath}
\tilde{\gamma}_1 \equiv \gamma_1 + s (2v_1+1) \mod{(2v_1+1) \, n_1}.
\end{displaymath}
By corollary \ref{thm:sigma_lattice}, $z$ and $\tZ$ are congruent modulo $\Le$ if and only if $\tilde{\gamma}_1 = \gamma_1$. 
Thus the $n_1$ distinct points modulo $\Le$ correspond to distinct values of $s$ modulo $n_1$. Furthermore, if we consider the value of $\tilde{\gamma}_1$ modulo $n_1$, we have:
\begin{displaymath}
 \tilde{\gamma}_1 \equiv \gamma_1 + s (2v_1+1) \mod{n_1}.
\end{displaymath}
Now $2v_1+1$ is coprime to the modulus, as by the coprimality condition (\ref{eq:v1_coprimality}) and the construction (\ref{eq:q}) of $q$, two is the highest power of $(2v_1+1)$ that divides $q$. It follows that each distinct value of $\tilde{\gamma}_1$ is distinct modulo $n_1$. This completes the base case.

To proceed with the inductive step, we suppose that the above hypotheses hold for some $j\in\{1,\dots,k-1\}$. 
In the proof of lemma \ref{thm:sigma_lattices} we used equation (\ref{eq:Psi^jp1}) to
describe the behaviour of points as they move from one vertex to the next in two cases. 
The first case occurs when $v_j = v_{j+1}$, so that the $j$th and $(j+1)$th vertices lie on parallel lines, and $n_j = n_{j+1}$. 
In this case, the value of $\sigma_{j+1}$ is determined uniquely by the value of $\sigma_j$. 
In particular, we saw that if the $j$th vertex lies on $y=n$ and the $(j+1)$th vertex lies on $y=n-1$, 
then $\sigma_{j+1}$ and $\sigma_j$ are related by equation (\ref{eq:sigma_j+1_case1}).

We can use the same methods, considering this time the non-integer component of equation 
(\ref{eq:Psi^jp1}), to show that $\gamma_{j+1}$ is determined by the 
pair $(\sigma_j,\gamma_j)$ via:
\begin{displaymath}
 \gamma_{j+1} \equiv \gamma_{j} +\epsilon_j +(2n-1)t \mod{(2v_1+1) \, n_j},
\end{displaymath}
where $\epsilon_j=\epsilon_j(\sigma_j)$, and $t = t(\sigma_j)$ is the transit time between vertices. 
The one-to-one relationship between $\sigma_j$ and $\sigma_{j+1}$, ensures that there are $n_{j+1}=n_j$ 
points with $\sigma_{-1}=\sigma_1$ that achieve any given value of $\sigma_{j+1}$ at the $(j+1)$th vertex. 
Then for any given value of $\sigma_j$, which uniquely determines $\sigma_{j+1}$, the above congruence
establishes a one-to-one relationship between $\gamma_j$ and $\gamma_{j+1}$ modulo $(2v_1+1)n_j$.
Because this bijection is a translation, it also holds modulo $n_j$.
In other words, there is a skew-product map of residue classes modulo $n_j$: 
$(\sigma_j,\gamma_j)\mapsto(\sigma_{j+1},\gamma_{j+1})$.
This completes the inductive step for the first case.

In the second case, where $v_j \neq v_{j+1}$, the $j$th and $(j+1)$th vertices lie on perpendicular lines. 
Again referring to the proof of lemma \ref{thm:sigma_lattices}, taking equation (\ref{eq:t,2a+b_eqn}) 
modulo $2v_{j+1}+1$ gives the following expression for $\sigma_{j+1}$ in terms of the pair $(\sigma_j,\gamma_j)$:
\begin{displaymath}
\left\lceil\frac{v_j+1}{\lambda}\right\rceil +\sigma_{j+1} 
 \equiv \left\lceil\frac{v_j}{\lambda}\right\rceil + \gamma_j +\epsilon_j \mod{2v_{j+1}+1}.
\end{displaymath}
Here we were able to replace $x_j$ with $\gamma_j$ as, by the construction (\ref{eq:q}) of $q$, 
$2v_{j+1}+1$ is a divisor of the modulus $q/(2v_j+1)=(2v_1+1)n_j$ which defines $\gamma_j$. 
If the coprimality condition (\ref{eq:v1_coprimality}) holds, then $2v_{j+1}+1$ also divides $n_j$. 
Hence for any given pair $(\sigma_j,\sigma_{j+1})$, there are $n_j/(2v_{j+1}+1)$ values of 
$\gamma_j$ modulo $n_j$ for which the following congruence is satisfied:
\begin{equation}\label{eq:gamma_j_II}
 \gamma_j \equiv \left\lceil\frac{v_j+1}{\lambda}\right\rceil +\sigma_{j+1} 
 - \left\lceil\frac{v_j}{\lambda}\right\rceil -\epsilon_j + s(2v_{j+1}+1) \mod{n_j}
\end{equation}
where $s\in\Z$. The total number of points with any given value of $\sigma_{j+1}$ is thus:
\begin{displaymath}
 (2v_j+1) \times \frac{n_j}{2v_{j+1}+1} = n_{j+1},
\end{displaymath}
which completes the inductive step for hypothesis (i).

Taking the second component of equation (\ref{eq:Psi^jp1}) modulo $n_{j+1}$ gives an expression for 
$\gamma_{j+1}$ in terms of the pair $(\sigma_j,\gamma_j)$:
\begin{equation} \label{eq:gamma_jp1}
 \gamma_{j+1} 
     \equiv \left\lceil\frac{n}{\lambda}\right\rceil + \sigma_j + (2v_j+1)t - \left\lceil\frac{n-1}{\lambda}\right\rceil \mod{n_{j+1}},
\end{equation}
where $t=t(\sigma_j,\gamma_j)$. For a given pair $(\sigma_j,\sigma_{j+1})$, 
$t$ is given by equation (\ref{eq:t,2a+b_eqn}). 
Hence taking equation (\ref{eq:t,2a+b_eqn}) modulo $n_j$, a multiple of $2v_{j+1}+1$, 
and using the expression (\ref{eq:gamma_j_II}) for $\gamma_j$, 
it follows that the values of $t$ satisfy
\begin{align*}
t+1 &\equiv \frac{ \left\lceil (v_j+1)/\lambda \right\rceil +\sigma_{j+1} 
  - \left\lceil v_j/\lambda \right\rceil - \gamma_j -\epsilon_j}{2v_{j+1}+1} \mod{\frac{n_j}{2v_{j+1}+1}} \\
 &\equiv -s \mod{n_j/(2v_{j+1}+1)},
\end{align*}
where $s\in\Z$. Thus $t$ takes all values modulo $n_j/(2v_{j+1}+1) = n_{j+1}/(2v_j+1)$. 
In turn, equation (\ref{eq:gamma_jp1}) 
implies that $\gamma_{j+1}$ takes all values satisfying:
\begin{displaymath}
 \gamma_{j+1} \equiv \left\lceil\frac{n}{\lambda}\right\rceil + \sigma_j -\left\lceil\frac{n-1}{\lambda}\right\rceil \mod{2v_j+1}.
\end{displaymath}
Applying this argument for all values $\sigma_j\in\{0,1,\dots,2v_j\}$, we get a complete  
residue class modulo $n_{j+1}$ as required. 
This completes the inductive step for hypothesis (ii).
\end{proof}

\noindent{\sc Proof of theorem \ref{thm:minimal_densities}}\quad

Let $e\in\cE$ be given, let $(v_1,\ldots,v_k)$ be the corresponding vertex list,
and let $\sigma^*$ be the unique element of the set $\{0,1,\dots,2v_k\}$ that satisfies
\begin{displaymath}
2\sigma^* \equiv \lfloor\sqrt{e}\rfloor \mod{2\lfloor\sqrt{e}\rfloor+1}.
\end{displaymath}

For $e=0,2,8$ all elements of the vertex list are the same. This case is dealt
with by the discussion preceding lemma \ref{thm:sigma_j_I}.
Thus we assume that the vertex list contains at least two distinct elements.

Suppose first that $2v_1+1$ is coprime to $2v_j+1$ for all $v_j\neq v_1$. 
Then for sufficiently small $\lambda$, lemma \ref{thm:sigma_j_II} states that the number 
$n_k$ of points in $\Xe\cap\Fix{G}$ modulo $\Le$ whose orbit code has $k$th entry $\sigma^*$ is
given by equation (\ref{eq:n_j}), with $j=k$.
Furthermore, since all points $z\in\Xe$ whose orbit code $(\sigma_{-1},\sigma_1,\dots,\sigma_{2k-1})$ 
satisfies $\sigma_{-1}=\sigma_1$ are congruent to some point in $\Xe\cap\Fix{G}$, it follows 
that $n_k$ is the number of points in $\Xe$ modulo $\Le$ satisfying the conditions $(i)$ and $(ii)$ of 
lemma \ref{thm:minimal_codes}. Therefore the number of symmetric fixed points of $\Phi$ in $\Xe$ 
modulo $\Le$ is independent of $\lambda$ and given by $n_k$.

Similarly if $2v_k+1$ is coprime to $2v_j+1$ for all $v_j\neq v_k$, then $2v_k+1$ is coprime 
to $q_{\iota(l-2)}$, given in closed form by (\ref{eq:q_j_closed_form}). If follows that the 
condition (\ref{eq:coprimality_I}) of lemma \ref{thm:sigma_j_I} holds, since the recursive expression (\ref{eq:q_j}) for $q_{\iota(l-1)}$ gives us that:
\begin{align*}
p_{\iota(l-1)} &= \frac{q_{\iota(l-1)} }{2v_{\iota(l-1)} +1} \\
&= \frac{\mbox{\rm lcm} ((2v_{\iota(l-1)} +1)(2v_{\iota(l-2)} +1), q_{\iota(l-2)} ) }{2v_{\iota(l-1)} +1}  \\
&= \frac{(2v_{\iota(l-2)} +1) \, q_{\iota(l-2)}  }{\gcd ((2v_{\iota(l-1)} +1)(2v_{\iota(l-2)} +1), q_{\iota(l-2)} )} .
\end{align*}
Applying lemma \ref{thm:sigma_j_I} for $j=1$, we have that for every cylinder set of $\Le_1$, 
the number of points modulo $\Le$ in the cylinder set whose orbit code has $k$th entry $\sigma^*$ is given by
\begin{displaymath}
 \frac{1}{2v_k+1} |\Le_1/\Le| = \frac{1}{2v_k+1}\frac{\#\Le}{(2v_1+1)^2} = \frac{n_k}{2v_1+1}.
\end{displaymath}
There are $2v_1+1$ cylinder sets of $\Le_1$ whose associated orbit code satisfies $\sigma_{-1}=\sigma_1$. 
Hence, as before, the number of points in $\Xe$ modulo $\Le$ satisfying the conditions $(i)$ and $(ii)$ 
of lemma \ref{thm:minimal_codes} is
\begin{displaymath}
(2v_1+1)\times \frac{n_k}{2v_1+1} = n_k,
\end{displaymath}
and the number of symmetric fixed points of $\Phi$ in  $\Xe$ modulo $\Le$ follows. 
This completes the proof of the first statement.

We have shown that for sufficiently small $\lambda$, and if (\ref{eq:v1_k_coprimality}) 
holds, then the fraction of symmetric fixed points of $\Phi$ in each fundamental domain of $\Le$ is
\begin{displaymath}
\frac{1}{(2v_1+1)(2v_k+1)} = \frac{1}{(2\lfloor\sqrt{e/2}\rfloor+1)(2\lfloor\sqrt{e}\rfloor+1)},
\end{displaymath}
where we have used equations (\ref{eq:v1}) and (\ref{eq:vk}), respectively. 
It remains to show that the density $\delta(e,\lambda)$ of symmetric fixed points in $\Xe$ converges 
to this fraction as $\lambda\rightarrow 0$.

By equation (\ref{eq:tildeX^eII}), the domain $\Xe$ is a subset of the lattice $(\lambda\Z)^2$, bounded 
by a rectangle lying parallel to the symmetry line $\Fix{G}$.
Similarly, a fundamental domain of the lattice $\Le$ is constructed by taking the points bounded 
by the parallelogram $\Omega$, given by
\begin{displaymath}
\Omega = \{ \alpha \mathbf{L} + \frac{\beta}{2}(\mathbf{L}-\mathbf{w}_{v_1,v_1}) : \; \alpha,\beta\in[0,1) \},
\end{displaymath}
where the generator $\mathbf{L}$ is also parallel to the symmetry line. 
These parallelograms tile the plane by translation: $\Omega + \Le = \R^2$.

The width of $\Xe$ (taken in the direction perpendicular to $\Fix{G}$) is exactly twice 
that of $\Omega$, independently of $\lambda$, as shown in Figure \ref{fig:lattice_Le}. 
The number of parallelograms which fit lengthwise into $\Xe$, however, goes to infinity as $\lambda$ goes to zero. 
In particular, the number of parallelograms which can be contained in the interior of the rectangle 
bounding $\Xe$ is given by
\begin{displaymath}
2 \left\lfloor \frac{|\tilde{I}^e|}{\lambda |\mathbf{L}|} \right\rfloor - 8,
\end{displaymath}
where $\lfloor |\tilde{I}^e|/\lambda |\mathbf{L}| \rfloor$ is the number of times that the vector 
$\mathbf{L}$ fits lengthways into the rectangle, and we subtract $8$ for the parallelograms which intersect 
the boundary. Each of these parallelograms contains a complete fundamental domain of $\Le$, and their 
contribution to $\delta(e,\lambda)$ dominates in the limit $\lambda\rightarrow 0$.

Explicitly, we have:
\begin{align*}
\delta(e,\lambda) &= \frac{\#\Le}{\# \Xe}
\left(  \frac{2 \left\lfloor |\tilde{I}^e|/\lambda |\mathbf{L}| \right\rfloor - 8}{(2v_1+1)(2v_k+1)} + O(1) \right) \\
&= \frac{q}{2(2v_1+1)\left\lfloor |\tilde{I}^e|/\lambda \right\rfloor}
\left(  \frac{2\left\lfloor (2v_1+1) |\tilde{I}^e|/\lambda q \right\rfloor - 8}{(2v_1+1)(2v_k+1)} + O(1) \right) \\
&= \frac{1}{(2v_1+1)(2v_k+1)} + O(\lambda),
\end{align*}
as $\lambda\rightarrow 0$.
\endproof

%--------------------------------------------------------------------------------------
%\bibliography{Heather}

\end{document}